\documentclass[10pt]{article}
\def\version{
May 21, 2016
}

\nonstopmode

\usepackage[dvips]{graphics}

\usepackage{ulem} 

\providecommand{\emph}[1]{{\it #1}}
\renewcommand{\emph}[1]{{\it #1}}

\usepackage{amsthm}

\usepackage{latexsym,epsfig,bm}

\usepackage{bm}
\usepackage{upgreek}

\usepackage{mathrsfs}
\usepackage{times}
\usepackage{amssymb}

\usepackage[usenames,dvipsnames,dvips]{color}
\definecolor{MyDarkBlue}{rgb}{0,0.08,0.45}
\usepackage[backref=none]{hyperref}
\hypersetup{pdfborder={0 0 0},
  colorlinks,
  urlcolor={MyDarkBlue},
  linkcolor={MyDarkBlue},
  citecolor={MyDarkBlue},
  breaklinks=true}

\usepackage{doi}
\providecommand{\eprint}[1]{}
\renewcommand{\eprint}[1]{arXiv:\href{http://arxiv.org/abs/#1}{#1}}

\usepackage{amsmath}

\providecommand{\eqref}[1]{{\rm (\ref{#1})}}
\providecommand{\itref}[1]{{\it (\ref{#1})}}

\DeclareSymbolFont{AMSb}{U}{msb}{m}{n}
\DeclareSymbolFontAlphabet{\mathbb}{AMSb}

\DeclareSymbolFont{EUB}{U}{eur}{b}{n}
\SetSymbolFont{EUB}{bold}{U}{eur}{b}{n}
\DeclareSymbolFontAlphabet{\eub}{EUB}

\newcommand{\jj}{\mathrm{i}}
\newcommand{\End}{\,{\rm End}\,}

\newcommand{\scrD}{\mathscr{D}}

\newcommand{\scrX}{\mathscr{X}}
\newcommand{\calN}{\mathcal{N}}

\newcommand{\eubD}{\eub{D}}
\newcommand{\eubJ}{\eub{J}}

\newcommand{\eubL}{\eub{L}}

\newcommand{\bmupalpha}{\bm\upalpha}
\newcommand{\bmupbeta}{\bm\upbeta}

\newcommand{\bmupphi}{\bm\upphi}
\newcommand{\bmuprho}{\bm\uprho}

\newcommand{\notyet}[1]{{}}

\newcommand{\tr}{\mathop{\rm Tr}}

\newcommand{\supp}{\mathop{\rm supp}}

\newcommand{\p}{\partial}
\newcommand{\at}[1]{\vert\sb{\sb{#1}}}

\def\R{\mathbb{R}}
\newcommand{\C}{\mathbb{C}}

\newcommand{\N}{\mathbb{N}}
\newcommand{\Abs}[1]{\left\vert#1\right\vert}
\newcommand{\abs}[1]{\vert #1 \vert}
\newcommand{\Norm}[1]{\Big\Vert #1 \Big\Vert}

\newcommand{\norm}[1]{\Vert #1 \Vert}

\newcommand{\sothat}{\,\,{\rm ;}\ \,}

\DeclareMathSymbol{\varGamma}{\mathord}{letters}{"00}
\DeclareMathSymbol{\varDelta}{\mathord}{letters}{"01}
\DeclareMathSymbol{\varTheta}{\mathord}{letters}{"02}
\DeclareMathSymbol{\varLambda}{\mathord}{letters}{"03}
\DeclareMathSymbol{\varXi}{\mathord}{letters}{"04}
\DeclareMathSymbol{\varPi}{\mathord}{letters}{"05}
\DeclareMathSymbol{\varSigma}{\mathord}{letters}{"06}
\DeclareMathSymbol{\varUpsilon}{\mathord}{letters}{"07}
\DeclareMathSymbol{\varPhi}{\mathord}{letters}{"08}
\DeclareMathSymbol{\varPsi}{\mathord}{letters}{"09}
\DeclareMathSymbol{\varOmega}{\mathord}{letters}{"0A}

\theoremstyle{plain}
\newtheorem{lemma}{Lemma}[section]
\newtheorem{theorem}[lemma]{Theorem}

\newtheorem{proposition}[lemma]{Proposition}

\theoremstyle{definition}
\newtheorem{definition}[lemma]{Definition}

\theoremstyle{remark}
\newtheorem{remark}[lemma]{Remark}
\newcounter{step}
\newenvironment{proofstep}
 {\begin{list}{\bf Step \arabic{step}.~~}{\usecounter{step} \labelsep=0em
\labelwidth=0em \leftmargin=0em \itemindent=0em}}
 {\end{list}}

\makeatletter\@addtoreset{equation}{section}
\makeatletter\@addtoreset{lemma}{section}
\makeatother

\def\dist{\mathop{\rm dist}\nolimits}
\renewcommand{\Re}{\mathop{\rm{R\hskip -1pt e}}\nolimits}
\renewcommand{\Im}{\mathop{\rm{I\hskip -1pt m}}\nolimits}

\begin{document}

\title{On spectral stability of the nonlinear Dirac equation}

\author{
{\sc Nabile Boussa{\"\i}d}
\\
{\it\small Universit\'e de Franche-Comt\'e, 25030 Besan\c{c}on CEDEX, France}
\\ \\
{\sc Andrew Comech}
\\
{\it\small Texas A\&M University, College Station, Texas 77843, U.S.A.}
{\small and}
{\it\small IITP, Moscow 127051, Russia}
}

\date{\version}

\maketitle

\begin{abstract}
We study the point spectrum of the nonlinear Dirac equation
in any spatial dimension,
linearized at one of the solitary wave solutions.
We prove that, in any dimension, the linearized equation has no embedded eigenvalues
in the part of the essential spectrum beyond the embedded thresholds.
We then prove that the birth of point eigenvalues with nonzero real part
(the ones which lead to linear instability)
from the essential spectrum is only possible from the embedded eigenvalues or thresholds,
and therefore can not take place beyond the embedded thresholds.
We also prove that ``in the nonrelativistic limit''
$\omega\to m$,
the point eigenvalues can only accumulate to $0$ and $\pm 2m\jj$.
\end{abstract}


\section{Introduction}

In the present work, we study the discrete spectrum
of linearization of nonlinear Dirac models.
The analysis of the discrete spectrum is crucial
for the question
of the dynamical stability of solitary wave solutions.
While this question is well-understood in many cases
for both the nonlinear Schr\"odinger
and Klein--Gordon equations (see e.g. the review
\cite{MR1032250}),
there are numerous open questions
for systems with Hamiltonians which are
sign-indefinite,
such as the nonlinear Dirac equation or the Dirac--Maxwell system.

The idea to consider self-interacting spinor field
has been studied in Physics
for a long time,
starting with the pioneering work of Ivanenko~\cite{jetp.8.260}
and then followed up in
\cite{PhysRev.83.326,PhysRev.103.1571,RevModPhys.29.269}.
Widely known are
the massive Thirring model~\cite{MR0091788}
(spinor field with the vector self-interaction)
and the Soler model~\cite{PhysRevD.1.2766}
(spinor field with the scalar self-interaction).
The one-dimensional analogue of the Soler model
is known as the (massive) Gross--Neveu model
\cite{1974PhRvD..10.3235G,PhysRevD.12.3880}.

In the past two decades there has been an increasing interest
in the nonlinear Dirac equation.
The bibliography is now so extensive
that we do not hope to cover it comprehensively,
only giving a very brief account.
The existence of standing waves in the nonlinear Dirac equation was
studied in
\cite{PhysRevD.1.2766,MR847126,MR949625,MR1344729}.
Local and global well-posedness
of the nonlinear Dirac equation
was further addressed in
\cite{MR1434039}
(semilinear Dirac equation in (3+1)D)
and in~\cite{MR2108356}
(nonlinear Dirac equation in (3+1)D).
There are many results
on the local and global well-posedness in (1+1)D;
we mention
\cite{MR2588476,MR2666663,MR2829499,MR2883845,MR3063953}.

The stability of solitary wave solutions
of the nonlinear Dirac equation
was approached via numerical simulations
\cite{PhysRevD.10.517,PhysRevLett.50.1230,
PhysRevD.34.644,MR2892774,PhysRevE.86.046602,MR3066202}
and via heuristic arguments
\cite{MR592382,MR832537,MR848095,PhysRevD.36.2422,PhysRevE.82.036604},
based on the analysis of whether the energy functional
is minimized under the charge constraint
with respect to dilations and other families of perturbations.
The spectrum of the linearization at solitary waves
of the nonlinear Dirac equation in (1+1)D
was computed in~\cite{chugunova-thesis,MR2892774},
suggesting
the absence of eigenvalues with positive real part
for linearizations at
small amplitude solitary waves;
we will say that such solitary waves
are \emph{spectrally stable}.
The numerical simulations of the evolution
of perturbed solitary waves
\cite{PhysRevE.86.046602}
suggest that the small amplitude solitary waves
in (1+1)D nonlinear Dirac equation
are also \emph{dynamically} stable
(or \emph{nonlinearly stable}).
For the massive
Thirring model, which is completely integrable,
the orbital stability was proved by means
of a coercive conservation law in \cite{MR3147657,MR3462129}.
The asymptotic stability of small amplitude solitary waves
in the external potential
has been studied in~\cite{MR2259214,MR2466169,MR2985264}.

Given a real-valued function $f\in C(\R)$, $f(0)=0$,
we consider
the following nonlinear Dirac equation
in $\R\sp n$, $n\ge 1$,
which is known as the Soler model~\cite{PhysRevD.1.2766}:
\begin{equation}\label{nld-0}
 \jj \p\sb t\psi=D\sb m\psi-f(\psi\sp\ast\beta\psi)\beta\psi,
\qquad
\psi(x,t)\in\C^N,
\quad
x\in\R\sp n.
\end{equation}
Above,
$D\sb m=-\jj \bm\alpha\cdot\bm\nabla+\beta m$
is the free Dirac operator.
Here $\bm\alpha=(\alpha\sp\jmath)\sb{1\le \jmath \le n}$,
with
$\alpha\sp\jmath$ and $\beta$
the self-adjoint $N\times N$ Dirac matrices
(see Section~\ref{sec:notations} for the details);
$m>0$ is the mass.
We are interested
in the stability properties
of the solitary wave solutions to \eqref{nld-0}:
\begin{equation}\label{sw}
\psi(x,t)=\phi\sb\omega(x)e\sp{-\jj \omega t},
\end{equation}
where the amplitude $\phi\sb\omega$ satisfies
the stationary equation
\[
 \omega \phi\sb\omega=D\sb m\phi\sb\omega-f(\phi\sb\omega\sp\ast\beta\phi\sb\omega)\beta\phi\sb\omega.
\]
In the present work, we study the \emph{spectral stability}
of solitary waves in the nonlinear Dirac equation.
Given a particular solitary wave \eqref{sw},
we consider its perturbation,
$\big(\phi\sb\omega(x)+\rho(x,t)\big)e\sp{-\jj \omega t}$,
and study the spectrum of the linearized equation on $\rho$.

\begin{definition}\label{def:linearinstability}
We will say that a particular solitary wave is
\emph{spectrally stable}
if the spectrum of the equation linearized at this wave
does not contain points with positive real part.
\end{definition}

\begin{remark}
Sometimes it is convenient
to include a further requirement
that the operator corresponding
to the linearization at this wave
does not contain
$4\times 4$ Jordan blocks at $\lambda=0$
and no $2\times 2$ Jordan blocks at
$\lambda\in\jj\R\setminus \{0\}$;
such blocks are expected to lead to dynamic instability.
Related definitions of linear instability are in
\cite{MR3180733,MR2924465}.
\end{remark}

In the context of spectral stability, the well-posedness of
the initial value problem associated with~\eqref{nld-0} is not crucial.
However in the investigation of the dynamical or Lyapunov stability the
local well-posedness would be essential. Once again the complete account of the works on this subject is beyond our objectives but
we briefly mention some of the classical results related to the three-dimensional case.
Escobedo and Vega~\cite{MR1434039} proved the local well-posedness in the $H^s$-setting with $s>(n-1)/2$. The charge-critical scaling power (at least in the massless case) being $(n-1)/2$, some works have been devoted to reach this endpoint. For instance, Machihara, Nakamura, Nakanishi and Ozawa show in \cite{MR2108356} that the $H^1$ regularity in the radial variable with an arbitrarily small regularity in the angular one is sufficient. This for instance settles the $H^1$-well-posedness problem
for radially symmetric initial data. This can for instance solve the problem for
initial data of the form \eqref{as-above} for \eqref{nld-0} in dimension $3$.
For \eqref{nld-0}, the non-linearity presents some null-structure which was
exploited by Bejenaru and Herr in \cite{MR3314499, MR3477346} to prove the local
well-posedness in the $H^1$ setting for $n=3$ and in the $H^{1/2}$ setting for $n=2$. Similar results have been obtained by
Candy~\cite{MR2829499} for the massive Thirring model
and
by Bournaveas and Candy~\cite{BournaveasCandy}
in the massless case
for both the Thirring and Soler models,
in dimension $2$ and $3$.

Once the spectral stability is known,
one hopes to prove the \emph{asymptotic stability}
of solitary waves using the dispersive estimates,
similarly to how this has been done for the nonlinear
Schr\"odinger equation.
First results in this direction
are already
appearing: in~\cite{MR2924465},
the asymptotic stability is proved
for the translation-invariant nonlinear Dirac equation
case in (3+1)D,
with the spectral stability assumption
playing a crucial role.
We point out that
due to the strong indefiniteness of the Dirac operator
(the energy conservation does not lead to any bounds
on the $H^{1/2}$-norm),
we do not know how to prove the \emph{orbital stability}
\cite{MR901236}
but via proving the asymptotic stability first.
The only exception is
the completely integrable massive Thirring model in (1+1)D,
where additional conserved quantities
arising from the complete integrability
allow to prove orbital stability of
solitary waves \cite{MR3147657,MR3462129}.
In the absence of spectral stability,
one expects to be able to prove \emph{orbital instability},
in the sense of~\cite{MR901236};
in \cite{MR2916078}, such instability is proved
in the context of the nonlinear Schr\"odinger equation,
in a very general situation.

\medskip

Since the isolated eigenvalues
depend continuously on the perturbation,
it is convenient to trace the location of ``unstable'' eigenvalues
(eigenvalues with positive real part)
for linearization at the solitary waves
$\phi\sb\omega e\sp{- \jj \omega t}$
considering $\omega$ as a parameter.
For example, if one knows that solitary waves
with $\omega$ in a certain interval
are spectrally stable,
one wants to know how and when the ``unstable'' eigenvalues
may emerge from the imaginary axis.

\medskip

\noindent
We distinguish the following scenari of the
development of linear instability:
\begin{enumerate}
\item
As $\omega$ changes,
two purely imaginary eigenvalues collide at $\lambda=0$
and then turn into a pair of one positive and one negative
eigenvalues.
In particular, such a collision happens
at the value of $\omega$
when
$dQ(\omega)/d\omega=0$,
which is the Vakhitov--Kolokolov condition \cite{VaKo},
or when $E(\omega)=0$;
see \cite{MR3311594}.
Above, $E(\omega)$ and $Q(\omega)$ are
the energy and charge of a solitary wave
$\phi\sb\omega e^{-\jj\omega t}$.
\item
The eigenvalues with positive real part
could bifurcate off the imaginary axis,
from the collision of purely imaginary
eigenvalues at some point in the spectral gap
but away from the origin.
Such a phenomenon has been observed
in the nonlinear Dirac equation
with cubic and quintic nonlinearities
in two dimensions
\cite{2015arXiv151203973C}.
\item
The eigenvalues with positive real part
could bifurcate from the essential spectrum.
We will address this scenario in the present work.
We will prove that
when $\abs{\omega}<m$,
the bifurcations from
$\lambda\in\jj\R$, $\abs{\lambda}<m+\abs{\omega}$,
are only possible when $\lambda$ is an eigenvalue.
We will also show that
such bifurcations are impossible from $\lambda\in\jj\R$
with $\abs{\lambda}>m+\abs{\omega}$,
that is,
outside of the region bounded by the embedded thresholds.
See Theorem~\ref{theorem-b}
and Remark~\ref{remark-nop}
below.

Let us mention related results.
The bifurcations from the embedded thresholds
$\pm(m+\abs{\omega})\jj$
have been observed in \cite{MR2217129}
in the context of coupled-mode equations.
At $\omega=0$, the bifurcations from
the collision of thresholds at $\lambda=\pm m \jj$
were shown in~\cite{MR1897705}
in the context of the massive Thirring model.
The bifurcation of a pair of a positive and a
negative eigenvalues at $\omega=m$
from the threshold $\lambda=0$
is shown in~\cite{MR3208458}
in the case when
$dQ(\omega)/d\omega>0$ for $\omega\lesssim m$,
in a formal agreement
with the Vakhitov--Kolokolov stability criterion \cite{VaKo}.
We do not know yet examples
of bifurcations from embedded eigenvalues.

We expect that in the
non-relativistic limit $\omega\to m$
the bifurcation of eigenvalues with nonzero real part
from $\lambda=0$ is completely characterized
by the Vakhitov--Kolokolov stability criterion:
that is, we expect that there are no such
bifurcations if $dQ(\omega)/d\omega\le 0$
for $\omega\lesssim m$.
We will address this scenario elsewhere.
\item
One might speculate that
at some value of $\omega$,
the eigenvalues with positive real part
could arrive from $\pm \jj\infty$.
We expect this never happens;
we will address this scenario elsewhere.
\end{enumerate}

While the nonlinear Dirac equation
is only a simplified model of self-interacting spinor fields,
our approach to the spectral stability of solitary waves
in the nonlinear Dirac equation is also applicable
to the Dirac--Maxwell system~\cite{MR0190520,wakano-1966}.
The local well-posedness of the Dirac--Maxwell system
was proved in~\cite{MR1391520}.
The existence of standing waves
$\phi(x)e^{-\jj\omega t}$
in the Dirac--Maxwell system
is proved in \cite{MR1386737} (for $\omega\in(-m,0)$)
and in~\cite{MR1618672} (for $\omega\in(-m,m)$);
for an overview of these results, see~\cite{MR1897689}.
We hope that the spectral stability
in this model could be approached by our methods.
Another situation where our methods are applicable
is the analysis of stability of gap solitons
in nonlinear coupled-mode equations.
Such systems appear
in the context of photonic crystals
\cite{MartijnDeSterke1994203},
where they
describe counter-propagating light waves
interacting with a linear grating in optical waveguides
made of material with periodically changing refractive index
\cite{1996PhRvE..54.1969D,MR1825812}.
Coupled-mode systems also describe
matter-wave Bose--Einstein condensates trapped in an optical lattice
\cite{PhysRevE.70.036618}.
The numerical analysis
of the spectrum of the linearizations
at the gap solitons
is performed in~\cite{PhysRevLett.80.5117,MR2217129}.
The stability analysis of
small-amplitude gap solitons
based on the study of bifurcations
from the embedded eigenvalues
of the linear equation
in the external potential is
in~\cite{MR2513792}.

\bigskip

Here is the plan of the present analysis.
The results are stated in Section~\ref{sect-results}.
Our main tools,
the Carleman--Berthier--Georgescu estimates,
are derived in Section~\ref{sect-carleman}.
The exponential decay
of solitary waves and of eigenfunctions
is proved in Section~\ref{sect-exp}.
Bifurcation of eigenvalues
from the essential spectrum is considered
in Sections~\ref{sect-b}.
The unique continuation
principle for the Dirac operator
from \cite[Appendix]{MR880983},
adapted to our needs,
is in Appendix~\ref{Sec:UniqueContinuation};
incidentally,
our Theorem~\ref{theorem-ln}
proves the conjecture put forward in \cite{MR865834}
on unique continuation principle for the Dirac operator $D_m+V$
with $V\in L_{\mathrm{loc}}^q(\R^n)$,
$q\ge n$ for $n\ne 2$ and $q>2$ for $n=2$.

\subsection{Notations}\label{sec:notations}

We denote the free Dirac operator by
\[
D\sb m
=-\jj \bm\alpha\cdot\bm\nabla+\beta m,
\qquad
m>0.
\]
We will also use the notation
$
D_0=-\jj \bm\alpha\cdot\bm\nabla.
$
Here
$
\bm\alpha\cdot\bm\nabla
=\sum\sb{\jmath=1}\sp{n}\alpha\sp\jmath\frac{\p}{\p x\sp\jmath},
$
with
$\alpha\sp\jmath$ and $\beta$
being hermitian $N\times N$ Dirac matrices
which satisfy
\[
\alpha\sp\jmath\alpha\sp k
+\alpha\sp k\alpha\sp\jmath=2\delta\sb{\jmath k}I\sb N,
\ \quad
\alpha\sp\jmath\beta+\beta\alpha\sp\jmath=0,
\ \quad
\beta\sp 2=I\sb N,
\ \quad
1\le\jmath,k\le n.
\]
$I\sb N$ is the $N\times N$ identity matrix.
The anticommutation relations lead to
\[
\tr\alpha\sp\jmath
=\tr\beta^{-1}\alpha\sp\jmath\beta=-\tr\alpha\sp\jmath=0,
\qquad
1\le\jmath\le n,
\]
and similarly $\tr\beta=0$;
together with $\sigma(\alpha\sp\jmath)=\sigma(\beta)=\{\pm 1\}$,
this yields the conclusion that $N$ is even.
Let us also mention that
the spatial dimension $n$ and the number
of spinor components $N$
satisfy the relation $N\ge 2^{[(n+1)/2]}$
\cite[Chapter 1, \S5.3]{MR1401125}.

We denote $r=\abs{x}$
for $x\in\R^n$,
and, abusing notations,
we will also denote the operator of multiplication
with $\abs{x}$ and $\langle x\rangle=(1+\abs{x}^2)^{1/2}$
by $r$ and $\langle r\rangle$, respectively.

The Fourier transform
$\mathcal{F}:\,\mathscr{S}(\R^n)\to\mathscr{S}(\R^n)$
is defined by
\[
(\mathcal{F}u)(\xi)=\int\sb{\R^n} e^{-\jj\xi\cdot x}u(x)\,dx,
\qquad
(\mathcal{F}^{-1}v)(x)=\int\sb{\R^n} e^{\jj\xi\cdot x}v(\xi)\,\frac{d\xi}{(2\pi)^n},
\]
and extended by duality to
$\mathscr{S}'(\R^n)\to\mathscr{S}'(\R^n)$.

Given an open subset $\Omega$ of $\R\sp n$,
we denote the standard $L^2$-based
Sobolev spaces by
$H\sp k(\Omega,\C^N)$.
We denote
$H^\infty(\Omega,\C^N)=\cap\sb{k\in\N}H\sp k(\Omega,\C^N)$.
We use the notation $H\sp k\sb{\mathrm{comp}}(\Omega,\C^N)$
for the subset of $H\sp k(\Omega,\C^N)$
made up by function with compact support in $\Omega$.
Its closure in $H\sp k(\Omega,\C^N)$ is denoted $H\sp k\sb 0(\Omega,\C^N)$.

The notation $H\sp k\sb{\mathrm{loc}}(\Omega,\C^N)$
is used
for the set of function $u\in L\sp{2}(\Omega,\C^N)$ such that
$\eta u \in H\sp{k}(\Omega,\C^N)$ for
any scalar smooth compactly supported function $\eta$ on $\Omega$.
The notation $H\sp k\sb{0,\mathrm{loc}}(\Omega,\C^N)$ stands for the set of
functions $u\in L\sp{2}(\Omega,\C^N)$ such that
$\eta u \in H\sp k\sb 0(\Omega,\C^N)$
for any scalar smooth compactly supported function $\eta$ on $\R\sp n$.

For $s,\,k\in\R$,
we define the weighted Sobolev spaces
\begin{equation}
H\sp k\sb{s} (\R\sp n,\C^N)=\left\{ u\in\mathscr{S}'(\R\sp n,\C^N),\,
\norm{u}\sb{H\sp k\sb s}
<\infty\right\},
\qquad
\norm{u}\sb{ H\sp k\sb s}=
\norm{\langle r\rangle\sp s\langle P\rangle\sp k u}\sb{L\sp 2},
\end{equation}
with $\langle P\rangle$
understood as the multiplication by $\sqrt{1+\xi^2}$
on the Fourier transform side.

We write
$L\sp{2}\sb{s}(\R\sp n,\C^N)$
for~$H\sp{0}\sb{s}(\R\sp n,\C^N)$.
For $u\in L\sp 2(\R\sp n,\C^N)$,
we denote $\norm{u}=\norm{u}\sb{L\sp 2}$.

For any pair of normed vector spaces $E$ and $F$,
let $\mathscr{B}(E,F)$
denote the set of bounded linear maps from $E$ to $F$.

For an unbounded linear operator $A$
acting in a Banach space $X$
with a dense domain $D(A)\subset X$,
the spectrum $\sigma(A)$
is the set of values $\lambda \in\C$ such that
the operator $A-\lambda:\;D(A)\to X$
does not have a bounded inverse.
The generalized null space of $A$ is defined by
\[
\mathscr{N}\sb g(A)
:=
\mathop{\cup}\limits\sb{k\in\N} \ker(A^k)
=
\mathop{\cup}\limits\sb{k\in\N}
\{
v\in D(A)\sothat
A^j v\in D(A)\ \forall j<k,
\ A^k v=0
\}.
\]
The discrete spectrum $\sigma\sb{\mathrm{disc}}(A)$ is the
set of isolated eigenvalues
$\lambda\in\sigma(A)$
of finite algebraic multiplicity,
\[
\dim\mathscr{N}\sb g(A-\lambda)<\infty.
\]
The essential spectrum $\sigma\sb{\mathrm{ess}}(A)$ is the
complementary set of discrete spectrum in the spectrum.
The point spectrum $\sigma\sb{\mathrm{p}}(A)$
is the set of eigenvalues (isolated or embedded
into the essential spectrum).

For $R>0$, we denote
\[
\mathbb{B}^n\sb R(x_0)
=\{x\in\R\sp n\sothat\abs{x-x_0}<R\},
\qquad
\mathbb{B}^n\sb R
=\mathbb{B}^n\sb R(0).
\]
We also use the following notation for the complement to the closure
of a ball:
\begin{equation}\label{def-varOmega}
\varOmega\sb R
=
\R^n\setminus\overline{\mathbb{B}^n\sb R}
=
\{x\in\R\sp n\sothat\abs{x}>R\}.
\end{equation}

\subsection*{Acknowledgments}
 Support from the grant
ANR-10-BLAN-0101 of the French Ministry of Research is gratefully
acknowledged by the first author.
The second author was partially supported by
Universit\'e Franche-Comt\'e
and by the Russian Foundation for Sciences (project No. 14-50-00150).

\section{Main results}
\label{sect-results}

We consider the nonlinear Dirac equation \eqref{nld-0},
\begin{equation}\label{nld}\tag{2.1; NLDE}
 \jj \p\sb t\psi=D\sb m\psi-f(\psi\sp\ast\beta\psi)\beta\psi,
\qquad
\psi(x,t)\in\C^N,
\quad
x\in\R\sp n.
\end{equation}
\setcounter{equation}{1}
The nonlinearity is such that
the equation is $\mathbf{U}(1)$-invariant
and hamiltonian,
with the Hamiltonian functional given by
\begin{equation}\label{def-E}
E(\psi)=\int\sb{\R^n}
\big(
\psi\sp\ast D\sb m\psi-F(\psi\sp\ast\beta\psi)
\big)\,dx,
\qquad
\mbox{where}
\quad
F(s)=\int\sb 0\sp s f(t)\,dt.
\end{equation}
The charge conserved due to the $\mathbf{U}(1)$-invariance
is given by
\begin{equation}\label{def-q}
Q(\psi)
=\int\sb{\R^n}\psi\sp\ast(x,t)\psi(x,t)\,dx.
\end{equation}


If $\phi\sb\omega(x)e^{-\jj\omega t}$ is a solitary wave solution
to \eqref{nld},
then the profile $\phi\sb\omega$ satisfies the stationary equation
\begin{equation}\label{nld-stationary}
\omega\phi\sb\omega
=D\sb m\phi\sb\omega-f(\phi\sb\omega\sp\ast\beta\phi\sb\omega)\beta\phi\sb\omega.
\end{equation}
Below, we assume that $\phi\sb\omega$ is localized in the sense that $\phi\sb\omega\in L^2(\R^n,\C^N)$.

To consider the linearization at this solitary wave,
we assume that in \eqref{nld} one has $f\in C^1(\R)$,
and consider the solution in the form of the Ansatz
\[
\psi(x,t)=(\phi\sb\omega(x)+\rho(x,t))e\sp{-\jj \omega t}.
\]
The linearization at a solitary wave
is then the linearized equation on $\rho$, given by
\begin{equation}\label{nld-linear-0}
\jj\p\sb t\rho=\mathcal{L}(\omega)\rho,
\end{equation}
where
\begin{equation}\label{def-cal-l}
\mathcal{L}(\omega)
=D\sb m
-\omega-f(\phi\sb\omega\sp\ast\beta\phi\sb\omega)\beta
-2f'(\phi\sb\omega\sp\ast\beta\phi\sb\omega)
\beta\phi\sb\omega
\Re(\phi\sb\omega\sp\ast\beta\,\,\cdot\,).
\end{equation}
The operator
$\mathcal{L}(\omega)$
is not $\C$-linear because of the term
with $\Re(\phi\sb\omega\sp\ast\beta\,\,\cdot\,)$.
To work with $\C$-linear operators,
we introduce the matrices representing the action of
$\alpha\sp\jmath$ with $1\le\jmath\le n$, $\beta$, and $-\jj$
on $\R^{2N}$-valued functions:
\[
\bmupalpha\sp\jmath
=
\begin{bmatrix}
\Re\alpha\sp\jmath&-\Im\alpha\sp\jmath\\\Im\alpha\sp\jmath&\Re\alpha\sp\jmath
\end{bmatrix},
\qquad
\bmupbeta
=
\begin{bmatrix}\Re\beta&-\Im\beta\\\Im\beta&\Re\beta\end{bmatrix},
\qquad
\eubJ
=\left[\begin{matrix}0&I\sb N\\-I\sb N&0\end{matrix}\right],
\]
where the real part of a matrix is the matrix
made of the real part of its entries
(and similarly for the imaginary part of a matrix).

When
$\phi\sb\omega(x)e\sp{-\jj \omega t}$
is a solitary wave solution to \eqref{nld},
the profile $\phi\sb\omega$ satisfies \eqref{nld-stationary};
this means that
\[
\bmupphi\sb\omega
=\begin{bmatrix}\Re\phi\sb\omega\\\Im\phi\sb\omega\end{bmatrix}
\in\R^{2N}
\]
satisfies
\begin{equation}\label{nld-stationary-bm}
\eubL\sb{-}(\omega)\bmupphi\sb\omega
=0,
\end{equation}
where
\begin{equation}\label{def-lm}
\eubL\sb{-}(\omega)
=
\eubD\sb m-\omega
-f(\bmupphi\sb\omega\sp\ast\bmupbeta\bmupphi\sb\omega)\bmupbeta,
\end{equation}
with
\begin{equation}\label{def-dm}
\eubD\sb m=\eubJ \bmupalpha\cdot\bm\nabla+m\bmupbeta
\end{equation}
representing the action of $D\sb m$
on $\R\sp{2N}$-valued functions.
The action of the operator \eqref{def-cal-l}
is represented in $\R\sp{2N}$ by
\begin{equation}\label{def-ll}
\eubL(\omega)
=
\eubD\sb m-\omega
-f(\bmupphi\sb\omega\sp\ast\bmupbeta\bmupphi\sb\omega)\bmupbeta
-2(\bmupphi\sb\omega\sp\ast\bmupbeta\,\,\cdot\,)
f'(\bmupphi\sb\omega\sp\ast\bmupbeta\bmupphi\sb\omega)
\bmupbeta\bmupphi\sb\omega.
\end{equation}
Recall that we assume $f\in C^1(\R)$\footnote{One can proceed with the linearization
assuming only Fr\'echet or even only G\^ateaux differentiability.},
hence under an assumption on $\phi_\omega$
such as
\[
\phi\sb\omega\in L^\infty(\R^n,\C^N),
\]
both $\eubL\sb{-}$ and $\eubL$
are closed and self-adjoint on the domain
\begin{equation}\label{def-xx}
\scrX=H\sp 1(\R\sp n,\C\sp{2N})
=H\sp 1(\R\sp n,\C\otimes\sb{\R}\R\sp{2N})
\end{equation}
to which they are extended by $\C$-linearity.
Thus, the linearization \eqref{nld-linear-0}
at the solitary wave
takes the form
\begin{equation}\label{nld-linear}
\p\sb t\bmuprho
=\eubJ\eubL(\omega)\bmuprho,
\qquad
\bmuprho(x,t)=\left[\begin{matrix}\Re\rho(x,t)\\\Im\rho(x,t)\end{matrix}\right]
\in\R\sp{2N}.
\end{equation}

Let us present some general results
on the point spectrum of the linearization operator.

\begin{lemma}\label{lemma-sym}
The operator $\eubJ\eubL$ is closed
and its spectrum $\sigma(\eubJ\eubL)$
is symmetric with respect to the real and imaginary axes.
\end{lemma}

\begin{proof}
The closedness is immediate.

Let $\lambda\in\sigma(\eubJ\eubL)$.
The inclusion
$\bar\lambda\in\sigma(\eubJ\eubL)$
follows from
$\eubJ\eubL$ acting invariantly in the subspace
$L^2(\R^n,\R^{2N})\subset L^2(\R^n,\C^{2N})$.
The inclusion
$-\bar\lambda\in\sigma(\eubJ\eubL)$ follows from $(-\eubJ\eubL)\sp\ast$
being conjugate to $\eubJ\eubL$:
\[
(-\eubJ\eubL)\sp\ast
=\eubL\sp\ast(-\eubJ)\sp\ast
=\eubL\eubJ
=\eubJ\sp{-1}(\eubJ\eubL)\eubJ.
\qedhere
\]
\end{proof}

For the reader's convenience,
we recall the results
\cite{MR3311594}
on the spectral subspace
of $\eubJ\eubL(\omega)$
(see \eqref{def-ll})
corresponding to the zero eigenvalue.

\begin{lemma}
[$0\in\sigma\sb{\mathrm{p}}(\eubJ\eubL)$]
\label{lemma-dim-ker-nld} Assume there exists an open interval $\mathcal{I}\subset(-m,m)$ such that for each $\omega \in\mathcal{I}$ \eqref{nld-stationary} has a solution
$\phi\sb\omega\in L^2_1(\R^n,\C^N)\cap H^1(\R^n,\C^N)$ such that the mapping
$\omega\in\mathcal{I}\mapsto \phi\sb\omega\in L^2(\R^n,\C^N)$ is differentiable.
Then one has:
\[
\mathop{\rm Span}\left\{
\eubJ\bmupphi\sb\omega,
\ \p\sb{x\sp\jmath}\bmupphi\sb\omega
\sothat
1\le \jmath\le n
\right\}
\subset
\ker\eubJ\eubL(\omega),
\]
\[
\mathop{\rm Span}\left\{
\eubJ\bmupphi\sb\omega,
\ \p\sb\omega\bmupphi\sb\omega,
\ \p\sb{x\sp\jmath}\bmupphi\sb\omega,
\ \bmupalpha\sp\jmath\bmupphi\sb\omega-2\omega x\sp\jmath\eubJ\bmupphi\sb\omega
\sothat
1\le \jmath\le n
\right\}
\subset
\mathscr{N}\sb g(\eubJ\eubL(\omega)).
\]
\end{lemma}

\begin{remark}
We do not claim
the complete characterization of $\ker\eubJ\eubL$
(in particular, there would be eigenvectors
which correspond to rotational symmetries).
According to \cite{VaKo},
the size of the Jordan block corresponding to
$\eubJ\bmupphi\sb\omega$ is exactly two
(nothing but $\eubJ\bmupphi\sb\omega$ and
$\p\sb\omega\bmupphi\sb\omega$)
as long as $\p\sb\omega Q(\omega)\ne 0$.
By \cite{MR3311594},
the size of the Jordan block
corresponding to each of $\p\sb{x\sb j}\bmupphi\sb\omega$
is exactly two
(nothing but
$\p\sb{x\sb j}\bmupphi\sb\omega$
and
$\bmupalpha\sp\jmath\bmupphi\sb\omega-2\omega x\sp\jmath\eubJ\bmupphi\sb\omega$)
as long as $E(\omega)\ne 0$,
with $E(\omega)=E(\phi\sb\omega e^{-\jj\omega t})$
defined in \eqref{def-E}.
\end{remark}

\begin{proof}
Recall that
the operators
$\eubL\sb{-}$, $\eubL$
are introduced in \eqref{def-lm} and \eqref{def-ll}.
We compute:
\begin{equation}\label{dvk-ast1}
\eubL(\eubJ\bmupphi\sb\omega)
=
\eubL\sb{-}(\eubJ\bmupphi\sb\omega)
-2(\bmupphi\sb\omega\sp\ast\bmupbeta(\eubJ\bmupphi\sb\omega))
f'(\bmupphi\sb\omega\sp\ast\bmupbeta\bmupphi\sb\omega)\bmupbeta\bmupphi\sb\omega
=0.
\end{equation}
Above,
the first term vanishes since $\bmupphi\sb\omega(x)$ satisfies
$\eubL\sb{-}\bmupphi\sb\omega=0$
(cf. \eqref{nld-stationary-bm})
and $\eubJ$ commutes with $\eubL\sb{-}$;
the second term vanishes since
the quantity
$\bmupphi\sb\omega\sp\ast\bmupbeta\eubJ\bmupphi\sb\omega$
is real-valued, while $\eubJ\bmupbeta$ is skew-adjoint.
Taking the derivative of
\eqref{nld-stationary-bm}
with respect to $x\sp\jmath$, $1\le \jmath\le n$, yields
\begin{equation}\label{dvk-ldxf}
0=
\eubL\sb{-}\p\sb{x\sp\jmath}\bmupphi\sb\omega
-2(\bmupphi\sb\omega\sp\ast\bmupbeta\p\sb{x\sp\jmath}\bmupphi\sb\omega)
f'(\bmupphi\sb\omega\sp\ast\bmupbeta\bmupphi\sb\omega)
\bmupbeta\bmupphi\sb\omega
=\eubL\p\sb{x\sp\jmath}\bmupphi\sb\omega.
\end{equation}
Now let us consider the
generalized null space.
Taking the derivative of \eqref{nld-stationary-bm}
with respect to $\omega$,
we get
\begin{equation}\label{dvk-ldof}
0=
\eubL\sb{-}\p\sb\omega\bmupphi\sb\omega
-\bmupphi\sb\omega
-2(\bmupphi\sb\omega\sp\ast\bmupbeta\p\sb\omega\bmupphi\sb\omega)
f'(\bmupphi\sb\omega\sp\ast\bmupbeta\bmupphi\sb\omega)
\bmupbeta\bmupphi\sb\omega
=\eubL\p\sb\omega\bmupphi\sb\omega-\bmupphi\sb\omega,
\end{equation}
which shows that $\p\sb\omega\bmupphi\sb\omega\in\mathscr{N}\sb
g(\eubJ\eubL(\omega))$.
Since
\[
\bmupphi\sb\omega\sp\ast\bmupbeta\bmupalpha\sp\jmath\bmupphi\sb\omega
+(\bmupalpha\sp\jmath\bmupphi\sb\omega)\sp\ast\bmupbeta\bmupphi\sb\omega
=\bmupphi\sb\omega\sp\ast\{\bmupbeta,\bmupalpha\sp\jmath\}\bmupphi\sb\omega=0,
\]
we have
\begin{equation}\label{dvk-hap1}
\eubL(\bmupalpha\sp\jmath\bmupphi\sb\omega)
=\eubL\sb{-}(\bmupalpha\sp\jmath\bmupphi\sb\omega)
=
2 \eubJ \p\sb{x\sp\jmath}\bmupphi\sb\omega
-2\omega\bmupalpha\sp\jmath\bmupphi\sb\omega
-\bmupalpha\sp\jmath\eubL\sb{-}\bmupphi\sb\omega
=2 \eubJ \p\sb{x\sp\jmath}\bmupphi\sb\omega
-2\omega\bmupalpha\sp\jmath\bmupphi\sb\omega.
\end{equation}
Similarly,
since
$
\bmupphi\sb\omega\sp\ast\bmupbeta( x\sp\jmath\eubJ\bmupphi\sb\omega)
+( x\sp\jmath\eubJ\bmupphi\sb\omega)\sp\ast\bmupbeta\bmupphi\sb\omega=0$,
\begin{equation}\label{dvk-hap2}
\eubL( x\sp\jmath\eubJ\bmupphi\sb\omega)
=\eubL\sb{-}( x\sp\jmath\eubJ\bmupphi\sb\omega)
=-\bmupalpha\sp\jmath\bmupphi\sb\omega
+ x\sp\jmath\eubJ\eubL\sb{-}\bmupphi\sb\omega
=-\bmupalpha\sp\jmath\bmupphi\sb\omega.
\end{equation}
Using \eqref{dvk-hap1} and \eqref{dvk-hap2},
we compute:
\begin{equation}\label{dvk-ast2}
\eubJ
\eubL(\bmupalpha\sp\jmath\bmupphi\sb\omega-2\omega
x\sp\jmath\eubJ\bmupphi\sb\omega)
=\eubJ(2 \eubJ \p\sb{x\sp\jmath}\bmupphi\sb\omega)
=-2\p\sb{x\sp\jmath}\bmupphi\sb\omega,
\end{equation}
which shows that
$\bmupalpha\sp\jmath\bmupphi\sb\omega-2\omega x\sp\jmath\eubJ\bmupphi\sb\omega
\in\mathscr{N}\sb g(\eubJ\eubL(\omega))$.
\end{proof}

\subsection{Essential spectrum}
\label{section:Weyl}

To be able to study the spectrum of the linearization
at a solitary wave,
we need to work with operators
of the form $JL$,
where
\begin{equation}\label{l-d-v}
L=D\sb m-\omega+V(x),
\qquad
V\in L^\infty\big(\R\sp n,\End(\C^N)\big),
\qquad
\omega\in[-m,m],
\end{equation}
where $\End(\C^N)$ denotes an endomorphism of $\C^N$,
while
$J\in\End(\C^{N})$ is skew-adjoint and invertible,
such that
\[
J^2=-I\sb{\C^N},
\qquad
[J,D_m]=0.
\]
The domain of $L$ is
$D(L)=H\sp 1(\R\sp n,\C^N)$.
The operator $L$ is closed.
When $V$ is hermitian-valued, $L$ is also self-adjoint;
but for the moment, we do not assume that $V$ is hermitian-valued.
The results on the spectrum of $JL$
are applicable to the operator
$\eubJ\eubL(\omega)$ from \eqref{nld-linear}
which describes the linearization at a solitary wave.
The only reason to change notations from $\eubJ\eubL$ to $JL$
is in order to change from $\C^{2N}$-valued spinors
to $\C^N$-valued ones.

\begin{lemma}\label{lemma-ess}
Let $n\ge 1$,
$\omega\in[-m,m]$, and assume that
\begin{equation}\label{v-relatively-compact-0}
V\in L^q(\R^n,\End(\C^{N})),
\qquad
\begin{cases}
2\le q<\infty,
\qquad n=1;
\\
2<q<\infty,
\qquad n=2;
\\
n\le q<\infty,
\qquad n\ge 3.
\end{cases}
\end{equation}
Then
\[
\sigma\sb{\mathrm{ess}}(J(D\sb m-\omega+V))
=\jj (\R\backslash(-m+\abs{\omega},m-\abs{\omega})).
\]
\end{lemma}

\begin{remark}
The above result
is similar to \cite[Chapter XIII, Problem 41]{MR0493421},
adapted to the case of the Dirac operator.
\end{remark}

\begin{proof}
We will give the proof for $n\ge 3$;
the cases $n\le 2$ are considered similarly.
By the Weyl theorem,
it is enough to prove that for
$z\in\C,
$
$z\not\in\jj
\big(\R\setminus(|\omega|-m,\,m-|\omega|)\,\big),
$
the operator
$V\big(J(D\sb m-\omega)-z\big)^{-1}:\;L^2\to L^2$
is compact.

Let $2<p<\infty$ be such that
\begin{equation}\label{2-p-q}
\frac 1 2=\frac{1}{p}+\frac{1}{q}.
\end{equation}
Note that, in particular, for $n\ge 3$ one has
$p\le\frac{2n}{n-2}$.
Let $\chi$ be the characteristic function of the unit ball,
and set $\chi_j(x)=\chi(x/j)$, $j\in\N$.
Then
\[
\norm{(1-\chi_j)V\big(J(D\sb m-\omega)-z\big)^{-1}}
\sb{L^2\to L^2}
\le
\norm{(1-\chi_j)V}\sb{L^p\to L^2}
\norm{\big(J(D\sb m-\omega)-z\big)^{-1}}\sb{L^2\to L^p}
\]
is bounded
(the second factor in the right-hand side
is bounded due to the Sobolev embedding
$H^1(\R^n)\subset L^p(\R^n)$),
and moreover the norm
of the operator of multiplication by $(1-\chi_j)V$
goes to zero as $j\to\infty$:
\[
\norm{(1-\chi_j)V}\sb{L^p\to L^2}
\le
\norm{(1-\chi_j)V}\sb{L^q}
\mathop{\longrightarrow}\limits\sb{j\to\infty}0.
\]
Above, the inequality is due to
\eqref{v-relatively-compact-0} and \eqref{2-p-q}.
On the other hand, the operator
$\chi_j V\big(J(D\sb m-\omega)-z\big)^{-1}:\;L^2\to L^2$
is compact, since
\begin{equation}\label{j-d-z}
\chi\sb j\big(J(D\sb m-\omega)-z\big)^{-1}:
\;L^2(\R^n,\C^N)\to H^1(\R^n,\C^N)\subset L^{p}(\R^n,\C^N)
\end{equation}
is compact by the Rellich--Kondrachov compactness theorem,
while the operator of multiplication
by $V$ acts continuously from $L^p$ to $L^2$.
Therefore, the operator
\[
V\big(J(D\sb m-\omega)-z\big):\;L^2(\R^n,\C^N)\to L^2(\R^n,\C^N)
\]
is a limit of the sequence of compact operators
$\chi_j V\big(J(D\sb m-\omega)-z\big)$
in the uniform operator norm, and hence is also compact.

\medskip

Weyl's theorem on the essential spectrum
\cite[Theorem XIII.14, Corollary 2]{MR0493421}
applied to operators $\jj J(D\sb m-\omega)$
and $\jj J(D\sb m-\omega+V)$,
which are considered on the domain
$H\sp 1(\R\sp n,\C^N)$,
allows us to conclude that
\[
\sigma\sb{\mathrm{ess}}\big(J(D\sb m-\omega+V)\big)
=
\sigma\sb{\mathrm{ess}}\big(J(D\sb m-\omega)\big)
=
\big\{
\pm\jj z\sothat z\in\sigma\sb{\mathrm{ess}}(D\sb m-\omega)
\big\},
\]
with
\[
\sigma\sb{\mathrm{ess}}(D\sb m-\omega)
=\R\setminus(-m-\omega,m-\omega),
\]
where we took into account that $[J,D\sb m]=0$
and $\sigma(J)=\{\pm\jj\}$.
\end{proof}

\subsection{Carleman--Berthier--Georgescu estimates}

One of our main tools is
the three-dimensional Carleman--Berthier--Georgescu estimates
proved in~\cite[Theorem 5]{MR880983}
which we adapt to the context
of linearization at solitary waves
and generalize to any dimension.

Following the proof of \cite[Theorem 3]{MR880983},
we define the following set of admissible phase functions.

\begin{definition}\label{definition-c}
Let $\lambda\in\C$ and $\omega\in[-m,m]$
be such that
\begin{equation}\label{lambda-omega-m}
\abs{\lambda}-\abs{\omega}>m.
\end{equation}
Let $M\ge 1,\,\calN\ge 1,\,\rho\ge 1$ and $\nu>0$.
We define the subset
\[
\mathscr{C}\sb{\lambda,\omega}(M,\calN,\rho,\nu)\subset C^2(\R\sb{+})
\]
to be the set of functions which satisfy the following properties:
\begin{enumerate}\label{def-c}
\item
$0<\varphi'\le \calN r$, $\ \forall r\ge\rho$ ;
\item
$r\abs{\varphi''}\le M\varphi'$, $\ \forall r\ge\rho$ ;
\item
$(\abs{\lambda}-\abs{\omega})^2-m^2+\varphi'\sp 2+2r\varphi'\varphi''\ge \nu$,
$\ \forall r\ge\rho$ ;
\item
if $\varphi'$ is unbounded then
$\varphi''(r)\geq 0$ for $r\ge\rho$
and
$\lim_{r\to\infty}\varphi'(r)=+\infty$.
\end{enumerate}
We also denote
\[
\mathscr{C}\sb\lambda(M,\calN,\rho,\nu)
:=\mathscr{C}\sb{\lambda,0}(M,\calN,\rho,\nu).
\]
We define
\begin{equation}\label{def-Lambda-pm}
\Lambda\sb{+}=\abs{\lambda}+\abs{\omega},
\qquad
\Lambda\sb{-}=\abs{\lambda}-\abs{\omega}>m,
\end{equation}
and introduce the weight functions
\begin{equation}\label{def-mu-l}
\mu(r)
=
2
\Big(
n+16\Lambda\sb{+}^2 r^2+8r\varphi'(r)
\Big)^{1/2},
\qquad
r\ge\rho;
\end{equation}
\begin{equation}\label{def-g-l}
\gamma(r)=
\left(
\Lambda\sb{-}^2-m^2+\varphi'(r)^2+2r\varphi'(r)\varphi''(r)\right)^{1/2},
\qquad
r\ge\rho.
\end{equation}
\end{definition}

\begin{theorem}[Carleman--Berthier--Georgescu estimates for $JL$]
\label{theorem-Carleman-l}
Let $n\ge 1$, $m>0$,
$\omega\in[-m,m]$.
Let $J\in\End(\C^N)$ be skew-adjoint and invertible,
such that $J^2=-I\sb{\C^N}$, $[J,D_m]=0$.
Let
\[
\lambda\in\jj\R,
\qquad
\abs{\lambda}>m+\abs{\omega}.
\]
\begin{enumerate}
\item
\label{theorem-Carleman-l-i}
Assume that
$
\varphi\in\mathscr{C}\sb{\lambda,\omega}(M,\calN,\rho,\nu)
$
with some $M\geq 1,\,\calN\geq 1,\,\rho\ge 1$, and $\nu>0$.

Let
$V\in\mathscr{B}\left(H\sp 1\sb{\mathrm{comp}}(\varOmega\sb\rho,\C^N),\,L\sp 2(\varOmega\sb\rho,\C^N)\right)$ be a multiplication operator,
and assume that there are $\kappa\in[0,1)$
and $R_1=R_1(\varphi,V)\ge\rho$
such that
\begin{equation}\label{ineq-0-l}
\norm{\mu V v}
\le
\kappa
\left(
\norm{\bm\nabla v}^2
+
\norm{\gamma v}^2
\right)^{1/2},
\qquad \forall v\in H^1\sb 0(\varOmega_{R_1},\C^N).
\end{equation}
Then there is
$R=R(\varphi,V)\ge R_1$
such that
for any $u\in H\sp 1(\R\sp n,\C^N)$
with $\supp u\subset\varOmega\sb{R}$
and
\[
\mu e\sp\varphi (J(D\sb m-\omega+V)-\lambda)u
\in L\sp 2(\R\sp n,\C^N),
\]
the functions
$\gamma e\sp{\varphi}u$,
\ $\bm\nabla(e\sp{\varphi}u)$,
\ $(r\varphi')\sp{1/2}\partial\sb r(e\sp{\varphi}u)$
are in $L\sp 2(\R\sp n,\C^N)$, and moreover
\begin{eqnarray}\label{car-l}
\norm{\bm\nabla(e^\varphi u)}^2
+
2\norm{(r\varphi')^{1/2}\p\sb r(e^\varphi u)}^2
+
\norm{
\gamma
e\sp{\varphi}u}\sp 2
\le
\frac{1}{(1-\kappa)^2}
\norm{
\mu
e\sp\varphi(J(D\sb m-\omega+V)-\lambda)u}\sp 2.
\end{eqnarray}
\item
\label{theorem-Carleman-l-ii}
Assume that there are $\kappa\in[0,1)$
and $R_2=R_2(V)$
such that
\begin{eqnarray}\label{Eq:VSmallInfinity-l}
&
\norm{\big(n+16\Lambda\sb{+}^2 r^2+8r f(r)\big)^{1/2} V v}^2
\le
\kappa^2
\left(
\norm{\bm\nabla v}^2
+
\norm{
\big(\Lambda\sb{-}^2-m^2+f(r)^2
\big)^{1/2}v}^2
\right),
\\[1ex]
&
\forall
f\in C(\R),
\quad
f\ge 0,
\quad
\sup\sb{r>0}\frac{f(r)}{\langle r\rangle}<\infty;
\qquad \forall v\in H^1\sb 0(\varOmega_{R_2},\C^N).
\nonumber
\end{eqnarray}
Let $M\ge 1$, $\calN\ge 1$, $\rho\ge 1$, and $\nu>0$.
Then there is
\[
R=R(M,\rho,m,n,\lambda,\omega,V)\ge\max(\rho,R_2(V))
\]
such that for any
$\varphi\in\mathscr{C}\sb{\lambda,\omega}(M,\calN,\rho,\nu)$
with
$\varphi''(r)\ge 0$, $r\ge\rho$,
and for any
\[
u\in H\sp 1\sb 0(\varOmega\sb{R},\C^N)
\]
which satisfies
\[
\mu e\sp\varphi (J(D\sb m-J+V)-\lambda)u \in L\sp 2(\R\sp n,\C^N),
\]
the functions
$\gamma e\sp{\varphi}u$,
\ $\bm\nabla(e\sp{\varphi}u)$,
\ $(r\varphi')\sp{1/2}\partial\sb r(e\sp{\varphi}u)$
are in $L\sp 2(\R\sp n,\C^N)$, and moreover
\begin{eqnarray}\label{carV-l}
\norm{\bm\nabla(e^\varphi u)}^2
+
2\norm{(r\varphi')^{1/2}\p\sb r(e^\varphi u)}^2
+
\norm{
\gamma e\sp{\varphi}u}\sp 2
\le
\frac{1}{(1-\kappa)^2}\norm{
\mu
e\sp\varphi (J(D\sb m-\omega+V)-\lambda)u}\sp 2.
\end{eqnarray}
\end{enumerate}
\end{theorem}

\begin{remark}
That is,
in Theorem~\ref{theorem-Carleman-l}~\itref{theorem-Carleman-l-ii}
we state that if $\varphi''\ge 0$,
then $R(\varphi,V)$
from Theorem~\ref{theorem-Carleman-l}~\itref{theorem-Carleman-l-i}
depends on the class $\mathscr{C}\sb{\lambda,\omega}(M,\calN,\rho,\nu)$,
but not on a particular representative
$\varphi\in\mathscr{C}\sb{\lambda,\omega}(M,\calN,\rho,\nu)$.
\end{remark}

In Section~\ref{sect-carleman},
we state and prove
Theorem~\ref{theorem-Carleman},
which is a version of Theorem~\ref{theorem-Carleman-l}
for the Dirac operator $D_m+V$.
The proof of Theorem~\ref{theorem-Carleman-l}
follows from Remark~\ref{remark-c-l}.

\medskip

Let us also formulate a convenient condition on the potential $V$
which would be sufficient for the condition
\eqref{Eq:VSmallInfinity-l}
to be satisfied.

\begin{lemma}
\label{lemma-v-small-coulomb}
A sufficient condition
for
\eqref{Eq:VSmallInfinity-l}
to hold
for some $R_2<\infty$ is that
there are $\kappa'\in(0,\kappa)$ and $R<\infty$
such that
\begin{equation}\label{v-small-coulomb}
\abs{V(x)}
\le\kappa'\frac{\sqrt{\Lambda\sb{-}^2-m^2}}
{4\Lambda\sb{+}\abs{x}},
\qquad \forall\abs{x}\ge R,
\end{equation}
with $\Lambda\sb\pm=\abs{\lambda}\pm\abs{\omega}$
from \eqref{def-Lambda-pm}.
\end{lemma}

\begin{proof}
The inequality \eqref{Eq:VSmallInfinity-l}
follows if we have
\[
(n+16\Lambda\sb{+}^2 \abs{x}^2+8\abs{x} K) \abs{V(x)}^2
\le
\kappa^2\left(\Lambda\sb{-}^2-m^2+K^2\right),
\]
valid for $\abs{x}\ge R$
(with $R$ sufficiently large)
and all $K\ge 0$.
Thus, we need
\begin{equation}\label{twn}
\abs{V(x)}^2
\le
\kappa^2
\frac{\Lambda\sb{-}^2-m^2+K^2}
{n+16\Lambda\sb{+}^2 \abs{x}^2+8\abs{x} K},
\qquad
\forall K\ge 0.
\end{equation}
For $0\le K\le 1$, one has
\begin{equation}\label{atst}
\kappa^2
\frac{\Lambda\sb{-}^2-m^2+K^2}
{n+16\Lambda\sb{+}^2 r^2+8r K}
\ge
\kappa^2
\frac{\Lambda\sb{-}^2-m^2}
{n+16\Lambda\sb{+}^2 r^2+8r}
\ge
\kappa'^2
\frac{\Lambda\sb{-}^2-m^2}
{16\Lambda\sb{+}^2 r^2},
\end{equation}
with the last inequality satisfied
as long as $r\ge R$
with $R=R(m,n,\lambda,\kappa,\kappa')$ sufficiently large.
The inequality \eqref{atst}
together with \eqref{v-small-coulomb}
yields \eqref{twn},
and hence the desired inequality \eqref{Eq:VSmallInfinity-l}
in the case $K\in[0,1]$.
At the same time,
taking the derivative in $K$,
we see that the right-hand side of \eqref{twn},
considered as a function of $K$,
is strictly monotonically increasing for $K\ge 1$
as long as
$\abs{x}\ge R=R(m,n,\lambda,\omega)$
is sufficiently large.
This completes the proof.
\end{proof}

\subsection{Absence of embedded eigenstates}

The immediate consequence of the
Carleman--Berthier--Georgescu estimates
is the following result on the absence of embedded eigenstates
for $L=D_m-\omega+V$ and $JL$,
for rather general potentials $V$:

\begin{theorem}
\label{theorem-absence}
Let $n\ge 1$,
$\omega\in[-m,m]$,
and
$V\in L^n\sb{\mathrm{loc}}(\R^n,\End(\C^N))$.
\begin{enumerate}
\item
\label{theorem-absence-i}
Let $\lambda\in\R\setminus[-m-\omega,m-\omega]$
and assume that there are $\kappa\in(0,1)$ and $R<\infty$
such that
\begin{eqnarray}\label{ass-v}
&
\norm{(n+16\lambda^2r^2+8r \tau)^{1/2}V v}^2
\le
\kappa^2
\left(
\norm{\bm\nabla v}^2
+
\norm{((\lambda+\omega)^2-m^2+\tau^2)^{1/2}v}^2
\right),
\\[1ex]
&
\forall
\tau \geq 1,
\qquad \forall v\in H^1\sb 0(\varOmega\sb R,\C^N).
\nonumber
\end{eqnarray}
Then
$
\lambda\not\in
\sigma\sb{\mathrm{p}}(D_m-\omega+V).
$
\item
\label{theorem-absence-ii}
Let $J\in\End(\C^N)$ be skew-adjoint and invertible,
such that $J^2=-I\sb{\C^N}$, $[J,D_m]=0$.
Let
\[
\lambda\in\R\setminus[-m-\abs{\omega},m+\abs{\omega}]
\]
and assume that there are $\kappa\in(0,1]$ and $R<\infty$
such that
\begin{eqnarray}\label{ass-v-jl}
&
\norm{(n+16\Lambda\sb{+}^2r^2+8r \tau)^{1/2}V v}^2
\le
\kappa^2
\left(
\norm{\bm\nabla v}^2
+
\norm{(\Lambda\sb{-}^2-m^2+\tau^2)^{1/2}v}^2
\right),
\\[1ex]
&
\forall
\tau \geq 1,
\qquad \forall v\in H^1\sb 0(\varOmega\sb R,\C^N),
\nonumber
\end{eqnarray}
where
\[
\Lambda\sb{+}=\abs{\lambda}+\abs{\omega},
\qquad
\Lambda\sb{-}=\abs{\lambda}-\abs{\omega}.
\]
Then
\[
\pm\jj\lambda
\not\in
\sigma\sb{\mathrm{p}}
\left(J(D_m-\omega+V)\right).
\]
\end{enumerate}
\end{theorem}

\begin{proof}
Let us prove
Theorem~\ref{theorem-absence}~\itref{theorem-absence-ii}.
Let
$L=D_m-\omega+V$
and assume that $\lambda\in\jj \R$, $|\lambda|> m+\abs{\omega}$,
is an embedded eigenvalue of $JL$,
with $\zeta\in L\sp 2(\R\sp n,\C^N)$ the corresponding eigenvector.

We are going to use
Theorem~\ref{theorem-Carleman-l}~\itref{theorem-Carleman-l-ii},
where we take $\varphi(r)=\tau r$ with $\tau\ge 1$.
We note that
due to \eqref{ass-v-jl} and Lemma~\ref{lemma-v-small-coulomb},
Assumption \eqref{Eq:VSmallInfinity-l}
in Theorem~\ref{theorem-Carleman-l}~\itref{theorem-Carleman-l-ii}
is satisfied.
Let $R=R(1,1,m,n,\lambda,\omega,V)$ be as in
Theorem~\ref{theorem-Carleman-l}~\itref{theorem-Carleman-l-ii}
(note that it is independent of $\tau\ge 1$).
Let $\varTheta\in C^\infty(\R\sp n)$
be a smooth radially symmetric cut-off function
with support in
the closure of $\varOmega\sb{R+1}$
and with value $1$
in $\varOmega\sb{R+2}$.
By Theorem~\ref{theorem-Carleman-l},
\begin{eqnarray}\label{vg-1}
\norm{((\abs{\lambda}-\abs{\omega})^2-m^2+\tau^2)^{1/2}
e\sp{\tau r}\varTheta\zeta}
\le
\norm{(n+16(\abs{\lambda}+\abs{\omega})^2 r^2+8\tau r)^{1/2}
e\sp{\tau r}(JL-\lambda)\varTheta \zeta}.
\end{eqnarray}
Since $JL\zeta=\lambda\zeta$, we have
$(JL-\lambda)\varTheta\zeta=[JL,\varTheta]\zeta
=J(D_0\Theta)\zeta$.
By \eqref{vg-1},
\[
\forall \tau\geq 1,
\qquad
\norm{e\sp{\tau r}\varTheta \zeta}
\le
\Norm{
\left(\frac{n}{\tau^2}
+16(\abs{\lambda}+\abs{\omega})^2\frac{r^2}{\tau^2}
+8\frac{r}{\tau}\right)\sp{1/2}
e\sp{\tau r}
(D_0\varTheta)\zeta
}.
\]
Taking into account that
$D_0\varTheta=-\jj\bm\alpha\cdot\bm\nabla\varTheta$
is identically zero outside of the ball
$\mathbb{B}^n\sb{R+2}$, we conclude that
\[
\forall \tau\geq 1,
\qquad
\norm{e\sp{\tau r}\zeta}\sb{L\sp 2(\varOmega\sb{R+2},\C^N)}
\le
C
e\sp{(R+2)\tau}
\norm{r\nabla\varTheta}\sb{L^\infty}
\|\zeta\|\sb{L\sp 2(\mathbb{B}^n\sb{R+2},\C^N)},
\]
with $C<\infty$
independent of $\tau\ge 1$.
Since $\tau$ could be
arbitrarily large,
we conclude that $\supp \zeta\cap\varOmega\sb{R+2}=\varnothing$.
The unique continuation principle, Lemma~\ref{lemma-unique-continuation},
ensures that $\zeta\equiv 0$,
contradicting our assumption that there were an embedded eigenvalue
$\lambda\in\jj \R$, $\abs{\lambda}>m+\abs{\omega}$.

Theorem~\ref{theorem-absence}~\itref{theorem-absence-i}
is proved similarly,
by using
Theorem~\ref{theorem-Carleman}~\itref{theorem-Carleman-ii}
instead of
Theorem~\ref{theorem-Carleman-l}~\itref{theorem-Carleman-l-ii}.
\end{proof}

A consequence
of Theorem~\ref{theorem-absence}~\itref{theorem-absence-i}
and Lemma~\ref{lemma-v-small-coulomb}
is the absence of solitary waves
with $\omega\in\R$, $\abs{\omega}>m$.
There are different possible formulations such as the following:
\begin{theorem}
\label{theorem-absence-2}
Let $n\ge 1$.
For $\omega\in\R\backslash[-m,m]$, there are no solutions
$\phi\sb\omega(x)$ to \eqref{nld-stationary}
such that
\[
\phi\sb\omega
\in
H^1(\R^n,\C^N)
\]
and such that
$F(x):=f(\phi\sb\omega\sp\ast(x)\beta\phi\sb\omega(x))$
satisfies
$F\in L^n\sb{\mathrm{loc}}(\R^n)$,
$F\ne 0$ almost everywhere in $x\in\R^n$,
and
\[
\abs{F(x)}
\le \kappa\frac{\sqrt{\omega^2-m^2}}{4\abs{\omega}\abs{x}}
\qquad
\mbox{for $x$ almost everywhere in}
\ \varOmega\sb R,
\]
with some
$R<\infty$ and $\kappa\in(0,1)$.
\end{theorem}

\subsection{Exponential decay of eigenstates}

Here we formulate our results
on the exponential decay of solitary wave solutions
to the nonlinear Dirac equation
and of eigenfunctions to linear Dirac equation
(which could be a linearization at a solitary wave;
cf. \eqref{nld-linear}).

\begin{theorem}[Exponential decay of solitary waves]
\label{theorem-properties}$\null$
Let $n\ge 1$
and assume that
$f$ is measurable and $\lim\sb 0 f=0$.
Let $\omega\in(-m,m)$
and let
\[
\phi\sb\omega
\in
L^2(\R^n,\C^N)\cap L^\infty(\R^n,\C^N)
\]
be a solution to \eqref{nld-stationary}
which satisfies
\begin{equation}\label{becomes-small}
\lim\sb{R\to\infty}\norm{\phi\sb\omega}\sb{L^\infty(\varOmega_R,\C^N)}=0,
\end{equation}
where
$\varOmega_R=\{x\in\R^n\sothat\abs{x}>R\}$.
Then for any $\mu<\sqrt{m\sp 2-\omega\sp 2}$
one has
$e\sp{\mu\langle r\rangle}\phi\sb\omega\in H\sp1(\R\sp n,\C^N)$.
\end{theorem}

Theorem~\ref{theorem-properties}
is proved in Section~\ref{sect-waves}.

\begin{theorem}[Exponential decay of eigenfunctions]
\label{theorem-embed}
Let $n\ge 1$
and assume that
$V\in L^\infty(\R^n,\End(\C^N))$.
\begin{enumerate}
\item
Assume that for any $\epsilon>0$ there is $R>0$ such that
\begin{equation}\label{v-small-0-1}
\norm{V v}\leq \epsilon \norm{v}_{H^1},
\qquad \forall v\in H\sp 1\sb{\mathrm{comp}}(\varOmega\sb R,\C^{N}).
\end{equation}
\label{theorem-embed-i}
Assume that
\[
\lambda\in\sigma\sb{\mathrm{p}}(D_m+V)\cap(-m,m).
\]
Then the corresponding eigenfunctions
are exponentially decaying:
if $\zeta$ is an eigenfunction corresponding to $\lambda$,
then for any
$\mu<\sqrt{m\sp 2-\omega\sp 2}$
one has
\[
e\sp{\mu\langle r\rangle}\zeta
\in H\sp1(\R\sp n,\C^{N}).
\]
\item
Let $J\in\End(\C^N)$ be skew-adjoint and invertible,
such that $J^2=-I\sb{\C^N}$, $[J,D\sb m]=0$.
Let $\omega\in[-m,m]$
and assume that
\[
\lambda\in\sigma\sb{\mathrm{p}}(JL(\omega))\cap \jj \R.
\]
\begin{enumerate}
\item
\label{theorem-embed-iia}
If $\abs{\lambda}<m-\abs{\omega}$
and for any $\epsilon>0$ there is $R>0$ such that
\begin{equation}\label{v-small-0-2}
\norm{V v}\leq \epsilon \norm{v}_{H^1},
\qquad \forall v\in H\sp 1\sb{\mathrm{comp}}(\varOmega\sb R,\C^{N}),
\end{equation}
then  the corresponding eigenfunctions
are exponentially decaying.
More precisely,
if $\zeta$ is an eigenfunction corresponding to $\lambda$,
then for any
\[
\mu<\sqrt{m^2-(|\lambda|+|\omega|)^2}
\]
one has
\[
e\sp{\mu\langle r\rangle}\zeta
\in H\sp1(\R\sp n,\C^{N}).
\]
\item
\label{theorem-embed-iib}
If $m-|\omega|<\abs{\lambda}<m+\abs{\omega}$
and for any $\epsilon>0$ there is $R>0$ such that
\begin{equation}\label{v-small-0-3}
\norm{\langle r\rangle V v}\leq \epsilon \norm{v}_{H^1},
\qquad \forall v\in H\sp 1\sb{\mathrm{comp}}(\varOmega\sb R,\C^{N}),
\end{equation}
then  the corresponding eigenfunctions
are exponentially decaying.
More precisely,
if $\zeta$ is an eigenfunction corresponding to $\lambda$,
then for any
\[
\mu<\sqrt{m^2-(|\lambda|-|\omega|)^2}
\]
one has
\[
e\sp{\mu\langle r\rangle}\zeta
\in H\sp1(\R\sp n,\C^{N}).
\]
\end{enumerate}
\end{enumerate}
\end{theorem}

\begin{remark}
In the above theorem, the potential $V$ is not necessarily self-adjoint.
\end{remark}

We prove
Theorem~\ref{theorem-embed}~\itref{theorem-embed-iib}
in Section~\ref{sect-embed};
the proofs of
Theorem~\ref{theorem-embed}~\itref{theorem-embed-i}
and Theorem~\ref{theorem-embed}~\itref{theorem-embed-iia}
are slightly shorter and readily follow
along the same lines.

\subsection{Bifurcations
of eigenvalues from the essential spectrum}

The next question we consider is
how the eigenvalues with nonzero real part
could arise.
As in Theorem~\ref{theorem-embed},
we formulate the results for general Dirac-type operators
of the form $JL(\omega)$, with
\[
L(\omega)=D_m-\omega+V(x,\omega),
\]
having in mind the linearization
\eqref{nld-linear}
of the nonlinear Dirac equation
at a solitary wave.

\begin{theorem}[Bifurcation of point eigenvalues]
\label{theorem-b}
Let $n\ge 1$.
Let $J\in\End(\C^N)$ be skew-adjoint and invertible,
such that $J^2=-I\sb{\C^N}$, $[J,D\sb m]=0$.
Let
$(\omega\sb j)\sb{j\in\N}$,
$\omega\sb j\in[-m,m]$,
be a sequence with
$\lim\sb{j\to\infty}\omega\sb j
=\omega\sb{0}\in[-m,m]$,
and assume that $V$ is hermitian
and that
there is $\varepsilon>0$ such that
\begin{equation}\label{lec-0}
\begin{cases}
\norm{\langle r\rangle\sp{1+\varepsilon}
V(\omega\sb{0})}\sb{L^\infty(\R\sp n,\End(\C^N))}<\infty,
\\[2ex]
\lim\sb{j\to\infty}
\,\norm{
\langle r\rangle\sp{1+\varepsilon}\left(V(\omega\sb j)-V(\omega\sb{0})\right)
}\sb{L^\infty(\R\sp n,\End(\C^N))}=0,
\end{cases}
\end{equation}
where $\norm{\langle r\rangle^{1+\varepsilon}V(\omega)}
=\norm{\langle\cdot\rangle^{1+\varepsilon}V(\cdot,\omega)}$.
Let
$\lambda\sb j\in\sigma\sb{\mathrm{p}}(JL(\omega\sb j))$,
$j\in\N$
be a sequence
such that
\[
\Re\lambda\sb j\ne 0\quad\forall j\in\N,
\qquad
\lambda\sb j
\mathop{\longrightarrow}\limits\sb{j\to\infty}
\lambda\sb{0}\in\jj\R,
\qquad
\lambda\sb{0}\ne\pm\jj(m+\abs{\omega\sb{0}}).
\]
If $\omega\sb{0}= \pm m$,
additionally assume that
\begin{equation}\label{ends-do-not-meet-0}
\lambda\sb{0}\ne 0.
\end{equation}
Then
\[
\lambda\sb{0}\in\sigma\sb{\mathrm{p}}(JL(\omega\sb{0})).
\]
\end{theorem}
This theorem will be proved in
Section~\ref{sect-b}.

\begin{remark}
The conclusion of
Theorem~\ref{theorem-b}
is trivial
when $V$ depends continuously on $\omega$
and if $\lambda\sb{0}\in\jj \R$
with $\abs{\lambda\sb{0}}<m-\abs{\omega\sb{0}}$,
so that $\lambda\sb{0}$ is not in the essential spectrum,
and the inclusion
$\lambda\sb{0}\in\sigma\sb{\mathrm{p}}(JL(\omega\sb{0}))$
follows from the continuous dependence
of isolated eigenvalues on a parameter $\omega$.
\end{remark}

\begin{remark}\label{remark-ts}
If $JL$ corresponds to the linearization
at solitary waves,
then, due to the exponential decay of $\phi\sb\omega$
(cf. Theorem~\ref{theorem-properties}),
the condition \eqref{lec-0}
is trivially satisfied
for any $\omega\sb{0}\in(-m,m)$
(and with any  $\varepsilon>0$).
\end{remark}


\begin{remark}\label{remark-nop}
Combining the results of Theorem~\ref{theorem-b}
with Theorem~\ref{theorem-absence}
on the absence of embedded eigenvalues,
we conclude that
for the linearizations at solitary waves
the bifurcations of point eigenvalues
from the continuous spectrum beyond the embedded thresholds
$\pm\jj(m+\abs{\omega})$ are not possible.
\end{remark}

\bigskip

We separately consider the case when $\omega\sb j\to\omega\sb{0}=m$
which corresponds to the nonrelativistic limit.
This case will be further investigated
in a subsequent work.

\begin{theorem}[Bifurcations
from the spectrum of the free Dirac operator]
\label{theorem-vb}
\quad

Let $n\ge 1$.
Let $J\in\End(\C^N)$ be skew-adjoint and invertible,
such that $J^2=-I\sb{\C^N}$, $[J,D\sb m]=0$.
Let $V(\omega)\in L^\infty(\R^n,\End(\C^N))$ for $\omega\in[-m,m]$,
and let
$(\omega\sb j)\sb{j\in\N}$,
$\omega\sb j\in[-m,m]$,
$\omega\sb j\to\omega\sb{0}=\pm m$,
and assume that there is $\varepsilon>0$ such that
\[
\lim\sb{j\to\infty}
\norm{\langle r\rangle\sp{1+\varepsilon}V(\omega\sb j)}
\sb{L^\infty(\R\sp n,\End(\C^N))}=0.
\]
Let
$\lambda\sb j\in\sigma\sb{\mathrm{p}}(JL(\omega\sb j))$,
and let $\lambda\sb{0}\in\jj\R\cup\{\infty\}$
be an accumulation point of the sequence $(\lambda\sb j)\sb{j\in\N}$.
Then
$\lambda\sb{0}\in\{0;\pm 2 m \jj \}$.
In particular,
$\lambda\sb{0}\ne\infty$.
\end{theorem}

\begin{remark}
In this theorem, we do not need to assume that
$\Re\lambda\sb j\ne 0$, $\forall j\in\N$
and that $V$ is hermitian.
\end{remark}

This theorem will be proved in
Section~\ref{sect-vb}.

\section{Carleman--Berthier--Georgescu estimates}
\label{sect-carleman}

The main ingredient of our proofs
is the version of the Carleman estimates
for the Dirac operator
due to
Berthier and Georgescu~\cite[Theorem 5]{MR880983}.
The following result
generalizes the Carleman estimates
for the Dirac operator in $\R^3$
to any dimension.

For $\lambda\in\R\setminus[-m,m]$
and
$\varphi\in\mathscr{C}\sb\lambda(M,\calN,\rho,\nu)$
with some $M\geq 1,\,\calN\geq 1,\,\rho\ge 1$, and $\nu>0$,
we denote
\begin{equation}\label{def-mu}
\mu(r)
=
2
\Big(
n+16\lambda^2 r^2+8r\varphi'(r)
\Big)^{1/2},
\qquad
r\ge\rho;
\end{equation}
\begin{equation}\label{def-g}
\gamma(r)=
\left(
\lambda^2-m^2+\varphi'(r)^2+2r\varphi'(r)\varphi''(r)\right)^{1/2},
\qquad
r\ge\rho.
\end{equation}

\begin{theorem}[Carleman--Berthier--Georgescu estimates]
\label{theorem-Carleman}
Let $n\ge 1$,
$m>0$,
$\lambda\in\R\setminus[-m,m]$.

\begin{enumerate}
\item
\label{theorem-Carleman-i}
Assume that
$
\varphi\in\mathscr{C}\sb\lambda(M,\calN,\rho,\nu)
$
with some $M\geq 1,\,\calN\geq 1,\,\rho\ge 1$, and $\nu>0$.

Let
$V\in\mathscr{B}\left(H\sp 1\sb{\mathrm{comp}}(\varOmega\sb\rho,\C^N),\,L\sp 2(\varOmega\sb\rho,\C^N)\right)$ be a multiplication operator,
and assume that there are $\kappa\in[0,1)$
and $R_1=R_1(\varphi,V)\ge\rho$
such that
\begin{equation}\label{ineq-0}
\norm{\mu V v}
\le
\kappa
\left(
\norm{\bm\nabla v}^2
+
\norm{\gamma v}^2
\right)^{1/2},
\qquad \forall v\in H^1\sb 0(\varOmega_{R_1},\C^N).
\end{equation}
Then there is
$R=R(\varphi,V)\ge R_1$
such that
for any $u\in H\sp 1(\R\sp n,\C^N)$
with $\supp u\subset\varOmega\sb{R}$
and
\[
\mu e\sp\varphi (D\sb m+V-\lambda)u \in L\sp 2(\R\sp n,\C^N),
\]
the functions
$\gamma e\sp{\varphi}u$,
\ $\bm\nabla(e\sp{\varphi}u)$,
\ $(r\varphi')\sp{1/2}\partial\sb r(e\sp{\varphi}u)$
are in $L\sp 2(\R\sp n,\C^N)$, and moreover
\begin{eqnarray}
\label{car}
\norm{\bm\nabla(e^\varphi u)}^2
+
2\norm{(r\varphi')^{1/2}\p\sb r(e^\varphi u)}^2
+
\norm{
\gamma
e\sp{\varphi}u}\sp 2
\le
\frac{1}{(1-\kappa)^2}
\norm{
\mu
e\sp\varphi (D\sb m+V-\lambda)u}\sp 2.
\end{eqnarray}
\item
\label{theorem-Carleman-ii}
Assume that there are $\kappa\in[0,1)$
and
$R_2=R_2(V)$
such that
\begin{eqnarray}\label{Eq:VSmallInfinity}
&
\norm{(n+16\lambda^2 r^2+8r f(r))^{1/2} V v}^2
\le
\kappa^2
\left(
\norm{\bm\nabla v}^2
+
\norm{(\lambda^2-m^2+f(r)^2
)^{1/2}v}^2
\right),
\\[1ex]
&
\forall
f\in C(\R),
\quad
f\ge 0,
\quad
\sup\sb{r>0}\frac{f(r)}{\langle r\rangle}<\infty;
\qquad \forall v\in H^1\sb 0(\varOmega\sb{R_2},\C^N).
\nonumber
\end{eqnarray}
Let $M\geq 1,\,\calN\geq 1,\,\rho\ge 1$, and $\nu>0$.
Then there is
\[
R=R(M,\rho,m,n,\lambda,V)\ge\max(\rho,R_2(V))
\]
such that for any
$\varphi\in\mathscr{C}\sb\lambda(M,\calN,\rho,\nu)$
with
$\varphi''(r)\ge 0$, $r\ge\rho$,
and for any
\[
u\in H\sp 1\sb 0(\varOmega\sb{R},\C^N)
\]
which satisfies
\[
\mu e\sp\varphi (D\sb m+V-\lambda)u \in L\sp 2(\R\sp n,\C^N),
\]
the functions
$\gamma e\sp{\varphi}u$,
\ $\bm\nabla(e\sp{\varphi}u)$,
\ $(r\varphi')\sp{1/2}\partial\sb r(e\sp{\varphi}u)$
are in $L\sp 2(\R\sp n,\C^N)$, and moreover
\begin{eqnarray}\label{carV}
\norm{\bm\nabla(e^\varphi u)}^2
+
2\norm{(r\varphi')^{1/2}\p\sb r(e^\varphi u)}^2
+
\norm{
\gamma e\sp{\varphi}u}\sp 2
\le
\frac{1}{(1-\kappa)^2}
\norm{
\mu
e\sp\varphi (D\sb m+V-\lambda)u}\sp 2.
\end{eqnarray}
\end{enumerate}
Above,
$\varphi$, $\mu$, and $\gamma$
(see \eqref{def-mu}
and \eqref{def-g})
are considered as functions of $r=\abs{x}$.
\end{theorem}

\begin{proof}
We choose to give a detailed proof
which closely follows the argument in~\cite{MR880983}.
We start with several lemmata.
Let
$D\sb m=-\jj \bm\alpha\cdot\bm\nabla+\beta m$,
$\varphi\in C\sp 2(\R\sb{+})$,
and denote
\begin{equation}\label{def-dm-varphi}
D\sb m\sp\varphi
=e\sp\varphi\circ D\sb m \circ e\sp{-\varphi}
=D\sb m+ \jj \bm\alpha\cdot\bm\nabla\varphi.
\end{equation}
The starting point of the analysis
is the following lemma
which helps to establish the exponential decay of
eigenvectors associated to eigenvalues in the gap.

\begin{lemma}[Lemma 3,~\cite{MR880983}]
\label{Lem:LemmaBG3}
Let $\Omega\subset \R\sp n$ be an open set and  $\varphi: \Omega\to
\R$ a $C^1$ map.
For $v\in H\sp 1\sb{\mathrm{comp}}(\Omega,\C^N)$,
\begin{equation}\label{g14}
\Re
\langle
(D\sb m-\jj \bm\alpha\cdot\bm\nabla\varphi+\lambda)v,
(D\sb m\sp\varphi-\lambda)v\rangle
=\norm{\bm\nabla v}\sp 2+\langle v,[m\sp 2-\lambda\sp 2-(\bm\nabla\varphi)\sp 2]v\rangle.
\end{equation}
\end{lemma}

\begin{proof}
Taking into account that $\bm\nabla\varphi$ is continuous,
due to density of smooth functions with compact support
in $H\sp 1\sb{\mathrm{comp}}(\Omega,\C^N)$,
it is enough to give the proof assuming that $v\in H\sb{\mathrm{comp}}^\infty(\Omega,\C^N)$.

The statement of the lemma
is a consequence
of the following computation performed in~\cite{MR880983}:
\begin{eqnarray}\label{ltt}
&(D\sb m\sp\varphi+\lambda)
(D\sb m\sp\varphi-\lambda)
=(D\sb m\sp\varphi)\sp 2-\lambda\sp 2
=e\sp\varphi\circ (D\sb m\sp 2-\lambda\sp 2)\circ e\sp{-\varphi}
=e\sp\varphi\circ(-\Delta+m\sp 2-\lambda\sp 2)\circ e\sp{-\varphi}
\nonumber
\\[1ex]
&
=-\Delta+m\sp 2-\lambda\sp 2
-(\bm\nabla\varphi)\sp 2
+\bm\nabla\varphi\cdot\bm\nabla
+\bm\nabla\cdot\bm\nabla\varphi,
\end{eqnarray}
where the last term is understood as
the multiplication by $\bm\nabla\varphi$
and then taking the divergence.
In the real part of the corresponding quadratic form
the last two terms from the right-hand side of \eqref{ltt}
cancel,
while the real part of the left-hand side
turns into that of \eqref{g14}.
\end{proof}

For brevity, we adopt the following notations from~\cite{MR880983}:
\[
\hat{X}=x\cdot\bm\nabla,
\qquad
\scrD=\frac 1 2\{x,-\jj \bm\nabla\}
=-\jj \hat{X}-\frac{\jj n}{2}.
\]
Notice that $\scrD$ is the generator of dilations and thus
in the sense of quadratic forms
on $H\sp1\sb{\mathrm{comp}}(\Omega,\C^N)$,
\begin{equation}\label{dilations-generator}
[\scrD,D\sb m]=
[\scrD,D\sb 0]=
[-\jj x\cdot\bm\nabla,-\jj \bm\alpha\cdot\bm\nabla]=
[\bm\alpha\cdot\bm\nabla,x\cdot\bm\nabla]
=\bm\alpha\cdot\bm\nabla=\jj D\sb 0.
\end{equation}

In order to analyze the eigenvectors associated to embedded eigenvalues, it
will be convenient to subtract the identity
\eqref{ltt}
from another one (which also
involves $(D\sb m\sp\varphi-\lambda)$ in the left hand side) that controls the
$\dot{H}\sp 1$ norm.
Starting with $D\sb 0(D\sb m\sp\varphi-\lambda)$ and trying
to eliminate inconvenient terms (with a factor $\lambda$ for instance), Berthier and Georgescu~\cite{MR880983}
have the following lemma (Lemma 4,~\cite{MR880983}),
which we rewrite for arbitrary dimension $n\ge 1$.

\begin{lemma}
\label{lemma-lemma4}
Let
$\varOmega\subset\R^n$ be an open set
and let
$\varphi\in C\sp 2(\varOmega)$.
Then for any $u\in H\sp 1\sb{\mathrm{comp}}(\varOmega,\C^N)$,
\begin{eqnarray}\label{g16}
&&
2\Re\langle
(D\sb 0+2 \jj \lambda\scrD
+\{\scrD,\bm\alpha\cdot\bm\nabla\varphi\})v,
(D\sb m\sp\varphi-\lambda)v\rangle
\\[1ex]
&&
=
2\norm{\bm\nabla v}\sp 2
+4\Re
\langle \hat{X} v,(\bm\nabla\varphi)\cdot\bm\nabla v\rangle
+2\Re\langle \hat{X} v,\varDelta\varphi\,v\rangle
+\langle v,\big(\hat{X}(\bm\nabla\varphi)\sp 2\big)v\rangle.
\nonumber
\end{eqnarray}
\end{lemma}

\begin{proof}
We present the proof from~\cite{MR880983}
stripped off the external fields.
Again, since $\bm\nabla \varphi$ and $\Delta \varphi$ are continuous,
and since smooth functions with compact support
are dense in $H\sp 1\sb{\mathrm{comp}}(\Omega,\C^N)$,
the computations below, made in the former case,
will provide the proof for the latter.
First, using \eqref{dilations-generator},
we get
\begin{eqnarray}\label{g20}
4\Im\langle\scrD v,(D\sb m\sp\varphi-\lambda)v\rangle
&=&
\frac{2}{ \jj }
\langle v,[\scrD,D\sb 0]v\rangle
+4\Im\langle\scrD v,(\beta m+ \jj \bm\alpha\cdot\bm{F}-\lambda)v\rangle
\nonumber
\\
\qquad
&=&
2\langle v,D\sb 0 v\rangle
+4\Re\langle \scrD v,\bm\alpha\cdot\bm{F} v\rangle,
\end{eqnarray}
where $\bm{F}=\bm\nabla\varphi$.
Using \eqref{def-dm-varphi}
and the identity
\begin{equation}\label{ts}
(\bm\alpha\cdot\bm{S})(\bm\alpha\cdot\bm{T})
=\bm{S}\cdot\bm{T}
+ \jj \Sigma(\bm{S},\bm{T}),
\qquad
\bm{S},\ \bm{T}\in\C\sp n,
\end{equation}
where
$\Sigma(\bm{S},\bm{T})=\Sigma\sb{j k}S\sb j T\sb k$,
with the matrices
$\Sigma\sb{j k}=\frac{1}{2 \jj }
[\alpha\sp j,\alpha\sp k]$
hermitian for each $j,\,k$,
we have
\[
D\sb 0(D\sb m\sp\varphi-\lambda)
=-\varDelta
+D\sb 0 \beta m+
(\bm\alpha\cdot\bm\nabla)
\circ
(\bm\alpha\cdot\bm{F})-\lambda D\sb 0
\]
\[
\qquad
=-\varDelta+m D\sb 0\beta
+ \bm{F}\cdot \nabla
+\varDelta\varphi
+ \jj \Sigma(\bm\nabla,\bm{F})-\lambda D\sb 0.
\]
But
$\Re\langle v,D\sb 0\beta v\rangle=0$
(due to $\{D\sb 0,\beta\}=0$),
$2\Re\langle v,\bm{F}\cdot\bm\nabla v\rangle
=-\langle v,\varDelta\varphi\, v\rangle
$,
and
\[
\Sigma(\bm\nabla,\bm{F})
=
\Sigma\sb{j k}
\p\sb j\circ \varphi\sb k
=
\Sigma\sb{j k}
\varphi\sb{j k}
+
\Sigma\sb{j k}
\varphi\sb k\p\sb j
=
-\Sigma\sb{k j}
\varphi\sb k\p\sb j
=-\Sigma(\bm{F},\bm\nabla),
\]
hence
\begin{align}\label{g21}
2\Re
\langle D\sb 0 v,(D\sb m\sp\varphi-\lambda)v\rangle
=2\norm{\bm\nabla v}\sp 2+
\langle v,\varDelta\varphi\,v\rangle
-2 \jj \langle v,\Sigma(\bm{F},\bm\nabla)v\rangle
-2\lambda\langle v,D\sb 0 v\rangle.
\end{align}
The last term is inconvenient as, due to the factor $\lambda$, it cannot be controlled properly and uniformly in $\lambda$.
Adding \eqref{g20}
(multiplied by $\lambda$)
to \eqref{g21}, to get rid of
$2\lambda\langle v,D\sb 0 v\rangle$,
we obtain
\begin{eqnarray}\label{g22}
&&
2\Re
\langle
D\sb 0 v,(D\sb m\sp\varphi-\lambda)v\rangle
+4\lambda
\Im\langle\scrD v,(D\sb m\sp\varphi-\lambda)v\rangle
\nonumber\\[1ex]
&&
=
2\norm{\bm\nabla v}\sp 2
+
\langle v,
\big(
\varDelta\varphi
-2 \jj
\Sigma(\bm{F},\bm\nabla)
\big)
v\rangle
+
4\lambda\Re\langle\scrD v,\bm\alpha\cdot\bm{F} v\rangle;
\end{eqnarray}
see \eqref{dilations-generator}.
Now we eliminate the inconvenient term $4\lambda\Re\langle\scrD
v,\bm\alpha\cdot\bm{F} v\rangle$. Recalling that
$[\scrD,D\sb 0]=\jj D\sb 0$, we derive the identity
\begin{align*}
\{\scrD,\bm\alpha\cdot\bm{F}\}D\sb 0
&=
\{\scrD,\bm\alpha\cdot\bm{F} D\sb 0\}
+\bm\alpha\cdot\bm{F}[\scrD,D\sb 0]
\\
&
=\{\scrD,\bm\alpha\cdot\bm{F} D\sb 0\}
+ \jj \bm\alpha\cdot\bm{F} D\sb 0
=\{\scrD+\frac{ \jj }{2},\bm\alpha\cdot\bm{F} D\sb 0\}
\\
&
=\{\scrD+\frac{ \jj }{2},(\bm\alpha\cdot\bm{F})(-\jj \bm\alpha\cdot\bm\nabla)\}
=-\jj \{\scrD+\frac{ \jj }{2},\bm{F}\cdot\bm\nabla\}
+\{\scrD+\frac{ \jj }{2},\Sigma(\bm{F},\bm\nabla)\},
\end{align*}
where in the last line we used \eqref{ts}.
The above relation leads to
\begin{eqnarray}\label{g23}
2\Re
\{\scrD,\bm\alpha\cdot\bm{F}\}D\sb 0
=
\Im\{2\scrD+ \jj ,\bm{F}\cdot\bm\nabla\}
+\Re\{2\scrD+ \jj ,\Sigma(\bm{F},\bm\nabla)\}
=
\Im\{2\scrD+ \jj ,\bm{F}\cdot\bm\nabla\}
+2 \jj \Sigma(\bm{F},\bm\nabla),
\end{eqnarray}
where we used the identity
$\Re\{\scrD,\Sigma(\bm{F},\bm\nabla)\}=0$.
The first term in the right-hand side of \eqref{g23}
can be written as
\begin{align}\label{g24}
\Im \{2\scrD+ \jj ,\bm{F}\cdot\bm\nabla\}
&
=
-\jj
\big(
\scrD\bm{F}\cdot\bm\nabla
+\bm{F}\circ\bm\nabla\scrD
+
\bm\nabla
\circ
\bm{F}\scrD
+\scrD\bm{F}\cdot\bm\nabla
+ \jj \bm{F}\cdot\bm\nabla
-\jj \bm\nabla\circ\bm{F}
\big)
\nonumber
\\[1ex]
&
=-\jj \{\scrD,\{\bm{F},\bm\nabla\}\}-\varDelta\varphi
=2\Im(\scrD\{\bm{F},\bm\nabla\})-\varDelta\varphi
=-2\Re (\bm\nabla\circ x\{\bm{F},\bm\nabla\})-\varDelta\varphi
\nonumber
\\[1ex]
&
=-4\Re(\bm\nabla\circ x\,\bm{F}\cdot\bm\nabla)
-2\Re(\bm\nabla\circ x\varDelta\varphi)-\varDelta\varphi.
\end{align}
    From \eqref{g23} and \eqref{g24} we obtain
\[
2\Re\{\scrD,\bm\alpha\nabla\varphi\}D\sb 0
=
-4\Re(\bm\nabla\circ x\,\bm{F}\cdot\bm\nabla)
-2\Re(\bm\nabla\circ x\varDelta\varphi)-\varDelta\varphi
+2 \jj \Sigma(\bm{F},\bm\nabla),
\]
\begin{align}
\label{g25}
&
2\Re\langle
\{\scrD,\bm\alpha\cdot\bm{F}\}v,
(D\sb m\sp\varphi-\lambda)v\rangle
\nonumber
\\[1ex]
&=2\Re\langle
\{\scrD,\bm\alpha\cdot\bm{F}\}v,D\sb 0 v\rangle
+2\Re\langle
\{\scrD,\bm\alpha\cdot\bm{F}\}v,
\jj \bm\alpha\cdot\bm{F} v\rangle
-2\lambda\Re\langle
\{\scrD,\bm\alpha\cdot\bm{F}\}v,v\rangle
\\[1ex]
&=4\Re\langle \hat{X} v,\bm{F}\cdot\bm\nabla v\rangle
+2\Re\langle \hat{X} v,\varDelta\varphi\,v\rangle
-\langle v,(\varDelta\varphi
-2 \jj \Sigma(\bm{F},\bm\nabla))v\rangle
+\langle v,\hat{X}(\bm{F})\sp 2v\rangle
-4\lambda\Re\langle\scrD v,\bm\alpha\cdot\bm{F} v\rangle.
\nonumber
\end{align}
Above, we used the identity
\[
2\Re
\langle
\{\scrD,\bm\alpha\cdot\bm{F}\}v, \jj \bm\alpha\cdot\bm{F} v\rangle
=
\langle
v,
\big(
\{\scrD,\bm\alpha\cdot\bm{F}\}
\jj \bm\alpha\cdot\bm{F}
-
\jj \bm\alpha\cdot\bm{F}
\{\scrD,\bm\alpha\cdot\bm{F}\}
\big)
v\rangle
=
\langle v,\hat X (\bm{F})\sp 2 v\rangle.
\]
Adding
\eqref{g22}
and \eqref{g25}
yields \eqref{g16}.
\end{proof}

The following lemma
parallels~\cite[Lemma 6]{MR880983} with explicit constants.

\begin{lemma}\label{lemma-bg-inequality}
Let $n\in\N$, $\rho>0$.
Let $\varphi\in C\sp 2([\rho,+\infty))$
with $\varphi'>0$,
and let us define
$Z\in C([\rho,+\infty))$
by
\begin{eqnarray}\label{def-Z}
Z(r)
&=&
2\left(
\lambda\sp 2-m\sp 2+\varphi'\sp 2+2 r\varphi'\varphi''
-(n-1)\varphi''
-\frac{(n-1)\sp 2\varphi'}{r}-\frac{r\varphi''\sp 2}{\varphi'}
\right)
\nonumber
\\
&&
\quad
-4\left(\frac{(2n-1)(\abs{\lambda}+\varphi')+m+2r\abs{\varphi''}}
{\mu}\right)^2.
\end{eqnarray}
Then for any $v\in H\sp 1\sb{\mathrm{comp}}(\R\sp n,\C^N)$
one has
\begin{equation}
\label{euw}
\norm{\bm\nabla v}\sp 2
+
2\norm{(r\varphi')\sp{1/2} \partial\sb r v}\sp 2
+
\big\langle v, Z v\big\rangle
\le
\norm{\mu(D\sb m\sp\varphi-\lambda)v}\sp 2,
\end{equation}
and for any
$u\in H\sp 1\sb{\mathrm{comp}}(\R\sp n,\C^N)$
one has
\begin{equation}
\label{euw-u}
\norm{\bm\nabla(e^\varphi u)}^2
+
2\norm{(r\varphi')\sp{1/2} \partial\sb r(e^\varphi u)}\sp 2
+
\big\langle e^\varphi u, Z e^\varphi u\big\rangle
\le\norm{\mu e^\varphi(D\sb m-\lambda)u}\sp 2.
\end{equation}
Above,
$\varphi$, $\mu$, and $Z$
(see \eqref{def-mu}, \eqref{def-Z})
are considered as functions of $r=\abs{x}$.
\end{lemma}

\begin{proof}
Denote
$\hat\alpha=r\sp{-1}\bm\alpha\cdot x$,
where $r=\abs{x}$.
We subtract \eqref{g14} from \eqref{g16}
with the aid of the identity
\[
\{\scrD,\bm\alpha\cdot\bm{F}\}
=
\{-\jj \hat{X}-\frac{ \jj n}{2},\hat\alpha\varphi'\}
=-2 \jj \hat\alpha\varphi'\hat{X}
-\jj \hat\alpha r\varphi''
-\jj n\hat\alpha\varphi',
\]
arriving at
\begin{align}\label{g41}
&
2\Re\Big\langle
\Big[
D\sb 0+2 \jj \lambda\scrD
+\{\scrD,\bm\alpha\cdot\bm{F}\}
-\frac 1 2(D\sb m-\jj \bm\alpha\cdot\bm{F}+\lambda)
\Big]
v,\,\,(D\sb m\sp\varphi-\lambda)v\Big\rangle
\nonumber\\
&
=2\Re\Big\langle\Big[
-\frac{ \jj }{2}\bm\alpha\cdot\bm\nabla
+2(\lambda-\jj \hat\alpha\varphi')\hat{X}
+(n-\frac{1}{2})\lambda
-\frac{m}{2}\beta
-\jj (n-\frac 1 2)\hat\alpha\varphi'
-\jj \hat\alpha r\varphi''
\Big]v,\,\,(D\sb m\sp\varphi-\lambda)v\Big\rangle
\nonumber\\
&
=
\norm{\bm\nabla v}\sp 2
+4\Re\langle
\hat{X} v,\frac{\varphi'}{r}\hat{X} v\rangle
+2\Re\langle \hat{X} v,\varDelta\varphi\,v\rangle
+\Big\langle v,\Big[
\lambda\sp 2-m\sp 2+\varphi'\sp 2+2 r\varphi'\varphi''
\Big]v\Big\rangle.
\end{align}
Since
\[
2\Re\Big\langle \hat{X} v,\frac{\varphi'}{r}v\Big\rangle
=
\Big\langle \hat{X} v,\frac{\varphi'}{r}v\Big\rangle
+
\Big\langle v,\frac{\varphi'}{r}\hat{X} v\Big\rangle
=
\Big\langle v,
\Big(-\hat X\circ\frac{\varphi'}{r}
-n\frac{\varphi'}{r}+\frac{\varphi'}{r}\hat{X}\Big)
v\Big\rangle
=
-\Big\langle v,\Big[\varphi''+(n-1)\frac{\varphi'}{r}\Big]v\Big\rangle,
\]
which is valid for
$v\in H\sp 1\sb{\mathrm{comp}}(\varOmega\sb\rho,\C^N)$
with $\rho>0$,
we have
\[
2\Re\langle \hat X v,\varDelta\varphi\,v\rangle
=
2\Re\langle \hat X v,\varphi'' v\rangle
-\Big\langle v,\Big[(n-1)\varphi''+(n-1)\sp
2\frac{\varphi'}{r}\Big]v\Big\rangle.
\]
We use the above relation to rewrite \eqref{g41} as
\begin{align}
&
\Re\Big\langle\Big[
-\jj \bm\alpha\cdot\bm\nabla+4(\lambda-\jj \hat\alpha\varphi')\hat{X}
+(2n-1)(\lambda-\jj \hat\alpha\varphi')
-\beta m
-2 \jj \hat\alpha r\varphi''
\Big]v,(D\sb m\sp\varphi-\lambda)v\Big\rangle
\nonumber\\[1ex]
&
=
\norm{\bm\nabla v}\sp 2
+4
\norm{(\frac{\varphi'}{r})\sp{\frac 1 2}\hat{X} v}\sp 2
+2\Re\langle \hat{X} v,\varphi''v\rangle
+\Big\langle v,\Big[
\lambda\sp 2-m\sp 2+\varphi'\sp 2+2 r\varphi'\varphi''-(n-1)\varphi''
-(n-1)\sp 2\frac{\varphi'}{r}
\Big]v\Big\rangle.
\nonumber
\end{align}
For any positive continuous function $\mu(r)$,
the above relation yields the following
inequality:
\begin{align}
\label{g41rr}
&
\frac 1 2
\Norm{
\frac{1}{\mu}
\big[
-\jj \bm\alpha\cdot\bm\nabla+4(\lambda-\jj \hat\alpha\varphi')\hat{X}
+(2n-1)
(\lambda-\jj \hat\alpha\varphi')
-\beta m
-2 \jj \hat\alpha r\varphi''
\big]
v
}\sp 2
+
\frac{\norm{\mu(D\sb m\sp\varphi-\lambda)v}\sp 2}{2}
\\
&
\ge
\norm{\bm\nabla v}\sp 2
+3
\norm{\big(\frac{\varphi'}{r}\big)\sp{\frac 1 2}\hat{X} v}\sp 2
-\norm{\big(\frac{r}{\varphi'}\big)\sp{\frac 1 2}\varphi''v}\sp 2
+\Big\langle v,\Big[
\lambda\sp 2-m\sp 2+\varphi'\sp 2+2 r\varphi'\varphi''-(n-1)\varphi''
-(n-1)\sp 2\frac{\varphi'}{r}
\Big]v\Big\rangle.
\nonumber
\end{align}
Since
$\frac 1 2(a+b+c+d)^2
\le 2(a^2 +b^2+c^2+d^2)$,
\eqref{g41rr} leads to
\begin{align}
\label{g41rrr}
&
2\Big(
\Norm{
\frac{\bm\alpha\!\cdot\!\bm\nabla u}{\mu}
}^2
+
\Norm{\frac{4\lambda\hat{X}u}{\mu}}^2
+
\Norm{\frac{4\hat\alpha\varphi'\hat{X}u}{\mu}}^2
+\Norm{
\frac{
[(2n-1)
(\abs{\lambda}+\varphi')
+m
+2r\abs{\varphi''}
]v}{\mu}
}\sp 2
\Big)
+
\frac{\norm{\mu(D\sb m\sp\varphi-\lambda)v}\sp 2}{2}
\nonumber
\\[1ex]
&
\ge
\norm{\bm\nabla v}\sp 2
+3
\norm{\big(\frac{\varphi'}{r}\big)\sp{\frac 1 2}\hat{X} v}\sp 2
-\norm{\big(\frac{r}{\varphi'}\big)\sp{\frac 1 2}\varphi''v}\sp 2
+\Big\langle v,\Big[
\lambda\sp 2-m\sp 2+\varphi'\sp 2+2 r\varphi'\varphi''-(n-1)\varphi''
-(n-1)\sp 2\frac{\varphi'}{r}
\Big]v\Big\rangle.
\end{align}
To eliminate the first three terms
from the left-hand side of \eqref{g41rrr}
(with the help of the first two terms from the right-hand side),
we require that $\mu(r)$ be such that
\begin{equation}\label{enough1}
2
\left(
\Norm{\frac{\bm\alpha\cdot\bm\nabla v}{\mu}
}\sp 2
+
\Norm{
\frac{4\lambda\hat{X}v}{\mu}
}\sp 2
\right)
\le
\frac 1 2\norm{\bm\nabla v}\sp 2,
\end{equation}
\begin{equation}\label{enough2}
\Norm{
\frac{4\hat\alpha\varphi'\hat{X}v}{\mu}
}\sp 2
\le
\Norm{(\frac{\varphi'}{r})\sp{\frac 1 2}\hat{X} v}\sp 2,
\end{equation}
where
\[
\norm{\bm\nabla v}
:=
\Big(\sum\sb{\jmath=1}\sp{n}\norm{\p\sb\jmath v}\sp 2\Big)\sp{\frac 1 2}.
\]
Since
\[
\Norm{\frac{\bm\alpha\cdot\bm\nabla v}{\mu}}\sp 2
\le
\Big(
\sum\sb{\jmath=1}\sp{n}
\Norm{\frac{\p\sb\jmath v}{\mu}}
\Big)\sp 2
\le
n
\sum\sb{\jmath=1}\sp{n}\Norm{\frac{\p\sb\jmath v}{\mu}}\sp 2
=:
n\Norm{\frac{\bm\nabla v}{\mu}}\sp 2,
\]
while $\hat X v=r\p\sb r v$,
resulting in
\begin{equation}\label{xu-nu}
\norm{\hat X v}\sb{\C^N}\le r\norm{\bm\nabla v}\sb{\C^N},
\qquad
\forall v\in C\sp 1(\R\sp n,\C^N),
\end{equation}
we see that \eqref{enough1} will hold
whenever
\[
2\left(
\frac{n+(4\lambda r)^2}{\mu^2}
\right)
\le
\frac 1 2,
\]
\begin{equation}\label{enough-1}
\mu(r)\sp 2
\ge 4 n + 64\lambda^2+r^2.
\end{equation}
To satisfy \eqref{enough2},
again in view of \eqref{xu-nu},
it is enough to have
\begin{equation}\label{enough-2}
32\frac{r\varphi'}{\mu(r)\sp 2}
\le
1.
\end{equation}
To comply with both
\eqref{enough-1} and \eqref{enough-2},
it is enough to require that
\begin{equation}\label{enough-3}
\mu(r)
\ge
2
\Big(
n+16\lambda^2 r^2+8r\varphi'
\Big)^{1/2}.
\end{equation}
Taking into account \eqref{enough1} and \eqref{enough2},
the inequality \eqref{g41rrr} yields
\begin{align}
&
2\Norm{\frac{1}{\mu}
\Big[
(2n-1)
(\abs{\lambda}+\varphi')
+m
+2r\abs{\varphi''}
\Big]v}\sp 2
+
\frac{\norm{\mu(D\sb m\sp\varphi-\lambda)v}\sp 2}{2}
\nonumber\\
&
\ge
\frac 1 2\norm{\bm\nabla v}^2
+\norm{\big(\frac{\varphi'}{r}\big)^{1/2}\hat X v}^2
+
\Big\langle v,
\Big[
\lambda\sp 2-m\sp 2+\varphi'\sp 2+2 r\varphi'\varphi''-(n-1)\varphi''
-(n-1)\sp 2\frac{\varphi'}{r}
-\frac{r\varphi''\sp 2}{\varphi'}
\Big]v\Big\rangle,
\nonumber
\end{align}
hence
\begin{eqnarray}
&&
\norm{\mu(D\sb m\sp\varphi-\lambda)v}\sp 2
\ge
\norm{\bm\nabla v}\sp 2
+
2\norm{(r\varphi')\sp{1/2} \partial\sb r v}\sp 2
\nonumber
\\
&&
\qquad\qquad\qquad\qquad
+
2\left\langle v,
\left[
\lambda\sp 2-m\sp 2+\varphi'\sp 2+2 r\varphi'\varphi''
-(n-1)\varphi''
-\frac{(n-1)\sp 2\varphi'}{r}-\frac{r\varphi''\sp 2}{\varphi'}
\right]
v\right\rangle
\nonumber
\\
&&
\qquad\qquad\qquad\qquad
-
4
\left\langle v,
\left[\frac{(2n-1)(\abs{\lambda}+\varphi')+m+2r\abs{\varphi''}}
{\mu}\right]^2
v\right\rangle,
\nonumber
\end{eqnarray}
and \eqref{euw} follows.

For $u\in H\sp 1\sb{\mathrm{comp}}(\varOmega\sb\rho,\C^N)$,
substituting $v=e\sp\varphi u$ into \eqref{euw}
and using the identity
$D\sb m\sp\varphi(e\sp\varphi u)=e\sp\varphi D\sb m u$
(cf. \eqref{def-dm-varphi}),
we also have
\begin{equation}\label{carleman-compact-u}
\norm{\bm\nabla(e^\varphi u)}^2
+
2\norm{(r\varphi')^{1/2}\p\sb r(e^\varphi u)}^2
+\langle e^\varphi u,Z e^\varphi u\rangle
\le
\norm{\mu e\sp\varphi (D\sb m-\lambda)u}^2
\end{equation}
for any $u\in H\sp 1\sb{\mathrm{comp}}(\varOmega\sb\rho,\C^N)$,
proving \eqref{euw-u}.
\end{proof}

\begin{lemma}\label{lemma-big-half}
For any $\varphi\in\mathscr{C}\sb\lambda(M,\calN,\rho,\nu)$
there is $R_0=R_0(\varphi)\ge\rho$, $R_0<\infty$, such that
for any $r\ge R_0$
the following inequality is satisfied:
\begin{equation}\label{f-ge-half}
Z(r)
\ge
\lambda^2-m^2
+\varphi'(r)^2+2r\varphi'(r)\varphi''(r),
\end{equation}
with $Z(r)$ defined in \eqref{def-Z}.
If additionally
\begin{equation}\label{ge-quarter}
\varphi''(r)\geq -\frac{\varphi'(r)}{4r},
\qquad
r\ge\rho,
\end{equation}
then
\begin{equation}\label{r-r}
R_0=R_0(M,\rho,m,n,\lambda),
\end{equation}
independent of $\calN>0$ and $\nu$,
and it can be chosen uniformly in
$\varphi\in\mathscr{C}\sb\lambda(M,\calN,\rho,\nu)$.
\end{lemma}

\begin{proof}
As follows from the definition \eqref{def-Z},
we need to satisfy the following inequality:
\begin{eqnarray}
&
2(n-1)\abs{\varphi''(r)}+(n-1)\sp 2\frac{\varphi'(r)}{r}
+\frac{r\varphi''(r)\sp 2}{\varphi'(r)}
+4\left[
\frac{(2n-1)(\abs{\lambda}+\varphi'(r))+m+2r|\varphi''(r)|}{\mu}
\right]^2
\nonumber\\[1ex]
&
\qquad
\le \label{big-half}
\lambda\sp 2-m\sp 2+\varphi'(r)\sp 2
+2r\varphi'(r)\varphi''(r).
\end{eqnarray}
Taking into account the bound
$r\abs{\varphi''}\le M\varphi'$
as long as $r\ge\rho$
(cf. Definition~\ref{definition-c})
and simplifying some coefficients,
we see that
the inequality \eqref{big-half} will follow from
\[
2(M+n)^2
\frac{\varphi'}{r}
+4\left[
\frac{2
\left(
n\abs{\lambda}+m+(M+n)\varphi'
\right)}{\mu}
\right]^2
\le
\lambda\sp 2-m\sp 2+\varphi'\sp 2
+2r\varphi'\varphi''.
\]
Taking into account the bound
$\mu\ge 8\abs{\lambda}r$
which follows from \eqref{def-mu},
we see that
it suffices to satisfy the inequality
\begin{equation}
\label{big-half-1}
2(M+n)^2
\frac{\varphi'}{r}
+
\frac{\big(n\abs{\lambda}+m+(M+n)\varphi'\big)^2}{4\lambda^2r^2}
\le
\lambda\sp 2-m\sp 2+\varphi'\sp 2
+2r\varphi'\varphi''.
\end{equation}

\medskip

First we consider the case when $\varphi'$ is bounded.
Using the bound
\[
\lambda\sp 2-m\sp 2+\varphi'\sp 2+2r\varphi'\varphi''\ge \nu,
\qquad \forall r\ge\rho,
\]
we see that
we will have \eqref{big-half-1} satisfied
for $r\ge R$
as long as
\begin{equation}\label{enough}
2\frac{(M+n)^2}{r}\sup\sb{r\geq R} \varphi'(r)
+
\frac{\big(n\abs{\lambda}+m+(M+n)\sup\sb{r\geq R}\varphi'(r)\big)^2}
{4\lambda^2 r^2}
\le\nu.
\end{equation}
Thus,
\eqref{big-half}
is satisfied for $r\ge R_0(\varphi)$
with $R_0(\varphi)=C\sup\sb{r\ge\rho}\varphi'(r)$
as long as the constant $C>0$ is large enough.

\medskip

Now let us consider the case
$r\varphi''\geq -\varphi'/4$;
this case includes the situation when
$\lim_{r\to\infty}\varphi'(r)=+\infty$.
Due to this bound from below on $r\varphi''$,
\eqref{big-half-1} will be satisfied
if we provide
\begin{equation}\label{big-half-2}
2(M+n)^2\frac{\varphi'}{r}
+
\frac{\big(n\abs{\lambda}+m+(M+n)\varphi'\big)^2}{4\lambda^2 r^2}
\le
\frac{\lambda\sp 2-m\sp 2+\varphi'\sp 2}{2}.
\end{equation}
The above inequality will be satisfied
for all $\varphi'\ge 0$
as long as
$r$ is large enough to ensure that all the roots
of the polynomial function
\[
\zeta
\mapsto
\frac{\lambda\sp 2-m\sp 2+\zeta\sp 2}{2}
-
2(M+n)^2\frac{\zeta}{r}
-
\frac{\big(n\abs{\lambda}+m+(M+n)\zeta\big)^2}{4\lambda^2 r^2}
\]
are negative.
One can see that the lower bound on $r$
only depends on $M$, $m$, $n$, and $\abs{\lambda}$.
Note that one needs $\abs{\lambda}>m$.
\end{proof}

\begin{lemma}\label{lemma-du}
Let
$\varphi\in\mathscr{C}\sb\lambda(M,\calN,\rho,\nu)$
and let $R_0=R_0(\varphi)$ be as in Lemma~\ref{lemma-big-half}.
\begin{enumerate}
\item
\label{lemma-du-i}
Then there is the following  inequality
for any
$u\in H\sp 1\sb{\mathrm{comp}}(\varOmega\sb{R_0},\C^N)$:
\begin{equation}\label{carleman-compact-u-gamma}
\norm{\bm\nabla(e^\varphi u)}^2
+
2\norm{(r\varphi')^{1/2}\p\sb r(e^\varphi u)}^2
+\norm{\gamma e^\varphi u}^2
\le\norm{\mu e\sp\varphi (D\sb m-\lambda)u}^2
,
\end{equation}
with $\mu$ and $\gamma$
defined in \eqref{def-mu}, \eqref{def-g}.

\item
\label{lemma-du-ii}
Let
$V\in\mathscr{B}\left(H\sp 1\sb{\mathrm{comp}}(\varOmega\sb\rho,\C^N),\,L\sp 2(\varOmega\sb\rho,\C^N)\right)$
be a multiplication operator
and assume that there are $\kappa\in[0,1)$
and $R_1=R_1(\varphi,V)<\infty$
such that
\begin{equation}\label{ineq}
\norm{
\mu V v}
\le
\kappa
\left(
\norm{\bm\nabla v}^2
+
\norm{\gamma v}^2
\right)^{1/2},
\qquad \forall v\in H^1\sb{\mathrm{comp}}(\varOmega_{R_1},\C^N).
\end{equation}
Then
\begin{eqnarray}\label{carleman-compact-u-gamma-v}
(1-\kappa)
\left(
\norm{\bm\nabla(e^\varphi u)}^2
+
2\norm{(r\varphi')^{1/2}\p\sb r(e^\varphi u)}^2
+\norm{\gamma e^\varphi u}^2
\right)^{1/2}
\le\norm{\mu e\sp\varphi (D\sb m+V-\lambda)u}
,
\end{eqnarray}
for any
$u\in H\sp 1\sb{\mathrm{comp}}(\varOmega\sb R,\C^N)$
with
$
R=\max(R_0(\varphi),R_1(\varphi,V)).
$
\end{enumerate}
\end{lemma}

\begin{proof}
The proof of Lemma~\ref{lemma-du}~\itref{lemma-du-i}
readily follows from
Lemma~\ref{lemma-bg-inequality}
and
Lemma~\ref{lemma-big-half}.

To prove Lemma~\ref{lemma-du}~\itref{lemma-du-ii},
we apply the assumption
\eqref{ineq}
(where we take
$v=e^\varphi u\in H^1\sb{\mathrm{comp}}(\varOmega\sb R,\C^N)$,
with $R=\max(R_0,R_1)$)
to the inequality
\eqref{carleman-compact-u-gamma},
obtaining
\[
(1-\kappa)
\left(
\norm{\bm\nabla(e^\varphi u)}^2
+
2\norm{(r\varphi')^{\frac 1 2}\p\sb r(e^\varphi u)}^2
+
\norm{\gamma e^\varphi u}^2
\right)^{\frac 1 2}
\le\norm{\mu e^\varphi(D_m+V-\lambda)u}.
\]
This proves \eqref{carleman-compact-u-gamma-v}.
\end{proof}

\begin{remark}\label{remark-c-l}
At this point, we need to mention how
Theorem~\ref{theorem-Carleman-l} is proved.
For the functions
$u\sb\pm=\Pi\sp\pm u$,
with
\begin{equation}\label{def-pi-pm}
\Pi\sp\pm=\frac 1 2(1\mp\jj J)
\end{equation}
the projections onto eigenspaces of $J$
corresponding to $\pm\jj\in\sigma(J)$,
we obtain from
Lemma~\ref{lemma-du}~\itref{lemma-du-i}
the following inequalities:
\[
\norm{e^\varphi\nabla u\sb{+}}^2
+2\norm{(r\varphi')^{1/2}u\sb{+}}
+\norm{\gamma\sb{+} e^\varphi u\sb{+}}^2
\le\norm{\mu\sb{+} e^\varphi((D_m-\omega)+\jj\lambda)u\sb{+}}^2,
\]

\[
\norm{e^\varphi\nabla u\sb{-}}^2
+2\norm{(r\varphi')^{1/2}u\sb{-}}
+\norm{\gamma\sb{-} e^\varphi u\sb{-}}^2
\le\norm{\mu\sb{-} e^\varphi((D_m-\omega)-\jj\lambda)u\sb{-}}^2,
\]
with appropriate expressions for $\mu\sb\pm$ and $\gamma\sb\pm$
(with $\abs{\lambda}\pm\abs{\omega}>m$
in place of $\lambda$),
and, since $\Pi\sp\pm$ are self-adjoint,
we may add up the above inequalities,
arriving at
\[
\norm{e^\varphi\nabla u}^2
+2\norm{(r\varphi')^{1/2}u}
+\norm{\min\sb\pm(\gamma\sb\pm) e^\varphi u}^2
\le\norm{\max\sb\pm(\mu\sb\pm)e^\varphi(J(D_m-\omega)-\lambda)u}^2.
\]
Introducing $V$ into the estimates
for $D_m$ and for $J(D_m-\omega)$
is done verbatim.
\end{remark}

Let us extend \eqref{euw-u}
to $u\in H\sp 1\sb 0(\varOmega\sb\rho,\C^N)$
which are no longer compactly supported.

\begin{lemma}\label{lemma-asdf}
Let
$\varphi\in\mathscr{C}\sb\lambda(M,\calN,\rho,\nu)$.

\begin{enumerate}
\item
\label{lemma-asdf-i}
Let $R_0=R_0(\varphi)\ge\rho$ be
as in Lemma~\ref{lemma-big-half},
so that
\eqref{f-ge-half}
is satisfied for $r\ge R_0$.
Then, for any $u\in H\sp 1\sb 0(\varOmega\sb{R_0},\C^N)$,
one has:
\begin{equation}
\label{euw-u-1}
\norm{\bm\nabla(e^\varphi u)}^2
+
2\norm{(r\varphi')^{1/2}\p\sb r(e^\varphi u)}^2
+
\norm{\gamma e^\varphi u}^2
\le\norm{\mu e^\varphi(D\sb m-\lambda)u}^2.
\end{equation}
\item
\label{lemma-asdf-ii}
Let $V\in L^n\sb{\mathrm{loc}}(\R^n,\End(\C^N))$
and assume that there are $\kappa\in[0,1)$
and $R_1=R_1(\varphi,V)<\infty$,
$R_1\ge\rho$ such that
for any $v\in H^1\sb 0(\varOmega_{R_1},\C^N)$
one has
\begin{equation}\label{ineq-2}
\norm{\mu V v}
\le
\kappa
\left(
\norm{\bm\nabla v}^2
+
\norm{\gamma v}^2
\right)^{1/2}.
\end{equation}
If $V=0$, then we set $R_1(\varphi,0)=\rho$.
Then
for any $u\in H\sp 1\sb 0(\varOmega\sb R,\C^N)$
with
$R=\max(R_0(\varphi),R_1(\varphi,V))$,
one has
\begin{eqnarray}\label{carleman-compact-u-gamma-v-2}
(1-\kappa)
\left(
\norm{\bm\nabla(e^\varphi u)}^2
+
2\norm{(r\varphi')^{1/2}\p\sb r(e^\varphi u)}^2
+
\norm{\gamma e^\varphi u}^2
\right)^{1/2}
\le\norm{\mu e\sp\varphi (D\sb m+V-\lambda)u}
,
\end{eqnarray}
\end{enumerate}
\end{lemma}
\begin{remark}
It is enough to assume that \eqref{ineq-2}
takes place for $v\in H^1\sb {\mathrm{comp}}(\varOmega_{R_1},\C^N)$,
since the latter implies \eqref{ineq-2} for
$v\in H^1\sb 0(\varOmega_{R_1},\C^N)$
using ideas of {\bf Step 0} of the proof below.
\end{remark}
\begin{proof}
We will prove
Lemma~\ref{lemma-asdf}~\itref{lemma-asdf-ii};
then
Lemma~\ref{lemma-asdf}~\itref{lemma-asdf-i}
will follow as well.

\begin{proofstep}
\item[{\bf Step 0.\ \ }]
First, we consider the case when $\varphi(r)$ is bounded.
Let $\eta\in C^\infty\sb{\mathrm{comp}}([-2,2])$,
$0\le\eta\le 1$,
$\eta\at{[-1,1]}\equiv 1$,
and define
$\eta\sb j(x)=\eta(x/j)$.
Let $u\in H\sp 1(\R\sp n,\C^N)$ with
$\supp u\subset\varOmega\sb R$,
and define
\[
u\sb j=\eta\sb j u\in H\sp 1\sb{\mathrm{comp}}(\varOmega\sb R,\C^N).
\]
Applying Lemma~\ref{lemma-du}~\itref{lemma-du-ii}
to $u\sb j$,
we have
\[
(1-\kappa)
\left(
\norm{\bm\nabla(e^\varphi u_j)}^2
+
2\norm{(r\varphi')^{1/2}\p\sb r(e^\varphi u_j)}^2
+\norm{\gamma e^\varphi u_j}^2
\right)^{1/2}
\le
\norm{\mu e\sp\varphi (D\sb m+V-\lambda)u_j}.
\]
Using the identities
\[
\bm\nabla(e^\varphi u\sb j)
=
\eta\sb j\bm\nabla(e^\varphi u)
+e^\varphi u\bm\nabla\eta\sb j,
\]

\[
(D\sb m+V-\lambda)u\sb j
=\eta\sb j(D\sb m+V-\lambda) u-\jj(\bm\alpha\cdot\bm\nabla\eta\sb j)u,
\]
one has
\begin{eqnarray}
\label{dct}
&
(1-\kappa)
\Big(
\left(
\norm{\eta\sb j\bm\nabla(e^\varphi u)}
-\norm{e^\varphi u\bm\nabla\eta\sb j}
\right)^2
+
2\left(\norm{\eta_j(r\varphi')^{1/2}\p\sb r(e^\varphi u)}
-\norm{(r\varphi')^{1/2}e^\varphi u\p\sb r \eta_j}\right)^2
+
\norm{\gamma e^\varphi u\eta\sb j}^2
\Big)^{1/2}
\nonumber
\\[1ex]
&
\le
\norm{\mu e\sp\varphi \eta\sb j(D\sb m+V-\lambda)u}
+
\norm{\mu e\sp\varphi (\bm\alpha\cdot\bm\nabla\eta\sb j)u}.
\end{eqnarray}
We claim that
\begin{equation}\label{to-zero-12}
\lim\sb{j\to\infty}
\norm{e^\varphi u\bm\nabla\eta\sb j}
=0,
\qquad
\lim\sb{j\to \infty}\norm{(r\varphi')^{1/2}e^\varphi u\p\sb r \eta_j}=0,
\end{equation}
\begin{equation}\label{to-zero-3}
\lim\sb{j\to\infty}
\norm{\mu e^\varphi
(\bm\alpha\cdot\bm\nabla\eta\sb j)u}
=0.
\end{equation}
Let us prove the inequality \eqref{to-zero-3}.
According to our assumptions, $u$ is in $L\sp 2$ and $\varphi$
is bounded,
hence $e\sp\varphi  u\in L\sp 2$ and
$
\lim\sb{j\to\infty}
\norm{\chi\sb{[j,2j]}(\abs{x})e\sp\varphi u}\sb{L\sp 2}=0,
$
while
\begin{equation}\label{q-nabla-eta-c}
\norm{\mu\nabla\eta\sb{j}}\sb{L^\infty}
\le
\norm{\nabla\eta\sb j}\sb{L^\infty}
\norm{\chi\sb{[-2j,2j]}\mu}\sb{L^\infty}
=\frac{1}{j}\norm{\nabla\eta}\sb{L^\infty}
\norm{\chi\sb{[-2j,2j]}\mu}\sb{L^\infty}
\end{equation}
is bounded since $\mu(r)=O(\langle r\rangle)$
(cf. \eqref{def-mu}),
due to $\varphi$ assumed bounded, so that $\varphi'\to 0$ as $r\to\infty$.
The proof of the second inequality
in \eqref{to-zero-12}
is the same
since $(r\varphi')^{1/2}\le(\calN r^2)^{1/2}=O(\langle r\rangle)$,
while
the proof of the first inequality is slightly simpler
since there is no linearly growing factor.

Taking into account \eqref{to-zero-12}, \eqref{to-zero-3}
and applying the Fatou lemma to
$\norm{\eta\sb j\bm\nabla(e^\varphi u)}$,
$\norm{\gamma e^\varphi u\eta\sb j}$,
and
$\norm{\mu e^\varphi\eta\sb j(D\sb m+V-\lambda) u}$,
we conclude from \eqref{dct} that, as we stated,
\begin{eqnarray}\label{carleman-noncompact}
(1-\kappa)
\left(
\norm{\bm\nabla(e^\varphi u)}^2
+
2\norm{(r\varphi')^{1/2}\p\sb r(e^\varphi u)}^2
+\norm{\gamma e^\varphi u}^2
\right)^{1/2}
\le
\norm{\mu e\sp\varphi (D\sb m+V-\lambda)u}.
\end{eqnarray}

Unbounded $\varphi$
are considered precisely as in~\cite[Theorem 3]{MR880983},
which closely follows the approach of~\cite{MR685028}.
We assume that
$\varphi\sb 0:=\varphi\in\mathscr{C}\sb\lambda(M,\calN,\rho,\nu)$,
for some $M,\,\calN,\,\rho\ge 1$ and $\nu>0$.
Without loss of generality, we also assume that
$\varphi\sb 0(1)=0$.
Below we will consider sequences $(\varphi\sb\epsilon)\sb{\epsilon>0}$
converging pointwise to $\varphi\sb 0$
from below.

\item
Let us assume that
$\varphi\in\mathscr{C}\sb\lambda(M,\calN,\rho,\nu)$
is unbounded but that $\varphi'\to 0$ as $r\to\infty$.
We approximate $\varphi\sb 0=\varphi$ by
$\varphi\sb\epsilon\in C\sp 2(\R\sb{+})$,
$\epsilon\in(0,1)$:
\[
\varphi\sb\epsilon(r)
=\int\sb{1}\sp r
\frac{\varphi\sb 0'(t)}{1+\epsilon t\sp 2}\,dt.
\]
Then,
for $r\ge\rho$,
the functions
$\varphi\sb\epsilon$
are monotonically increasing,
satisfy $\sup\sb{r\ge 0}\varphi\sb\epsilon(r)<\infty$,
and
\begin{equation}\label{convergences}
\varphi\sb\epsilon(r)\nearrow\varphi\sb 0(r),
\qquad
\varphi\sb\epsilon'(r)\nearrow\varphi\sb 0'(r),
\qquad
\varphi\sb\epsilon''(r)\to\varphi\sb 0''(r),
\qquad
r\ge\rho.
\end{equation}
To reduce the argument to {\bf Step 0}
(based on Lemma~\ref{lemma-du}~\itref{lemma-du-ii}),
we need to check that
the inequality \eqref{f-ge-half} in Lemma~\ref{lemma-big-half}
is satisfied by $\varphi\sb\epsilon$
and the corresponding $Z\sb\epsilon$
(given by \eqref{def-Z}
with $\varphi\sb\epsilon$ instead of $\varphi$)
for $r\ge R_0=R_0(\varphi)$
as long as $\epsilon>0$ is sufficiently small.
Since
\begin{equation}\label{p-p-e-1}
0<
\varphi\sb\epsilon'(r)=
\frac{\varphi\sb 0'(r)}{1+\epsilon r\sp 2}
<\varphi\sb 0'(r)
\le \calN r,
\qquad
r\ge\rho,
\end{equation}
one has
$
\varphi\sb\epsilon''(r)
=\frac{\varphi\sb 0''(r)}{1+\epsilon r\sp 2}
-\frac{2\varphi\sb 0'(r)}{r}\frac{\epsilon r\sp 2}{(1+\epsilon r\sp 2)\sp 2}
$
and hence
\begin{equation}\label{p-p-e-2}
\frac{\varphi\sb\epsilon''(r)}{\varphi\sb\epsilon'(r)}
=
\frac{\varphi\sb 0''(r)}{\varphi\sb 0'(r)}
-\frac{2}{r}\frac{\epsilon r\sp 2}{1+\epsilon r\sp 2},
\qquad
\Abs{\frac{r\varphi\sb\epsilon''(r)}{\varphi\sb\epsilon'(r)}}
\le
\Abs{\frac{r\varphi\sb 0''(r)}{\varphi\sb 0'(r)}}
+2
\le M+2.
\end{equation}
We claim that for any $s<1$
arbitrarily close to $1$
there exists $\epsilon\sb 0=\epsilon\sb 0(s)\in(0,1)$
such that
\begin{equation}\label{weaker-class}
\varphi\sb\epsilon\in\mathscr{C}\sb\lambda(M+2,\calN,R_0,s\nu),
\qquad \forall\epsilon\in(0,\epsilon_0);
\end{equation}
according to Definition~\ref{definition-c},
we are left to verify that
\begin{equation}\label{s-nu}
\lambda^2-m^2+\varphi'^2+2r\varphi'\varphi''\ge s\nu,
\qquad
r\ge R_0.
\end{equation}
Define
\begin{equation}\label{def-nu-epsilon}
\nu\sb\epsilon
=\inf\sb{r>\rho}
\big(
\lambda^2-m^2
+\varphi\sb\epsilon'(r)^2
+2r\varphi\sb\epsilon'(r)\varphi\sb\epsilon''(r)
\big).
\end{equation}
Since
\[
2r\varphi\sb\epsilon'\varphi\sb\epsilon''
=
\frac{2r\varphi\sb 0'\varphi\sb 0''(r)}{(1+\epsilon r^2)^2}
-\frac{4\epsilon r^2\varphi\sb 0'(r)^2}{(1+\epsilon r^2)^3},
\]
one has
\begin{align}\label{r-r-r}
&
\varphi\sb\epsilon'(r)^2
+2r\varphi\sb\epsilon'(r)\varphi\sb\epsilon''(r)
-\varphi\sb 0'(r)^2-2r\varphi\sb 0'(r)\varphi\sb 0''(r)
=
\Big(\frac{1}{(1+\epsilon r^2)^2}-1\Big)
\big(\varphi\sb 0'(r)^2
+
2r\varphi\sb 0'\varphi\sb 0''(r)\big)
-\frac{4\epsilon r^2\varphi\sb 0'(r)^2}{(1+\epsilon r^2)^3}.
\end{align}
Due to
$2r\abs{\varphi\sb 0''(r)}\le M\varphi\sb 0'(r)$
and $\varphi\sb 0'(r)\to 0$
as $r\to\infty$,
the absolute value of the right-hand side of
\eqref{r-r-r}
goes to zero as $r\to\infty$
uniformly in $\epsilon\in(0,1)$.
Therefore,
for any fixed $s\in (0,1)$,
we may choose some finite $R\sp\ast=R\sp\ast(s)\ge R_0$
(with $R_0=R_0(\varphi)$ from Lemma~\ref{lemma-big-half})
large enough so that
the right-hand side of \eqref{r-r-r}
is smaller than $s\nu$
for $r\ge R\sp\ast$, $\epsilon\in(0,1)$,
while
the left-hand side of
\eqref{big-half-1}
(with $\varphi\sb\epsilon$ instead of $\varphi$
and with $M+2$ instead of $M$)
is smaller than $(1-s)\nu$.
Then
for $\epsilon\in(0,1)$
the functions
$\varphi\sb\epsilon$
will satisfy \eqref{big-half-1}
(with $M+2$ instead of $M$)
and hence \eqref{f-ge-half}
for $r\ge R\sp\ast$.

At the same time,
the convergences
\eqref{convergences}
are uniform on the interval $r\in[R_0,R\sp\ast]$.
Since $\varphi$ satisfies
the inequality \eqref{f-ge-half}
for $r\ge R_0$,
there is $\epsilon_0=\epsilon_0(s)\in(0,1)$ sufficiently small
so that the functions $\varphi\sb\epsilon$
with $\epsilon\in(0,\epsilon_0)$
satisfy
\begin{equation}\label{f-ge-half-epsilon}
Z\sb\epsilon(r)
\ge
s\Big
(\lambda^2-m^2
+\varphi\sb\epsilon'(r)^2
+2r\varphi\sb\epsilon'(r)
\varphi\sb\epsilon''(r)
\Big),
\qquad
r\in[R_0,R\sp\ast],
\end{equation}
where $Z\sb\epsilon$ is defined by
\eqref{def-Z} with $\varphi\sb\epsilon$
instead of $\varphi$,
while
$\nu\sb\epsilon$ defined in \eqref{def-nu-epsilon}
will satisfy
\begin{equation}\label{nu-epsilon-s-nu}
\nu\sb\epsilon>s\nu,
\qquad \forall\epsilon\in(0,\epsilon_0).
\end{equation}
Together with
\eqref{p-p-e-1} and \eqref{p-p-e-2},
this leads to the desired inclusion \eqref{weaker-class}
and to the inequality \eqref{f-ge-half-epsilon}
satisfied for all $r\ge R_0=R_0(\varphi)$.

The previous argument shows that
there is $\epsilon_0\in(0,1)$
so that for any $\epsilon\in(0,\epsilon_0)$,
\begin{equation}\label{ineq-2-epsilon}
\left|\varphi\sb\epsilon'(r)^2
+2r\varphi\sb\epsilon'(r)\varphi\sb\epsilon''(r)
-\varphi\sb 0'(r)^2-2r\varphi\sb 0'(r)\varphi\sb 0''(r)\right|
\leq (1-s)\nu,
\quad
r\geq R_0.
\end{equation}
Then notice that
since \eqref{ineq-2} is satisfied for $\varphi_0$,
one has
\begin{equation}\label{mu-v-v}
\norm{\mu_\epsilon V v}
\le
\norm{\mu V v}
\le
\kappa
\left(
\norm{\bm\nabla v}^2
+
\norm{\gamma v}^2
\right)^{1/2},
\qquad \forall v\in H^1\sb 0(\varOmega_{R_1},\C^N),
\end{equation}
where $\mu_\epsilon$ is the expression $\mu$ in \eqref{def-mu} for $\varphi_\epsilon$
and
$R_1=R_1(\varphi,V)$.

Due to \eqref{mu-v-v},
we deduce that for all $\epsilon \in (0,\epsilon_0)$
and $v\in H^1\sb 0(\varOmega_{R_1},\C^N)$,
one has
\begin{equation}\label{mu-v-v-gamma}
\norm{
\mu_\epsilon V v}
\le
\kappa
\max\left\{1,\Norm{\frac{\gamma}{\gamma_\epsilon}
}\sb{L^\infty(\supp v)}\right\}
\left(
\norm{\bm\nabla v}^2
+
\norm{\gamma_\epsilon v}^2
\right)^{1/2},
\end{equation}
where $\gamma_\epsilon$ is defined by the expression \eqref{def-g}
with $\varphi_\epsilon$ in place of $\varphi$.
We notice that,
in view of
\eqref{def-nu-epsilon},
\eqref{nu-epsilon-s-nu},
and \eqref{ineq-2-epsilon},
for $\epsilon\in(0,\epsilon_0)$,
\[
\max\left\{
1,\Abs{\frac{\gamma}{\gamma_\epsilon}}
\right\}
\le
\max\left\{
1,\frac{\gamma^2}{\gamma_\epsilon^2}
\right\}
\le
1+\frac{\abs{\gamma^2-\gamma_\epsilon^2}}{\gamma_\epsilon^2}
\le 1+\frac{(1-s)\nu}{s\nu}
=\frac{1}{s},
\qquad
r\ge R_0,
\]
hence \eqref{mu-v-v-gamma} yields
\[
\norm{
\mu_\epsilon V v}
\le
\frac{\kappa}{s}
\left(
\norm{\bm\nabla v}^2
+
\norm{\gamma_\epsilon v}^2
\right)^{1/2},
\qquad \forall v\in H^1\sb 0(\varOmega_R,\C^N),
\]
with $R=\max(R_0(\varphi),R_1(\varphi,V))$.
Due to this bound, \eqref{carleman-noncompact}
yields the inequality
\begin{eqnarray}\label{carleman-noncompact-epsilon}
\left(1-\frac{\kappa}{s}\right)
\left(
\norm{\bm\nabla(e^{\varphi\sb\epsilon} u)}^2
+
2\norm{(r\varphi\sb\epsilon')^{1/2}\p\sb r(e^{\varphi\sb\epsilon} u)}^2
+\norm{\gamma e^{\varphi\sb\epsilon} u}^2
\right)^{1/2}
\le
\norm{\mu e\sp{\varphi\sb\epsilon} (D\sb m+V-\lambda)u},
\end{eqnarray}
valid for any $\epsilon\in(0,\epsilon\sb 0)$
and any $u\in H\sp 1\sb 0(\varOmega\sb R,\C^N)$.

By the Fatou lemma
applied to the left-hand side
of \eqref{carleman-noncompact-epsilon}
(where we use the decomposition
$\nabla(e^{\varphi\sb\epsilon} u)=e^{\varphi\sb\epsilon}\nabla u+  e^{\varphi\sb\epsilon}(\nabla\varphi\sb\epsilon) u $)
and the dominated convergence theorem applied
to the right-hand side,
the same inequality
\eqref{carleman-noncompact-epsilon}
holds for $\varphi\sb 0=\varphi$ instead of $\varphi\sb\epsilon$.
Since $s\in(0,1)$ could be chosen arbitrarily close to $1$,
we also have
\begin{eqnarray}\label{carleman-noncompact-epsilon-1}
(1-\kappa)
\left(
\norm{\bm\nabla(e^{\varphi} u)}^2
+
2\norm{(r\varphi')^{1/2}\p\sb r(e^{\varphi} u)}^2
+\norm{\gamma e^{\varphi} u}^2
\right)^{1/2}
\le
\norm{\mu e\sp{\varphi} (D\sb m+V-\lambda)u},
\end{eqnarray}
for any $u\in H\sp 1\sb 0(\varOmega\sb R,\C^N)$.

\item
Assume that
$\varphi\in\mathscr{C}\sb\lambda(M,\calN,\rho,\nu)$
is such that
$\varphi'$ is bounded at infinity.
Let $R_0=R_0(\varphi)\ge\rho$,
$R_0<\infty$
be as in Lemma~\ref{lemma-big-half},
so that $\varphi$
satisfies the inequality
\eqref{f-ge-half}
for $r\ge R_0$.
(Since $\varphi'$ is bounded,
such $R_0$ exists by \eqref{enough}.)
We approximate $\varphi\sb 0:=\varphi$ by
$\varphi\sb\epsilon\in C\sp 2(\R\sb{+})$,
$\epsilon\in(0,1/4)$,
as follows:
\[
\varphi\sb\epsilon(r)
=\varphi\sb 0(r\sp{1-\epsilon}).
\]
Since $\varphi\sb 0(r)$ is increasing
and $\rho\ge 1$,
for each $\epsilon\in(0,1/4)$ and $r\ge\rho$
we have:
$\varphi\sb\epsilon(r)\le\varphi\sb 0(r)$,
$\varphi\sb\epsilon'(r)>0$,
$\lim\sb{r\to\infty}\varphi\sb\epsilon'(r)=0$,
$\varphi\sb\epsilon(r)\nearrow\varphi\sb 0(r)$,
$\varphi\sb\epsilon'(r)\to\varphi\sb 0'(r)$.
Thus,
\[
0<\varphi\sb\epsilon'(r)
=(1-\epsilon)r\sp{-\epsilon}
\varphi\sb 0'(r\sp{1-\epsilon})
\le
(1-\epsilon)r\sp{-\epsilon}
\calN r\sp{1-\epsilon}
\le\calN r,
\]

\[
\varphi\sb\epsilon''(r)
=(1-\epsilon)\sp 2r\sp{-2\epsilon}
\varphi\sb 0''(r\sp{1-\epsilon})
-\epsilon(1-\epsilon)r\sp{-\epsilon-1}
\varphi\sb 0'(r\sp{1-\epsilon}),
\]

\[
\frac{\varphi\sb\epsilon''(r)}{\varphi\sb\epsilon'(r)}
=
(1-\epsilon)r\sp{-\epsilon}
\frac{\varphi\sb 0''(r\sp{1-\epsilon})}{\varphi\sb 0'(r\sp{1-\epsilon})}
-\epsilon r\sp{-1},
\]

\[
\Abs{\frac{r\varphi\sb\epsilon''(r)}{\varphi\sb\epsilon'(r)}}
\le
(1-\epsilon)
\Abs{
\frac{r\sp{1-\epsilon}\varphi\sb 0''(r\sp{1-\epsilon})}
{\varphi\sb 0'(r\sp{1-\epsilon})}}
+\epsilon
\le M.
\]
In the last inequality,
we took into account that $M\ge 1$.

    From the identity
\begin{align*}
&
\varphi\sb\epsilon'(r)^2+2r \varphi\sb\epsilon'(r)\varphi\sb\epsilon''(r)
\\[1ex]
&=(1-\epsilon)^2(1-2\epsilon)r\sp{-2\epsilon}(\varphi\sb 0'(r\sp{1-\epsilon}))^2+2r (1-\epsilon)^3r\sp{-3\epsilon}\varphi\sb
0'(r\sp{1-\epsilon})\varphi\sb 0''(r\sp{1-\epsilon})
\\[1ex]
&=(1-\epsilon)^2(1-2\epsilon)r\sp{-2\epsilon}\left[(\varphi\sb 0'(r\sp{1-\epsilon}))^2+2r\varphi\sb
0'(r\sp{1-\epsilon})\varphi\sb 0''(r\sp{1-\epsilon})\right]
\\[1ex]
&
\quad
+\left[(1-\epsilon)^3r\sp{-3\epsilon}-(1-\epsilon)^2(1-2\epsilon)r\sp{-2\epsilon}\right]2r \varphi\sb0'(r\sp{1-\epsilon})\varphi\sb 0''(r\sp{1-\epsilon})
\end{align*}
we deduce as in the previous step that for any $s\in(0,1)$
there is $\epsilon\sb 0=\epsilon\sb 0(s)\in(0,1/4)$
such that for $\epsilon\in(0,\epsilon\sb 0)$
one has
$\varphi\sb\epsilon\in\mathscr{C}\sb\lambda(M,\calN,R_0(\varphi),s\min\{\nu,\lambda^2-m^2\})$.
We also deduce that $\gamma\sb\epsilon$ converges uniformly to $\gamma\sb0$.

Now we consider the term
\[
 \mu\sb\epsilon(r)
=
2
\Big(
n+16\lambda^2 r^2+8r\varphi\sb\epsilon'
\Big)^{1/2},
\]
rewritten as
\[
 \mu\sb\epsilon(r)
=
2r\rho_\epsilon(r),
\]
with
\[
\rho\sb\epsilon(r)= \Big(
\frac{n}{r^2}+16\lambda^2 +8\frac{\varphi\sb\epsilon'}{r}
\Big)^{1/2}.
\]
Due to the local uniform convergence of $\varphi\sb\epsilon'$ to $\varphi\sb0'=\varphi$, $\rho\sb\epsilon$ converges locally uniformly to $\rho\sb0$. Since $\varphi\sb\epsilon'$ is bounded at infinity uniformly in $\epsilon$, $\rho\sb\epsilon-\rho\sb0$ is small at infinity uniformly in $\epsilon$. Hence $\rho\sb\epsilon-\rho\sb0$ is smaller than
$\rho\sb0$
(which is bounded from below by $4\abs{\lambda}$)
multiplied by an arbitrarily small constant,
uniformly in $\epsilon$. We thus obtain
from \eqref{ineq-2}
that for any $\kappa'\in(0,\kappa)$ there exists $\epsilon_1=\epsilon_1(s,\kappa')\in(0,\epsilon_0(s))$ such that, for any $\epsilon\in(0,\epsilon_1)$
and
$v\in H^1\sb 0(\varOmega_R,\C^N)$,
with $R=\max(R_0(\varphi),R_1(\varphi,V))$,
one has
\begin{equation*}
\norm{
\mu_\epsilon V v}
\le
\kappa
\left(
\norm{\bm\nabla v}^2
+
\norm{\gamma v}^2
\right)^{1/2}
\le
\kappa'
\left(
\norm{\bm\nabla v}^2
+
\norm{\gamma_\epsilon v}^2
\right)^{1/2}.
\end{equation*}
This allows to conclude as in the previous step (using the previous step instead of
{\bf Step 0}),
proving that
\begin{eqnarray}
(1-\kappa')
\left(
\norm{\bm\nabla(e^{\varphi\sb\epsilon} u)}^2
+
2\norm{(r\varphi\sb\epsilon')^{1/2}\p\sb r(e^{\varphi\sb\epsilon} u)}^2
+\norm{\gamma e^{\varphi\sb\epsilon} u}^2
\right)^{1/2}
\nonumber
\\
\qquad
\le
\norm{\mu e\sp{\varphi\sb\epsilon} (D\sb m+V-\lambda)u},
\nonumber
\end{eqnarray}
with $R=\max(R_0(\varphi),R_1(\varphi,V))$
independent of $\epsilon\in(0,\epsilon\sb 1)$.
Using the same reasoning as above,
we conclude that this inequality is also satisfied by
$\varphi\sb 0=\varphi$.
Finally, since $\kappa'\lesssim\kappa$
could be chosen arbitrarily close to $\kappa$,
we also have
\[
(1-\kappa)
\left(
\norm{\bm\nabla(e^{\varphi} u)}^2
+
2\norm{(r\varphi')^{1/2}\p\sb r(e^{\varphi} u)}^2
+\norm{\gamma e^{\varphi} u}^2
\right)^{\frac 1 2}
\le
\norm{\mu e\sp{\varphi} (D\sb m+V-\lambda)u}.
\]

\item
Assume that
$\varphi\in\mathscr{C}\sb\lambda(M,\calN,\rho,\nu)$
is such that
$\varphi'$ is unbounded at infinity
(this implies that $\varphi''\ge 0$).
It follows from Lemma~\ref{lemma-big-half}
that there is
$R_0=R_0(M,\rho,m,n,\lambda)$,
$R_0\ge\rho$,
such that $\varphi$ satisfies the inequality
\eqref{f-ge-half}
for $r\ge R_0$.

We approximate $\varphi\sb 0:=\varphi$ by
$\varphi\sb\epsilon\in C\sp 2(\R\sb{+})$,
$\epsilon\in(0,1)$:
\[
\varphi\sb\epsilon(r)
=\int\sb{1}\sp r\frac{\varphi\sb 0'(t)}{1+\epsilon\varphi\sb 0'(t)}\,dt.
\]
Then $\varphi\sb\epsilon$ are monotonically increasing,
satisfy $\sup\sb{r\ge 0}\varphi\sb\epsilon'(r)<\infty$,
and, for each $r\ge\rho$,
$\varphi\sb\epsilon(r)\nearrow\varphi\sb 0(r)$,
$\varphi\sb\epsilon'(r)\nearrow\varphi\sb 0'(r)$.
Moreover, for $r\ge\rho$,
the following inequalities hold:
\[
0<\varphi\sb\epsilon'(r)
=
\frac{\varphi\sb 0'(r)}{1+\epsilon\varphi\sb 0'(r)}
\le \calN r,
\]

\[
\varphi\sb\epsilon''(r)
=\frac{\varphi\sb 0''(r)}{1+\epsilon\varphi\sb 0'(r)}
-\frac{\epsilon\varphi\sb 0'(r)\varphi\sb 0''(r)}
{(1+\epsilon\varphi\sb 0'(r))\sp 2}
=\frac{\varphi\sb 0''(r)}{(1+\epsilon\varphi\sb 0'(r))\sp 2}
\ge 0,
\]

\[
\frac{\varphi\sb\epsilon''(r)}{\varphi\sb\epsilon'(r)}
=
\frac{\varphi\sb 0''(r)}
{\varphi\sb 0'(r)}
\frac{1}{1+\epsilon\varphi\sb 0'(r)},
\qquad
\Abs{\frac{r\varphi\sb\epsilon''(r)}{\varphi\sb\epsilon'(r)}}
\le M.
\]
One has
\begin{align*}
(\varphi\sb\epsilon'(r))^2+2r \varphi\sb\epsilon'(r)\varphi\sb\epsilon''(r)
&=\frac{\varphi\sb 0'(r)^2+2r\varphi\sb 0'(r)\varphi\sb 0''(r)}{(1+\epsilon\varphi\sb 0'(r))\sp 3}
+\frac{\epsilon\varphi\sb 0'(r)^3}{(1+\epsilon\varphi\sb 0'(r))\sp 3};
\end{align*}
in the case $\epsilon\varphi'<2^{1/3}-1$, one concludes that
\begin{eqnarray}
\lambda^2-m^2
+
(\varphi\sb\epsilon'(r))^2+2r \varphi\sb\epsilon'(r)\varphi\sb\epsilon''(r)
\ge
\lambda^2-m^2
+
\frac{\nu-(\lambda^2-m^2)}{2}
\ge
\frac{\lambda^2-m^2+\nu}{2},
\nonumber
\end{eqnarray}
while in the case $\epsilon\varphi'\ge 2^{1/3}-1$, one deduces
\begin{align*}
\lambda^2-m^2
+
(\varphi\sb\epsilon'(r))^2+2r \varphi\sb\epsilon'(r)\varphi\sb\epsilon''(r)
\ge
\lambda^2-m^2
-
\frac{\abs{\nu-(\lambda^2-m^2)}}{2}
+\frac{(2^{\frac 1 3}-1)^3}{2\epsilon^2},
\end{align*}
which is also larger than
$\frac{\lambda^2-m^2+\nu}{2}$ provided that
$\epsilon\in(0,\epsilon_0)$ with $\epsilon_0\in(0,1)$ sufficiently small.
One concludes that
$\varphi\sb\epsilon\in\mathscr{C}\sb\lambda(M,\calN,\rho,
\frac{\lambda^2-m^2+\nu}{2})$
for all $\epsilon\in(0,\epsilon_0)$
and uses the result and the ideas
from the previous step
to prove the inequality
\[
(1-\kappa)
\left(
\norm{\bm\nabla(e^{\varphi\sb\epsilon} u)}^2
+
2\norm{(r\varphi\sb\epsilon')^{1/2}\p\sb r(e^{\varphi\sb\epsilon} u)}^2
+\norm{\gamma e^{\varphi\sb\epsilon} u}^2
\right)^{\frac 1 2}
\le
\norm{\mu e\sp{\varphi\sb\epsilon} (D\sb m+V-\lambda)u},
\qquad \forall u\in H^1_0(\varOmega\sb R,\C^N),
\]
with $R=\max(R_0(\varphi),R_1(\varphi,V))$
independent of $\epsilon\in(0,\epsilon_0)$,
and applies the Fatou lemma
and the dominated convergence theorem
to the above inequality
to prove \eqref{euw-u-1}.
\qedhere
\end{proofstep}
\end{proof}
Theorem~\ref{theorem-Carleman}~\itref{theorem-Carleman-i}
follows from
Lemma~\ref{lemma-asdf}~\itref{lemma-asdf-ii}.

\bigskip

Now we proceed to the proof
of Theorem~\ref{theorem-Carleman}~\itref{theorem-Carleman-ii}.
Let
$\varphi\in\mathscr{C}\sb\lambda(M,\calN,\rho,\nu)$,
with $\varphi''(r)\ge 0$ for $r\ge\rho$.
Then, by Lemma~\ref{lemma-big-half},
the inequality \eqref{f-ge-half}
is satisfied for $r\ge R_0=R_0(M,\rho,m,n,\lambda)$,
independent of a particular representative
$\varphi\in\mathscr{C}\sb\lambda(M,\calN,\rho,\nu)$.
At the same time,
using \eqref{Eq:VSmallInfinity} with $f(r)=\varphi'(r)$
(note that
$\sup\sb{r>0}{f(r)}/{\langle r\rangle}\le\calN<\infty$,
in agreement with the assumptions on the function $f$
which appears in \eqref{Eq:VSmallInfinity}),
and taking into account that $\varphi''\ge 0$,
we see that the inequality \eqref{ineq-0}
is satisfied for $r\ge R_1=R_1(\varphi,V)$
if one chooses $R_1=R_1(\varphi,V)=R_2(V)$,
which would not depend on $\varphi$.
As follows from
Lemma~\ref{lemma-asdf},
the inequality \eqref{carleman-compact-u-gamma-v-2}
is satisfied for
$u\in H^1_0(\varOmega\sb R,\C^N)$
with
\[
R=\max(R_0(\varphi),R_1(\varphi,V))
=\max(R_0(M,\rho,m,n,\lambda),R_2(V))
\]
independent of a particular $\varphi$.
This finishes the proof of
Theorem~\ref{theorem-Carleman}.
\end{proof}

\section{Exponential decay}
\label{sect-exp}

\subsection{Decay properties of solitary waves}
\label{sect-waves}

In this section, we provide the proof of
Theorem~\ref{theorem-properties};
it will follow
from Lemmata~\ref{lemma-exp-decay}--\ref{lemma-exp-decay-assumption}.

Let us start with a simple technical inequality adapted from
\cite[Theorem 1]{MR880983}.
\begin{lemma}\label{lemma-bb}
Let $\lambda\in (-m;m)$ and $R\sb 0>0$.
Let $\varphi:[R\sb 0,+\infty)\to\R$ be
a monotonically increasing $C^1$ function such that
\begin{equation}\label{alpha}
\limsup\sb{r\to\infty}\varphi'(r)<\sqrt{m\sp 2-\lambda\sp 2}.
\end{equation}
Assume that $V:\,\varOmega\sb{R\sb 0}\mapsto \End(\C^N)$
satisfies the following condition:
\begin{equation}\label{vu-du}
\forall \varepsilon >0
\quad\exists R>0\quad
\forall u\in H\sp 1\sb{\mathrm{comp}}(\varOmega\sb R,\C^N)
\quad
\mbox{implies that}
\quad
\|V u\|
\leq \varepsilon \|u\|_{H^1}.
\end{equation}
Then there are constants $c,\,R$
such that
\begin{equation}\label{e-f-c}
\|e\sp{\varphi}u\|\sb{H\sp 1}\leq c\|e\sp{\varphi}(D\sb m+V-\lambda)u\|,
\ \quad
\forall
u\in L\sp 2(\varOmega\sb R,\C^N)\cap
H\sp 1\sb{0,\mathrm{loc}}(\varOmega\sb R,\C^N).
\end{equation}
\end{lemma}

\begin{remark}\label{rem:expdecay}
Since
$\norm{D\sb m u}=\norm{(-\Delta^2+m^2)^{1/2}u}\ge m\norm{u}$,
the assumption \eqref{vu-du}
is weaker than $\norm{V(x)}\sb{\End(\C^N)}\to 0$ as $\abs{x}\to\infty$.
\end{remark}

\begin{proof}
Without loss of generality,
we may assume that $R\sb 0=1$. The general case
follows by the same ideas and can be recovered by rescaling.

We deduce from Lemma~\ref{Lem:LemmaBG3} that for any $v\in H\sp
1\sb{\mathrm{comp}}(\R\sp n,\C^N)$,
\begin{equation}\label{bg-smth-1}
\Re
\langle
(D\sb m-\jj \bm\alpha\cdot\bm\nabla\varphi+\lambda)v,
(D\sb m\sp\varphi-\lambda)v\rangle
=\norm{\bm\nabla v}\sp 2+\langle v,[m\sp 2-\lambda\sp 2-(\bm\nabla\varphi)\sp
2]v\rangle,
\end{equation}
where
$
D\sb m\sp\varphi
=e\sp\varphi\circ D\sb m \circ e\sp{-\varphi}
=D\sb m+ \jj \bm\alpha\cdot\bm\nabla\varphi
$
was introduced in \eqref{def-dm-varphi}.
For any $\varepsilon>0$, we have
\begin{eqnarray}\label{bg-smth-2}
&&
\Re
\langle
(D\sb m-\jj \bm\alpha\cdot\bm\nabla\varphi+\lambda)v,
(D\sb m\sp\varphi-\lambda)v\rangle
\nonumber\\[1ex]
&&
\leq
\frac\varepsilon2
\|(D\sb m-\jj \bm\alpha\cdot\bm\nabla\varphi+\lambda)v\|\sp 2+
\frac{1}{2\varepsilon}\|(D\sb m\sp\varphi-\lambda)v\|\sp 2
\\[1ex]
&&
\leq
\frac{3\varepsilon}{2}
\|\bm\nabla v\|\sp 2
+\frac{3\varepsilon}{2}\|\abs{\bm\nabla\varphi}v\|\sp 2
+\frac{3\varepsilon}{2}\| (\beta m +\lambda)v\|\sp 2+
\frac{1}{2\varepsilon}\|(D\sb m\sp\varphi-\lambda)v\|\sp 2,
\nonumber
\end{eqnarray}
where we took into account that
\[
\norm{(D\sb m-\jj\bm\alpha\cdot\bm\nabla\varphi+\lambda)u}^2
\le
(\norm{D\sb 0 u}+\norm{(\bm\alpha\cdot\bm\nabla\varphi)u}
+\norm{(\beta m+\lambda)u})^2
\]
\[
\le
3
(\norm{D\sb 0 u}^2+\norm{(\bm\alpha\cdot\bm\nabla\varphi)u}^2
+\norm{(\beta m+\lambda)u}^2)
\]
and
$
\norm{\bm\alpha\cdot\bm\nabla\varphi}\sb{\End(\C^N)}
=\abs{\bm\nabla\varphi}
$.
We deduce from \eqref{bg-smth-1} and \eqref{bg-smth-2}
that
\[
\Big(1-\frac{3\varepsilon}{2}\Big) \norm{\bm\nabla v}\sp 2
+
\Big\langle v,
\Big[m\sp 2-\lambda\sp 2
-\big(1+\frac{3\varepsilon}{2}\big)(\bm\nabla\varphi)\sp
2-\frac{3\varepsilon}{2}(\beta m+\lambda)\sp 2
\Big]v
\Big\rangle
\leq
\frac{1}{2\varepsilon}\|(D\sb m\sp\varphi-\lambda)v\|\sp 2.
\]
Thus, due to the assumption \eqref{alpha},
there exist $c>0$ and $R>0$
such that
\begin{equation}\label{h1-bound-noV}
\|v\|\sb{H\sp 1}\leq c\|(D\sb m\sp\varphi-\lambda)v\|,
\qquad \forall v\in H\sp 1\sb{\mathrm{comp}}(\varOmega\sb R,\C^N).
\end{equation}
Then, due to the assumption on $V$, there exist $c>0$ and
$R>0$ such that
\[
\|v\|\sb{H\sp 1}\leq c\|(D\sb m\sp\varphi+V-\lambda)v\|,
\qquad \forall v\in H\sp 1\sb{\mathrm{comp}}(\varOmega\sb R,\C^N).
\]
Substituting $e\sp\varphi u$
in place of $v$
and using the identity
$(D\sb m\sp\varphi+V)(e\sp\varphi u)=e\sp\varphi(D\sb m+V)u$,
we conclude that
there exist $c>0$ and $R>0$
such that
\begin{equation}\label{h1-bound}
\|e\sp{\varphi}u\|\sb{H\sp 1}\leq c\|e\sp{\varphi}(D\sb m+V-\lambda)u\|,
\qquad \forall u\in H\sp 1\sb{\mathrm{comp}}(\varOmega\sb R,\C^N).
\end{equation}

Let us extend \eqref{h1-bound}
to functions $u\in H\sp 1\sb 0(\varOmega\sb{R},\C^N)$
which are no longer compactly supported.

First, we consider the case when $\varphi(r)$ is bounded.
Let
\[
\eta\in C^\infty\sb{\mathrm{comp}}([-2,2]),
\qquad
0\le\eta\le 1,
\qquad
\eta\at{[-1,1]}\equiv 1,
\]
and define
$\eta\sb j(x)=\eta(x/j)$.
Let $L\sp 2(\varOmega\sb R,\C^N)\cap
H\sp 1\sb{\mathrm{loc}}(\varOmega\sb R,\C^N)$
with
$\supp u\subset\varOmega\sb{R}$,
and define
\[
w\sb j=\eta\sb j u\in H\sp 1\sb{\mathrm{comp}}(\varOmega\sb{R},\C^N).
\]
Using the identity
$(D\sb m-\lambda)w\sb j=\eta\sb j(D\sb m-\lambda) u
-\jj(\alpha\nabla \eta\sb j)u$,
one has
\begin{equation}\label{dct0}
\norm{
e\sp\varphi w\sb j}
\le
c\norm{e\sp\varphi\eta\sb j(D\sb m+V-\lambda)u}
+
c\norm{
e\sp\varphi(\alpha\nabla \eta\sb j)u}.
\end{equation}
The second term in the right-hand side tends to
zero as $j\to\infty$.
Indeed,
according to our assumptions, $u$ is in $L\sp 2$ and $\varphi$
is bounded,
hence $e\sp\varphi u\in L\sp 2$ and
\[
\lim\sb{j\to\infty}
\norm{\chi\sb{[j,2j]}(\abs{x})e\sp\varphi  u}\sb{L\sp 2}=0,
\]
while
\[
\norm{\nabla\eta\sb{j}}\sb{L^\infty}
\le
\norm{\nabla\eta\sb j}\sb{L^\infty}
\norm{\chi\sb{[-2j,2j]}}\sb{L^\infty}
=\frac{1}{j}\norm{\nabla\eta}\sb{L^\infty}
\norm{\chi\sb{[-2j,2j]}}\sb{L^\infty}
\]
is bounded.
Applying the dominated convergence theorem
to the first term in the right-hand side of \eqref{dct0}
and the Fatou lemma to the left-hand side,
we conclude that
\[
\norm{
e\sp\varphi u}
\le c\norm{e\sp\varphi (D\sb m+V-\lambda)u},
\qquad \forall u\in H\sp 1\sb 0(\varOmega\sb{R},\C^N).
\]

\bigskip

Unbounded $\varphi$
are considered precisely as in~\cite[Theorem 3]{MR880983},
which in turn follows the approach of~\cite{MR685028}.
We already presented the details in the proof
of Theorem~\ref{theorem-Carleman}.
\end{proof}

\begin{lemma}\label{lemma-exp-decay}
Let $n\ge 1$.
If $n=1$, assume that
$f$ is measurable and bounded on bounded sets,
with $\lim_0f=0$;
if instead $n>1$, assume that there exists $C>0$ such that
\[
|f(s)|\leq C\abs{s}^k,
\qquad \forall s\in\R,
\qquad
\begin{cases}
0<k\le 1/(n-2), \quad n\geq 3,
\\[1ex]
k>0,\quad n=2.
\end{cases}
\]
Let $\omega\in(-m,m)$,
and assume that
$\phi\sb\omega\in H^1(\R\sp n,\C^N)$
is a solution to \eqref{nld-stationary}.
Then for any
$\mu<\sqrt{m\sp 2-\omega\sp 2}$
one has
\[
e\sp{\mu\langle r\rangle}\phi\sb\omega
\in H\sp1(\R\sp n,\C^N).
\]
\end{lemma}

\begin{proof}
For the sake of completeness,
we choose to provide a proof of the above lemma.
One can use the Combes--Thomas method, see~\cite{MR1785381}.
For consistency, let us use Lemma~\ref{lemma-bb}.
The solitary wave profile $\phi\sb\omega$ satisfies
\begin{equation}\label{dvk-nlds}
\omega\phi\sb\omega
=
D\sb m\phi\sb\omega
-f(\phi\sb\omega\sp\ast\beta\phi\sb\omega)\beta\phi\sb\omega.
\end{equation}
The assumptions on $f$ (combined with H\"older and Sobolev
inequalities)
are to ensure that $V=f(\phi\sb\omega\sp\ast\beta\phi\sb\omega)\beta$
satisfies the assumption \eqref{vu-du}
of Lemma~\ref{lemma-bb}.

Pick
$\mu\in(0,\sqrt{m\sp 2-\omega\sp 2})$.
Let $\eta\in C^\infty(\R\sb{+})$, $\supp\eta\in(1,+\infty)$,
$0\le\eta\le 1$ and
$\eta\at{[2,+\infty)}\equiv 1$,
and define
$\eta\sb j(x)=\eta(x/j)$.
By Lemma~\ref{lemma-bb} with
$\phi(r)=\mu r$, $u=\eta\sb j \phi\sb\omega$, and
$V=f(\phi\sb\omega\sp\ast\beta\phi\sb\omega)\beta$, we have
$
 \|e\sp{\mu r}\eta\sb j\phi\sb\omega\|\sb{H\sp 1}\leq
c\|e\sp{\mu r}\bm\alpha\cdot\nabla\eta\sb j\phi\sb\omega\|,
$
which provides the conclusion.
\end{proof}

In the previous proof, the assumptions on $\phi_\omega$ and $f$
were made to ensure that $V=f(\phi\sb\omega\sp\ast\beta\phi\sb\omega)\beta$
for large $x$
is small
compared to the Dirac operator.
If for instance we consider (as in e.g. \cite{MR847126})
\begin{equation}\label{as-above}
\phi\sb\omega
=\begin{bmatrix}
  v(r)\bm{n}\sb 1\\
  u(r)\,(\bm{e}\sb r\cdot\bm\sigma)\,\bm{n}\sb 1
 \end{bmatrix},
\end{equation}
where $u$ and $v$ are real-valued and
$
 \bm{n}\sb 1=
[1,0,\dots,0]\sp{t}
\in\C^{N/2},
$
$
\bm{e}\sb r=\frac{x}{r}\in\R\sp n,
$
$
\bm\sigma=(\sigma\sb\jmath)\sb{1\le\jmath\le n},
$
then we have the following statement by repeating the proof
of the one-dimensional case away from the origin.
\begin{lemma}\label{lemma-exp-decay-soler}
Let $n\ge 1$.
Assume that $f$ is measurable and bounded on bounded sets, with
$\lim_0 f=0$.
Let $\phi\sb\omega$ be a solution to \eqref{nld-stationary}
of the form \eqref{as-above},
with $u(\abs{x})$ and $v(\abs{x})$ in $H^1(\R^n,\C^{N/2})$
considered as functions of $x\in\R^n$.
Then for any
$\mu<\sqrt{m\sp 2-\omega\sp 2}$
one has
$
e\sp{\mu \langle r\rangle}\phi\sb\omega
\in H\sp1(\R\sp n,\C^N).
$
\end{lemma}

As the matter of fact,
the exponential decay holds for any solitary wave
regardless of whether it is of the form \eqref{as-above},
as long as one knows
that it becomes small at infinity:

\begin{lemma}\label{lemma-exp-decay-assumption}
Let $n\ge 1$. Assume that $f$ is measurable and bounded on bounded sets,
with $\lim_0 f=0$.
Suppose that $\omega\in(-m,m)$ and that
\[
\phi\sb\omega\in
L^2(\R^n,\C^N),
\qquad
\lim\sb{R\to\infty}\norm{\phi\sb\omega}_{L^\infty(\varOmega_R,\C^N)}=0
\]
is a solution to \eqref{nld-stationary}.
Then for any
$\mu<\sqrt{m\sp 2-\omega\sp 2}$
one has
$
e\sp{\mu \langle r\rangle}\phi\sb\omega
\in H\sp1(\R\sp n,\C^N).
$
\end{lemma}

\begin{proof}
The proof reduces to proving that
$V=f(\phi\sb\omega\sp\ast\beta\phi\sb\omega)\beta$ is small
at infinity compared to the Dirac operator.
This, in turn, follows from the assumptions
$\lim_0 f=0$
and
$\lim\sb{R\to\infty}\norm{\phi\sb\omega}_{L^\infty(\varOmega_R,\C^N)}=0$

\end{proof}

\subsection{Decay of embedded eigenstates before the embedded threshold}

\label{sect-embed}

In this section, we prove Theorem~\ref{theorem-embed}~\itref{theorem-embed-iib};
its proof follows from
Lemma~\ref{lemma-embed-exp-decay} below.

Let $n\ge 1$.
Let $V:\,\R\sp n\to\End(\C^N)$ be measurable,
and assume that
for any $\epsilon>0$ there exists $R>0$
such that
\begin{equation}\label{v-small}
\norm{\langle r\rangle V v}\leq \epsilon \norm{v}_{H^1},
\qquad \forall v\in H\sp 1\sb{\mathrm{comp}}(\varOmega\sb R,\C^{N}).
\end{equation}

\begin{remark}
Note that
the operator $\eubL(\omega)$
in \eqref{def-ll}
is such that \eqref{v-small} is satisfied,
due to the exponential spatial decay of $\phi\sb\omega$
(cf. Lemmata~\ref{lemma-exp-decay}--\ref{lemma-exp-decay-assumption}).
\end{remark}

\begin{lemma}\label{lemma-embed-exp-decay}
Let $J\in\End(\C^N)$ be skew-adjoint and invertible,
such that $J^2=-I\sb{\C^N}$, $[J,D\sb m]=0$.
Let $\lambda\in\sigma\sb{\mathrm{p}}(JL(\omega))$ satisfy
$\lambda\in\jj \R$,
$m-\abs{\omega}<\abs{\lambda}<m+\abs{\omega}$.
Then, for any $\mu<\sqrt{m^2-(|\lambda|-|\omega|)^2}$,
an eigenfunction $\zeta$ corresponding to $\lambda$
satisfies
$
e\sp{\mu\langle r\rangle}\zeta
\in H\sp1(\R\sp n,\C^{N}).
$
\end{lemma}

\begin{proof}
Let
$M,\,\calN,\,\rho\ge 1$, $\nu>0$.
Assume that $\varphi\in C\sp 2(\R\sb{+})$ satisfies
(cf. Definition~\ref{definition-c} and \eqref{alpha})
\begin{enumerate}
\item
$0<\varphi'\le \sqrt{m^2-(|\lambda|-|\omega|)^2}$, $\ \forall r\ge\rho$;
\item
$\limsup\sb{r\to\infty}\varphi'(r)<\sqrt{m\sp 2-(\abs{\lambda}-\abs{\omega})\sp 2}$;
\item
$r\abs{\varphi''}\le M\varphi'$, $\ \forall r\ge\rho$ ;
\item
$\Lambda\sb{-}\sp 2-m\sp 2+\varphi'\sp 2+2r\varphi'\varphi''\ge \nu$,
$\ \forall r\ge\rho$.
\end{enumerate}
Let
$\Pi\sp\pm=\frac 1 2(1\mp\jj J)$ as in \eqref{def-pi-pm}.
Assume that $ \jj \lambda$ is of the same sign as $\omega$
(the other case is treated verbatim by exchanging
the treatment of $\zeta\sp\pm$);
then
$\omega+\jj \lambda$ is outside the spectral gap of $D\sb m$
while
$\omega-\jj \lambda$ is inside.
So,
$\abs{\omega+\jj\lambda}>m$,
and Lemma~\ref{lemma-asdf}~\itref{lemma-asdf-i} yields
\[
\norm{\bm\nabla(e^\varphi u)}^2
+
\norm{\gamma e^\varphi u}^2
\le\norm{\mu e^\varphi(D\sb m-(\omega+\jj\lambda))u}^2,
\qquad \forall u\in H\sp 1\sb 0(\varOmega\sb{R_0},\C^N).
\]
Since
$\abs{\omega-\jj\lambda}<m$
and
$\abs{\varphi'}<\sqrt{m^2-(\omega-\jj\lambda)^2}$,
we apply \eqref{h1-bound}:
\[
\|e\sp{\varphi}u\|\sb{H\sp 1}\leq c\|e\sp{\varphi}
(D\sb m-(\omega-\jj\lambda))u\|,
\qquad \forall u\in H\sp 1\sb{\mathrm{comp}}(\varOmega\sb R,\C^N).
\]
Summing up the above inequalities
(applied to $\Pi\sp{-}u$ and $\Pi\sp{+}u$, respectively,
with $\Pi\sp\pm$ from \eqref{def-pi-pm}),
we have:
\begin{eqnarray}
\norm{\bm\nabla(e^\varphi\Pi\sp{-}u)}^2
+\norm{\gamma e^\varphi\Pi\sp{-}u}^2
+\|e^\varphi\Pi\sp{+}u\|\sb{H\sp 1}^2
\leq
\norm{\mu e^\varphi(D\sb m-\omega-\jj\lambda)\Pi\sp{-}u}^2
+c^2\|e\sp\varphi(D\sb m-\omega+\jj\lambda)\Pi\sp{+}u\|^2
\nonumber
\\[2ex]
\leq
\norm{\mu e^\varphi\Pi\sp{-}(D\sb m-\omega-J^{-1}\lambda)u}^2
+
c^2\|e\sp\varphi\Pi\sp{+}(D\sb m-\omega-J^{-1}\lambda)u\|^2,
\nonumber
\end{eqnarray}
valid for all
$u\in H\sp 1\sb{\mathrm{comp}}(\varOmega\sb R,\C^{N})$,
which we rewrite as
\[
\|e^\varphi u\|\sb{H\sp 1}
\leq
\langle c\rangle
\|e^\varphi(\Pi\sp{+}+\mu\Pi\sp{-})
\big((D\sb m-\omega)-J^{-1}\lambda\big)u\|.
\]
Increasing $R$ if necessary,
we use \eqref{v-small} to arrive at
\[
\|e^\varphi u\|\sb{H\sp 1}
\leq 2 \langle c\rangle\|e^\varphi(\Pi\sp{+}+\mu\Pi\sp{-})
\big(J(D\sb m-\omega+V)-\lambda\big)u\|,
\ \quad
\forall u\in H\sp 1\sb{\mathrm{comp}}(\varOmega\sb R,\C^{N}).
\]
The extension of the above estimate
to $u\in H\sp 1\sb0(\varOmega\sb R,\C^{N})$ can be done as in the proof of Lemma~\ref{lemma-asdf}
up to {\bf Step 2}.
We then conclude as in the proof of Lemma~\ref{lemma-exp-decay} by applying this estimate to a smooth localization of an eigenvector to the region $\varOmega\sb R$.
\end{proof}

\section{Bifurcations of eigenvalues from the essential spectrum}
\label{sect-b}

In this section we prove Theorem~\ref{theorem-b}:
the case $\abs{\lambda\sb{0}}<m+\abs{\omega\sb{0}}$
follows from
Lemma~\ref{lemma-between-thresholds} below, while
Theorem~\ref{theorem-b}
for the case $\abs{\lambda\sb{0}}>m+\abs{\omega\sb{0}}$
follows from
Lemma~\ref{lemma-beyond-thresholds}.
We start with some basic results on the
limiting absorption principle.

\subsection{Results on the limiting absorption principle
for the Dirac operator}

The limiting absorption principle for the Dirac
operator was studied
in~\cite{MR0320547,MR1162140,MR1241704,MR1743451,MR2718928}.
We reformulate it here
since we need a version
valid for a spectral parameter from a non-compact set:

\begin{lemma}[Limiting absorption principle for the Dirac operator]
\label{lemma-lap-dirac}
Let $s>1/2$, $\delta>0$, and $m\ge 0$.
There exists $C_0=C_0(s,\delta,m)<\infty$
(locally bounded in $s$, $\delta$, and $m$)
such that
for all
$z\in\C\setminus((-\infty,-m]\cup[m,+\infty))$
with $\abs{z\sp 2-m\sp 2}\ge\delta$
one has
\begin{equation}
\label{agmon-yamada}
\norm{u}\sb{L\sp2\sb{-s}}
\le C_0(s,\delta,m)\norm{(D\sb m-z)u}\sb{L\sp 2\sb s},
\qquad \forall
u\in L\sp 2(\R\sp n,\C^N).
\end{equation}
Let $s>1/2$, $m\ge 0$, and
$z\in\C\setminus((-\infty,-m]\cup[m,+\infty))$.
There exists $C_1=C_1(s,z,m)<\infty$
(locally bounded in $s$, $z$, and $m$)
such that
\begin{equation}
\label{agmon-yamada-2}
\norm{u}\sb{H\sp1\sb{-s}}
\le C_1(s,z,m)\norm{(D\sb m-z)u}\sb{L\sp 2\sb s},
\qquad \forall u\in L\sp 2(\R\sp n,\C^N).
\end{equation}
\end{lemma}

\begin{proof}
By~\cite[Remark 2 in Appendix A]{MR0397194},
for any $s>1/2$ and $\delta>0$
there is $C_{s,\delta}<\infty$ such that
for all $v\in H^2(\R^n)$
and $\zeta\in\C$,
$\abs{\zeta}\ge\delta$,
one has
\begin{equation}\label{agmon-remark-2}
(\abs{\zeta}+1)^{\frac{1-k}{2}}
\norm{v}_{H^k_{-s}}
\le
C_{s,\delta}
\norm{(-\Delta-\zeta)v}\sb{L^2_s},
\qquad
0\le k\le 2.
\end{equation}
We will apply this inequality
to vector-valued functions $v\in H^2(\R^n,\C^N)$.

Let $u\in L\sp 2(\R^n,\C^N)$
and
$z\in\C\setminus((-\infty,-m]\cup[m,\infty))$.
Without loss of generality,
we can assume that
$(D_m-z)u\in L^2_s(\R^n,\C^N)$
(or else there is nothing to prove);
it then follows that $u\in H^1(\R^n,\C^N)$
and $v:=(D_m+z)^{-1}u\in H^2(\R^n,\C^N)$.
One has:
\begin{equation}\label{rhs-rhs}
\norm{u}_{H^{k-1}_{-s}}
=
\norm{(D_m+z)v}_{H^{k-1}_{-s}}
\le
C(s)\norm{v}_{H^{k}_{-s}}
+
(m+\abs{z})\norm{v}_{H^{k-1}_{-s}},
\end{equation}
where $C(s)<\infty$ depends on $s$ only.
Applying \eqref{agmon-remark-2}
with $\zeta=z^2-m^2$,
$\abs{\zeta}\ge\delta>0$,
to the right-hand side of
\eqref{rhs-rhs}, we have:
\[
\norm{u}_{H^{k-1}_{-s}}
\le
\left(
C(s)(\abs{\zeta}+1)^{\frac{k-1}{2}}
+
(m+\abs{z})(\abs{\zeta}+1)^{\frac{k-2}{2}}\right)
C_{s,\delta}\norm{(-\Delta-\zeta)u}\sb{L^2_s},
\]
for $1\le k\le 2$.
Taking $k=1$ and $k=2$
and using the identity
$(-\Delta-\zeta)v=(D_m-z)u$,
we arrive at the inequalities
\eqref{agmon-yamada} and \eqref{agmon-yamada-2}.
\end{proof}

We also need the following Hardy-type inequality,
along the lines of \cite[Appendix B]{MR0397194}.

\begin{lemma}\label{Lem:BerthierGeorgescuLogarithm}
For any $s>1/2$ and $z\in\R\setminus[-m,m]$,
there is $C_2=C_2(s,z,m)<\infty$
(locally bounded in $s$ and $z$)
such that
\begin{equation}\label{bg-th2}
\norm{u}\sb{H^1_s}
\le C_2\norm{(D\sb m-z)u}\sb{L^2_{s+1}},
\qquad \forall u\in L^2 (\R^n,\C^N)\cap H^1_{\mathrm{loc}}(\R^n,\C^N).
\end{equation}
\end{lemma}

\begin{proof}
While the result by Berthier and
Georgescu is stated in the three-dimensional case,
a careful look at the
proof shows that it is independent of the dimension,
being based on~\cite[Appendix B]{MR0397194},
which treats any dimension.

For the sake of completeness, we provide the proof.
Fix $s>1/2$.
We use Theorem~\ref{theorem-Carleman}
with $\varphi(r)=s\log\langle r\rangle$;
this gives some $R\sb s<\infty$ and $C(s,z)<\infty$
such that for any $R\ge R\sb s$
and for any $v\in H\sp 1(\R\sp n,\C^N)$,
$\supp v\subset\varOmega\sb{R}$,
which satisfies
\[
\langle r \rangle^{s+1}(D\sb m-z)v
=\langle r \rangle e\sp\varphi (D\sb m-z)v
\in L\sp 2(\R\sp n,\C^N)
\]
one has
$\langle r\rangle^s v
=e\sp{\varphi}v\in L\sp 2(\R\sp n,\C^N)$ and moreover
\begin{equation}\label{same-bound}
\norm{
\langle r \rangle^s v}
\le
C(s,z)\norm{
\langle r \rangle^{s+1}
(D\sb m-z)v}.
\end{equation}
Let $\eta\in C^\infty(\R^n)$
be such that
$\supp\eta\subset\varOmega\sb R$
and
$\eta\at{\varOmega\sb{R+1}}=1$
for some $R>R_s$.
Let $u\in L^2(\R^n,\C^N)\cap H^1\sb{\mathrm{loc}}(\R^n,\C^N)$.
Without loss of generality,
we may assume that
$\norm{(D_m-z)u}_{L^2\sb{s+1}}$ is finite
(or else there is nothing to prove);
then we conclude that $u\in H^1(\R^n,\C^N)$.
Applying \eqref{same-bound} to $\eta u\in H^1(\R^n,\C^N)$,
$\supp\eta u\subset\varOmega_R$,
one has:
\begin{equation}\label{eta-u-1}
\norm{\eta u}\sb{H^1_s}
\le C(s,z)\norm{(D\sb m-z) \eta u}\sb{L^2_{s+1}}.
\end{equation}
At the same time,
since
$\supp(1-\eta)u\subset\mathbb{B}^n_{R+1}$,
we deduce that
\begin{eqnarray}\label{eta-u-2}
\norm{(1-\eta) u}\sb{H^1_s}
\leq
\langle R+1\rangle^{2s}
\norm{(1-\eta) u}\sb{H^1_{-s}}
\le
\langle R+1\rangle^{2s}
C_1(s,z,m)\norm{(D\sb m-z) (1-\eta) u}\sb{L^2_{s}};
\end{eqnarray}
in the last inequality,
we applied \eqref{agmon-yamada-2}.
Using
\eqref{eta-u-1} and \eqref{eta-u-2}, we have:
\begin{eqnarray}\nonumber
\norm{ u}\sb{H^1_s}
&\le&
\norm{\eta u}\sb{H^1_{s}}
+
\norm{(1-\eta)u}\sb{H^1_{s}}
\\[1ex]
&\le&
C(s,z)\norm{(D\sb m-z)\eta u}\sb{L^2_{s+1}}
+
\langle R+1\rangle^{2s}
C_1\norm{(D\sb m-z) (1-\eta) u}\sb{L^2_{s}}
\nonumber
\\[1ex]
&\le&
(C(s,z)+\langle R+1\rangle^{2s}C_1)
\Big(
\norm{(D\sb m-z)u}\sb{L^2_{s+1}}
+\norm{(\bm\alpha\cdot\nabla\eta) u}\sb{L^2_{s+1}}
\Big).
\nonumber
\end{eqnarray}
Due to the compact support of
$\nabla\eta$,
the inequality \eqref{agmon-yamada-2}
shows that the second term in the brackets
in the right-hand side
is dominated by the first term,
which concludes the proof.
\end{proof}

We now consider an extension of this result for values
outside the real line. Such extension is false in full
generality.\footnote{For instance, in the one-dimensional case,
$\int\sb\R |x|^{s-2}\abs{F(x)}^2\,dx$
can not be bounded from above by
$\frac{4}{(s-1)^2} \int\sb\R |x|^s\abs{F'(x)-2F(x)}^2\,dx$
if
no restriction on $F$ is imposed such as $F$ has support away from $-\infty$.
}
The one we obtain is due to a simple commutator estimate.

\begin{lemma}\label{lemma-hardy-complex-gap}
Assume that $\lambda\in\C\setminus\big((-\infty,-m]\cup[m,+\infty)\big)$.
If $s\in\R$ satisfies
$\abs{s}<\dist(\lambda,\sigma(D_m))$,
then
\[
\norm{u}\sb{L^2\sb s}
\le
\frac{1}{\dist\big(\lambda,\sigma(D_m)\big)-\abs{s}}
\norm{(D\sb m-\lambda)u}\sb{L\sp 2\sb{s}},
\qquad \forall u\in L^2\sb{s-1}(\R^n,\C^N).
\]
\end{lemma}

\begin{proof}
First, we notice that
for any $u\in C^\infty\sb{\mathrm{comp}}(\R^n,\C^N)$,
one has
\begin{equation}\label{commutator-bound}
\norm{[\langle r\rangle^s,D_m]u}
=
\norm{(D_0\langle r\rangle^s)u}
=
\Norm{\frac{\bm\alpha\cdot x}{r}s\langle r\rangle^{s-1}u}
\le \abs{s}\norm{u}\sb{L^2_{s-1}};
\end{equation}
note that
$\norm{\bm\alpha\cdot x}\sb{\End(\C^N)}
=\norm{(\bm\alpha\cdot x)(\bm\alpha\cdot x)}\sb{\End(\C^N)}^{1/2}
=\norm{x^2}\sb{\End(\C^N)}^{1/2}=r$.
Using \eqref{commutator-bound}, we compute for such $u$:
\begin{eqnarray}
\norm{\langle r\rangle^s(D_m-\lambda)u}
\ge
\norm{(D_m-\lambda)\big(\langle r\rangle^s u\big)}
-
\norm{[D_m,\langle r\rangle^s]u}
\ge
\norm{(D_m-\lambda)\big(\langle r\rangle^s u\big)}
-
\abs{s}
\norm{u}\sb{L^2_{s-1}}.
\nonumber
\end{eqnarray}
The above inequality shows that
if
$u\in L^2_{s-1}(\R^n,\C^N)$
and
$(D_m-\lambda)u\in L^2_s(\R^n,\C^N)$
(if the latter inclusion were not satisfied then
there would be nothing to prove),
then
$(D_m-\lambda)(\langle r\rangle^s u)\in L^2(\R^n,\C^N)$.
Since $D_m$ is self-adjoint,
one has
$\norm{(D_m-\lambda)^{-1}}=1/\dist(\lambda,\sigma(D_m))$;
therefore,
$\langle r\rangle^s u\in H^1(\R^n,\C^N)$,
and
\begin{eqnarray}
\norm{(D_m-\lambda)u}\sb{L^2_s}
\ge
\dist(\lambda,\sigma(D_m))\norm{\langle r\rangle^s u}
-\abs{s}\norm{u}\sb{L^2_{s-1}}
\ge
\big(
\dist(\lambda,\sigma(D_m))-\abs{s}
\big)
\norm{u}\sb{L^2_s}.
\nonumber
\end{eqnarray}
This concludes the proof.
\end{proof}

\subsection{Bifurcation of eigenvalues before the embedded threshold}
\label{sect-vb-1}

Let us consider bifurcations
from the interval of the imaginary axis
between the embedded thresholds,
proving Theorem~\ref{theorem-b} for the case $\abs{\lambda\sb{0}}<m+\abs{\omega\sb{0}}$.
We will formulate our results
for the operator family
$
L(\omega)=D\sb m-\omega+V(\omega),
$
with $V(\omega)\in L^\infty(\R^n,\C^N)$
zero order and hermitian.
Note that
$L(\omega)$ is not necessarily
a linearization at a solitary wave of the nonlinear Dirac equation.

We start with the following elementary result.

\begin{lemma}[Krein's theorem]
\label{lemma-php2}
Let $J\in\End(\C^N)$ be skew-adjoint and invertible
and let $L$ be self-adjoint on $L\sp 2(\R\sp n,\C^N)$.
If $\lambda\in\sigma\sb{\mathrm{p}}(JL)\setminus \jj \R$
and $\zeta$ is a corresponding eigenvector,
then
\[
\langle\zeta,L\zeta\rangle=0,
\qquad
\langle\zeta,J^{-1}\zeta\rangle=0.
\]
\end{lemma}

\begin{proof}
One has
$JL\zeta=\lambda\zeta$,
$L\zeta=\lambda J\sp{-1}\zeta$,
hence
\begin{equation}\label{php}
\langle\zeta,L\zeta\rangle
=\lambda\langle\zeta,J\sp{-1}\zeta\rangle.
\end{equation}
Since $\langle\zeta,L\zeta\rangle\in\R$
and $\langle\zeta,J\sp{-1}\zeta\rangle\in\jj \R$,
the condition $\Re\lambda\ne 0$
implies that both sides in \eqref{php} are equal to zero.
\end{proof}


\begin{lemma}\label{lemma-between-thresholds}
Let $n\ge 1$.
Let $J\in\End(\C^N)$ be skew-adjoint and invertible,
such that $J^2=-I\sb{\C^N}$, $[J,D\sb m]=0$.
Let $\omega\sb j\in(-m,m)$,
$\omega\sb j
\mathop{\longrightarrow}\limits\sb{j\to\infty}
\omega\sb{0}\in[-m,m]$.
Let
\[
L(\omega)=D\sb m-\omega+V(\omega),
\qquad
\omega\in[-m,m],
\]
with
$V(\omega)\in L^\infty(\R\sp n,\End(\C^N))$
a zero-order operator-valued function
which is hermitian for each $\omega\in[-m,m]$,
and assume that there is $\varepsilon>0$ such that
\begin{equation}\label{lec}
\begin{cases}
\norm{\langle r\rangle\sp{1+\varepsilon}
V(\omega\sb{0})}\sb{L^\infty(\R\sp
n,\End(\C^N))}<\infty,
\\[2ex]
\lim\sb{j\to\infty}
\,\norm{\langle r\rangle\sp{1+\varepsilon}\left(
V(\omega\sb j)-V(\omega\sb{0})\right)}\sb{L^\infty(\R\sp
n,\End(\C^N))}=0.
\end{cases}
\end{equation}
Let $\lambda\sb j\in\sigma\sb{\mathrm{d}}(JL(\omega\sb j))$
be a sequence such that
\begin{equation}\label{in-between}
\lambda\sb j
\mathop{\longrightarrow}\limits\sb{j\to\infty}
\lambda\sb{0}\in\jj \R,
\qquad
\abs{\lambda\sb{0}}<m+\abs{\omega\sb{0}},
\end{equation}
with
\begin{equation}\label{re-lambda-nonzero}
\Re\lambda\sb j\ne 0,\quad\forall j\in\N.
\end{equation}
If $\omega\sb{0}= \pm m$,
additionally assume that
\begin{equation}\label{ends-do-not-meet}
\lambda\sb{0}\ne 0.
\end{equation}
Then
\begin{equation}\label{lambda-is-embedded}
\lambda\sb{0}\in\sigma\sb{\mathrm{p}}(JL(\omega\sb{0})).
\end{equation}
\end{lemma}

\begin{proof}
Let $(\zeta\sb j)\sb{j\in\N}$ be a sequence of unit eigenvectors
associated
with eigenvalues $\lambda\sb j$,
so that
$JL(\omega\sb j)\zeta\sb j=\lambda\sb j\zeta\sb j$.
It follows that
\begin{equation}\label{d-zeta-v-zeta}
(D\sb m-\omega\sb j+\lambda\sb j J)\zeta\sb j
=
-V(\omega\sb j)\zeta\sb j.
\end{equation}
Let
$\Pi\sp\pm$
(cf. \eqref{def-pi-pm})
be the projectors onto eigenspaces of $J$
corresponding to $\pm\jj\in\sigma(J)$, respectively.
We denote
$\zeta\sb j\sp\pm=\Pi\sp\pm\zeta\sb j$.
By \eqref{re-lambda-nonzero},
applying Lemma~\ref{lemma-php2},
we conclude that
$
0
=\langle\zeta\sb j,J\zeta\sb j\rangle
=\jj \norm{\zeta\sb j\sp{+}}\sp 2-\jj \norm{\zeta\sb j\sp{-}}\sp 2,
$
$
j\in\N,
$
while
$
1=\norm{\zeta\sb j}\sp 2
=\norm{\zeta\sb j\sp{+}}\sp 2+\norm{\zeta\sb j\sp{-}}\sp 2,
$
$
j\in\N;
$
we conclude that
$\norm{\zeta\sb j\sp\pm}=1/\sqrt{2}$.
Applying $\Pi\sp\pm$ to \eqref{d-zeta-v-zeta},
we have:
\begin{equation}\label{d-zeta-v-zeta-p}
\big(D\sb m-\omega\sb j+\jj\lambda\sb j\big)\zeta\sb j\sp{+}
=-\Pi\sp{+} V(\omega\sb j)\zeta\sb j,
\end{equation}
\begin{equation}\label{d-zeta-v-zeta-m}
\big(D\sb m-\omega\sb j-\jj\lambda\sb j\big)\zeta\sb j\sp{-}
=-\Pi\sp{-} V(\omega\sb j)\zeta\sb j.
\end{equation}
Above, we took into account that
$[J,D\sb m]=0$, hence the projections $\Pi\sp\pm$
also commute with $D\sb m$.

If the condition \eqref{in-between}
(as well as \eqref{ends-do-not-meet}
when $\omega\sb{0}=\pm m$)
is satisfied,
then either
$\omega\sb 0+\jj\lambda\sb 0\in(-m,m)$
or
$\omega\sb 0-\jj\lambda\sb 0\in(-m,m)$.
Both cases are considered similarly;
for definiteness, we will assume that
\[
\omega\sb 0-\jj\lambda\sb 0\in(-m,m).
\]
In this case, without loss of generality,
we may also assume that
\begin{equation}\label{inside-the-gap}
\omega\sb j-\jj\lambda\sb j\in(-m,m),
\qquad \forall j\in\N.
\end{equation}
By \eqref{lec},
the right-hand side of \eqref{d-zeta-v-zeta-p}
belongs to $L^2_s$, $s\le 1+\varepsilon$.
Due to \eqref{inside-the-gap},
we may apply Lemma~\ref{lemma-hardy-complex-gap}
to \eqref{d-zeta-v-zeta-p},
concluding that there is $s\in(0,1)$
such that
$\norm{\zeta\sb j\sp{+}}\sb{H^1_s}$ are uniformly bounded
when $j$ is large enough
(so that
$\jj \lambda\sb j+\omega\sb j$
are sufficiently close to $\jj \lambda\sb{0}+\omega\sb{0}$).
Thus
the sequence
$(\zeta\sb j\sp +)\sb{j\in\N}$ is precompact in
$L\sp 2(\R\sp n,\C^N)$,
and we can choose a subsequence which converges
to some vector
$\zeta\sb{0}\sp{+}\in L^2(\R^n,\C^{N})$
of norm
$\norm{\zeta\sb{0}\sp{+}}
=\lim\sb{j\to\infty}\norm{\zeta\sb j\sp +}=1/\sqrt{2}$.
At the same time, any subsequence
of the bounded sequence $(\zeta\sb j\sp -)\sb{j\in\N}$
also contains a weakly convergent subsequence.
We conclude that there is a subsequence
$(\zeta\sb j)\sb{j\in\N}$
which has a nonzero weak limit;
this limit is necessarily an eigenvector
of $JL(\omega\sb{0})$
corresponding to $\lambda\sb{0}$.
\end{proof}

\subsection{Bifurcation of eigenvalues beyond the embedded thresholds}

We now
turn to the proof of the limiting absorption principle for the linearized
operator in a neighborhood of any purely imaginary point
beyond the embedded thresholds $\pm\jj(m+\abs{\omega})$,
proving Theorem~\ref{theorem-b}
for the case $\abs{\lambda\sb{0}}>m+\abs{\omega\sb{0}}$.
In that respect we closely follow the strategy initiated
in a work by Jensen and Kato \cite{MR544248}
which is related to the approach by \cite{MR0397194}.
We start with the following  identity:
\[
J(D_m-\omega +V(\omega))-\lambda
=\left( J(D_m-\omega)-\lambda\right)\left(
1+ \big(J(D_m-\omega)-\lambda\big)^{-1}JV(\omega)\right),
\]
$\lambda\in\C\setminus\sigma(J(D_m-\omega))$.
After diagonalizing $J$ (which commutes with $D_m$),
Lemma~\ref{lemma-lap-dirac}
provides the limiting absorption principle for
$J(D_m-\omega)-\lambda$;
hence, our task reduces to proving that the operator
$
A(\lambda,\omega)=1+\big(J(D_m-\omega)-\lambda\big)^{-1}JV(\omega),
$
\[
A(\lambda,\omega):\;
L^2_{-s}(\R^n,\C^N)\to L^2_{-s}(\R^n,\C^N),
\qquad
s>1/2,
\]
has an inverse which is bounded uniformly in $\lambda$, with $\Re \lambda>0$,
in the vicinity of any
particular point
\[
(\omega\sb 0,\lambda\sb 0),
\qquad
\omega\sb 0\in[-m,m],
\qquad
\lambda_0\in\jj\R,
\qquad
\abs{\lambda_0}>\jj(m+|\omega\sb 0|)
\]
where we know that $A$ is invertible,
and also proving that
$A(\omega\sb 0,\lambda\sb 0)$ is not invertible if and only if
$\lambda\sb 0$
is an eigenvalue of
$J(D_m-\omega\sb 0+V(\omega\sb 0))$.

\begin{proposition}\label{prop-point}
Let $V:\,\R\sp n\to\End(\C^N)$ be measurable, hermitian-valued,
and assume that there are $\varepsilon>0$ and $C<\infty$ such that
\[
\norm{\langle
r\rangle\sp{1+\varepsilon}V}\sb{L^\infty(\R\sp n,\End(\C^N))}
<\infty.
\]
Let $\omega\in[-m,m]$,
$\lambda\in\jj\R$, $\abs{\lambda}>m+\abs{\omega}$,
$s\in(1/2,(1+\varepsilon)/2)$.
Then either the operator
\[
A=1+\big(J(D_m-\omega)-\lambda\big)^{-1}JV,
\qquad
A
:\;
H^{1/2}_{-s}(\R^n,\C^N)\to H^{1/2}_{-s}(\R^n,\C^N)
\]
is invertible, or there is a nonzero function $F\in\ker A$, $F\in L^2(\R^n,\C^N)$.
\end{proposition}

\begin{proof}
By \eqref{lec-0}, for any
$s<(1+\varepsilon)/2$, the map $V:\,u\mapsto V u$
is bounded
from $H^{1/2}\sb{-s}$ to $H^{-1/2}\sb{s}$.
Due to the limiting absorption principle
for the Dirac operator
(cf. Lemma~\ref{lemma-lap-dirac}),
the resolvent
\[
R_0(\lambda):=\left(J(D_m-\omega)-\lambda\right)^{-1},
\qquad
\lambda\in\C,
\quad
\Re\lambda>0,
\]
can be extended onto the closure of the right half-plane,
excluding arbitrarily small open neighborhoods of $\pm\jj(m\pm\omega)$
(we keep the same notation $R_0$ for this extension),
so that for any $s>1/2$ one has
\[
R_0(\lambda)\in B(L^{2}_{s}(\R^n,\C^N),\,H^{1}_{-s}(\R^n,\C^N)),
\ \quad
\lambda\in\C,
\ \Re\lambda\ge 0,
\ \lambda\ne\pm\jj(m\pm\omega).
\]
Due to the decay of $V$,
the operator $A-1$ is compact
in $L^2_{-s}$
if $s\in\big(1/2,(1+\varepsilon)/2\big)$.
Hence, by the Fredholm alternative,
$A$ is invertible
in $L^{2}_{-s}$
if and only if its null space
\[
\mathfrak{M}_s:=\ker A\at{L^{2}_{-s}}
\]
is trivial.

Let us introduce the Fredholm operator
$B=1+VJ\big(J(D_m-\omega)+\lambda\big)^{-1}$
on $H^{-1/2}_{s}(\R^n,\C^N)$,
with
$s\in\big(1/2,(1+\varepsilon)/2\big)$.
We denote its null space
by
\[
\mathfrak{N}_s:=\ker B\at{L^{2}_{s}}.
\]
Being compact perturbations of the identity,
both $A$ and $B$ are Fredholm operators of index zero.
As in \cite[Section 3]{MR544248},
the finite-dimensional spaces
$\mathfrak{M}_s$ and $\mathfrak{N}_s$
are respectively non-decreasing and non-increasing
as $s$ grows.
Since $A\at{H^{1/2}_{-s}}$ and $B\at{H^{-1/2}_{s}}$
are mutually adjoint,
\[
0=\mathop{\rm ind} A\at{H^{1/2}_{-s}}
=\dim\ker A\at{L^{2}_{-s}}\!-\dim\mathop{\rm coker} A\at{L^{2}_{-s}}
=\dim\ker A\at{L^{2}_{-s}}\!-\dim\ker B\at{L^{2}_{s}},
\]
is a non-decreasing function of
$s\in\big(1/2,(1+\varepsilon)/2\big)$,
hence
$\dim\mathfrak{M}_s=\dim\mathfrak{N}_s$
does not depend on $s\in\big(1/2,(1+\varepsilon)/2\big)$.
We conclude that the spaces
$\mathfrak{M}_s$, $\mathfrak{N}_s$ do not depend on $s\in\big(1/2,(1+\varepsilon)/2\big)$;
we will denote these spaces by
$\mathfrak{M}$ and $\mathfrak{N}$,
respectfully.

One key fact is the following lemma.

\begin{lemma}\label{lemma-k-v}
Let $s>1/2$,
$\omega\in[-m,m]$,
$\lambda\in\jj\R$,
$\abs{\lambda}>m+\abs{\omega}$.
Then
\[
\left(J(D_m-\omega)-\lambda\right)R_0(\lambda)v=v,
\qquad \forall v \in H^{-1/2}_{s}(\R^n,\C^N).
\]
\end{lemma}

\begin{proof}
This is an adaptation of \cite[Lemma 2.4]{MR544248}.
Fix $v \in H^{-1/2}_{s}(\R^n,\C^N)$.
We note that for $\lambda\in\jj\R$
one has
\[
(J(D_m-\omega)-\lambda)\sp\ast
=
-(D_m-\omega)J-\bar\lambda
=-(J(D_m-\omega)-\lambda),
\qquad \forall \lambda\in\jj\R,
\]
where we took into account that $[J,D_m]=0$.
Therefore, for any $\varphi\in C^\infty\sb{\mathrm{comp}}(\R^n,\C^N)$,
since $R_0(\lambda)v\in H^{1/2}_{-s}$, we have:
\[
\langle
(J(D_m-\omega)-\lambda)R_0(\lambda)v,
\varphi
\rangle
=
-\langle
R_0(\lambda)v,
(J(D_m-\omega)-\lambda)
\varphi
\rangle.
\]
Using $R_0(\lambda)\sp\ast=-R_0(\lambda)$
for $\lambda\in\jj\R$,
we then write
\[
\big\langle
(J(D_m-\omega)-\lambda)R_0(\lambda)v,
\varphi
\big\rangle
=
\big\langle
v,
R_0(\lambda)(J(D_m-\omega)-\lambda)
\varphi
\big\rangle
=
\langle
v,\varphi\rangle,
\]
finishing the proof.
\end{proof}

We deduce from Lemma~\ref{lemma-k-v}
that any $u\in\mathfrak{M}$ satisfies
\[
\big(J(D_m-\omega+V)-\lambda\big)u=0.
\]
In the following, we argue that
\begin{equation}\label{inclusion}
\mathfrak{M}\subset L^2(\R^n,\C^{N}),
\end{equation}
which would conclude the proof of
Proposition~\ref{prop-point}.

The inclusion \eqref{inclusion}
is proved using the following three complementary results.

\begin{lemma}\label{lemma-f-d-f}
Let $s>1/2$,
$\epsilon>0$.
Let $\Lambda\in\R$, $\abs{\Lambda}>m$.
If $f\in H^{-1/2}_{s}(\R^n,\C^N)$, then
\[
 \lim_{\epsilon\to 0+}\Im\left\langle f,
(D_m-\Lambda-\jj\epsilon)^{-1}f\right\rangle
=
\frac{\pi
  \abs{\Lambda}
}{\sqrt{\Lambda^2-m^2}}\int_{\xi^2+m^2=\Lambda^2} |\tau
P^+\hat{f}(\xi)|^2\,d\sigma(\xi),
\]
where $\tau$ denotes the trace operator on Sobolev space $H^\kappa(\R^n,\C^N)$
of order $\kappa>1/2$, $\hat{\cdot}$ is the unitary Fourier transform on
tempered distributions,
$d\sigma$ is the induced measure on the surface  $\xi^2+m^2=\Lambda^2$,
and
$
P^\pm=\frac 1 2\Big(1\pm\frac{d_m(\xi)}{\sqrt{\xi^2+m^2}}\Big)
$
are the projectors onto positive and negative eigenvalues,
$\pm\sqrt{\xi^2+m^2}$,
of
the symbol $d_m(\xi)=\bm\alpha\cdot\xi+\beta m$.
\end{lemma}
\begin{proof}
First we notice that for $f\in H^{-1/2}_{s}(\R^n,\C^N)$,
one has:
\begin{eqnarray}
\left\langle f, (D_m-\Lambda-\jj\epsilon)^{-1}f\right\rangle
&=&
\int_{\R^n}
\left(
\bm\alpha\cdot \xi + \beta m -\Lambda-\jj\epsilon\right)^{-1}
|\hat{f}(\xi)|^2\,d\xi
\nonumber
\\[1ex]
&=&
\int_{\R^n}
\frac{|P^{+}(\xi)\hat{f}(\xi)|^2\,d\xi}
{\sqrt{\xi^2+m^2}-\Lambda-\jj\epsilon}
+\int_{\R^n}
\frac{|P^{-}(\xi)\hat{f}(\xi)|^2\,d\xi}
{-\sqrt{\xi^2+m^2}-\Lambda-\jj\epsilon}
,
\nonumber
\end{eqnarray}
and hence
\begin{align}
\Im \left\langle f, (D_m-\Lambda-\jj\epsilon)^{-1}f\right\rangle
=\int_{\R^n} \frac{\epsilon
|P^{+}(\xi)\hat{f}(\xi)|^2\,d\xi}
{(\sqrt{\xi^2+m^2}-\Lambda)^2+\epsilon^2}
+
\int_{\R^n} \frac{\epsilon |P^{-}(\xi)\hat{f}(\xi)|^2\,d\xi}
{(\sqrt{\xi^2+m^2}+\Lambda)^2+\epsilon^2}.
\label{two-integrals}
\end{align}
Let us assume that $\Lambda>m$.
In the limit $\epsilon\to 0+$,
the second integral in the right-hand side of \eqref{two-integrals}
tends to $0$.
The first integral can be written as
\begin{align*}
&\int_{\R^n} \frac{2\epsilon}{(\sqrt{\xi^2+m^2}-\Lambda)^2+
\epsilon^2}|P^{+}(\xi)\hat{f}(\xi)|^2\,d\xi
\\[1ex]
&=\int_{\mathbb{S}^{n-1}}\int_{\R^+}
\frac{2\epsilon}{(\sqrt{r^2+m^2}-\Lambda)^2+
\epsilon^2}|P^{+}(r\omega)\hat{f}(r\omega)|^2\,r^{n-1}\,dr\,d\omega
\\[1ex]
&=\int_{\mathbb{S}^{n-1}}\int_{m}^\infty
\Abs{P^{+}(\omega\sqrt{r^2-m^2})\hat{f}(\omega\sqrt{r^2-m^2})}^2
\frac{2\epsilon
\,(r^2-m^2)^{(n-2)/2}r\,dr\,d\omega
}{(r-\Lambda)^2+\epsilon^2}
,
\end{align*}
which converges, in the limit $\epsilon\to 0+$, to
\begin{align*}
&
2\pi
(\Lambda^2-m^2)^{\frac{n-2}{2}}
\Lambda \int\limits_{\mathbb{S}^{n-1}}
\left|P^{+}(\omega\sqrt{\Lambda^2-m^2})\hat{f}(\omega\sqrt{\Lambda^2-m^2})\right|^2\,
d\omega
=\frac{2\pi\Lambda}{\sqrt{\Lambda^2-m^2}}\int\limits_{\xi^2+m^2=\Lambda^2}
|\tau(P^{+}(\xi)\hat{f}(\xi)|^2\,d\sigma.
\end{align*}
This proves the required identity in the case $\Lambda>m$.
The case $\Lambda<-m$ is considered similarly.
\end{proof}

The second result we need is directly inspired by
\cite[Proof of Lemma 2.4]{MR2094265}.
\begin{lemma}\label{lemma-j-d-f}
Assume that $V$
is hermitian
and there are $\varepsilon>0$ and $C<\infty$ such that
\[
\norm{\langle
r\rangle\sp{1+\varepsilon}V}\sb{L^\infty(\R\sp n,\End(\C^N))}
<\infty.
\]
Let
$s\in\big(1/2,(1+\varepsilon)/2\big)$,
$\omega \in [-m,m]$.
Let $\lambda\in\jj \R$,
$\abs{\lambda}>m+\abs{\omega}$.

Then for any $F\in H^{1/2}_{-s}(\R^n,\C^N)$ such that
\[
J(D_m-\omega +V)F=\lambda F,
\]
the function
$G:=(D_m-\omega+J\lambda)F=-VF$
satisfies
\[
\int_{\sqrt{\xi^2+m^2}=
\abs{\omega\mp\jj\lambda}
}
\Abs{\Pi\sp\pm
\tau\big(P^{+}(\xi)\hat{G}(\xi)\big)}^2\,
d\sigma=0,
\]
where
$\Pi\sp\pm=\frac 1 2(1\mp\jj J)$
is the projector onto eigenspaces
of $J$ corresponding to $\pm\jj\in\sigma(J)$
(cf. \eqref{def-pi-pm}).
\end{lemma}

\begin{proof}
We assume that
$\lambda=\jj\Lambda$,
with $\Lambda>m+\abs{\omega}$.
(The case $\lambda=-\jj\Lambda$ is considered verbatim.)

Applying the spectral projectors
$\Pi\sp\pm$
to the relation
\[
G=(D_m-\omega+J\lambda)F
\]
and denoting
$
F^\pm=\Pi\sp\pm F$,
$
G^\pm=\Pi\sp\pm G$,
we have:
\[
G^\pm
=(D_m-\omega\pm\jj\lambda)F^\pm
=(D_m-\omega\mp\Lambda)F^\pm
=-\Pi\sp{+}V F.
\]
One has:
\begin{eqnarray}
\lim_{\epsilon \to 0+}
\langle G^\pm,(D_m-\omega\mp\Lambda-\jj\epsilon
)^{-1}G^\pm\rangle
&=&
-\lim_{\epsilon \to 0+}
\big\langle G^\pm,(D_m-\omega\mp\Lambda-\jj\epsilon
)^{-1}\Pi\sp\pm V F
\big\rangle
\nonumber
\\[1ex]
&=&
-\big\langle F^\pm, \Pi\sp\pm V F\big\rangle
=
-\langle F^\pm, V F\rangle.
\nonumber
\end{eqnarray}
Summing up
the expressions corresponding to $\pm$ signs,
taking the imaginary part,
and applying Lemma~\ref{lemma-f-d-f}
leads to
\begin{align*}
0
=-\Im
\langle F,VF\rangle
&=
-\Im\big(
\langle F\sp{+},VF\rangle+\langle F\sp{-},VF\rangle
\big)
=
\sum\sb\pm
\lim_{\epsilon \to 0+}\Im\langle G^\pm,(D_m-\omega\mp\Lambda-\jj\epsilon
)^{-1}G^\pm\rangle
\\
&=
\frac{\pi\abs{\omega+\Lambda}}{\sqrt{(\omega+\Lambda)^2-m^2}}
\int_{\sqrt{\xi^2+m^2}=\abs{\omega+\Lambda}}
\Abs{\Pi\sp{+}\tau\big(P^+(\xi)\hat{G}(\xi)\big)}^2\,d\sigma(\xi)
\\
&
\qquad
+
\frac{\pi\abs{\omega-\Lambda}}{\sqrt{(\omega-\Lambda)^2-m^2}}
\int_{\sqrt{\xi^2+m^2}=\abs{\omega-\Lambda}
}
\Abs{\Pi\sp{-}\tau\big(P^+(\xi)\hat{G}(\xi)\big)}^2\,d\sigma(\xi).
\end{align*}
In the very first equality,
we used our assumption that $V$ is hermitian.

Since both the coefficients and the integrals
in the right-hand side are positive,
the conclusion follows.
\end{proof}

The last step is needed to exclude non-square-integrable resonances.
It is directly inspired by~\cite[Theorem 2]{MR880983}.
\begin{lemma}\label{lemma-f-j-d}
Let $s>1/2$.
Let $\omega\in [-m,m]$.
Let $\lambda=\jj\Lambda\in\jj\R$
with $\Lambda>m+\abs{\omega}$.
If $F\in\mathscr{S}'(\R^n,\C^N)$ is such that
\[
[J(D_m-\omega)-\lambda]F\in L^{2}_{s}(\R^n,\C^N)
\]
and
if $(1\pm \jj J)P^+(\xi)\big(J(d_m(\xi)-\omega)-\lambda\big)\hat{F}$
vanish on the spheres
$\sqrt{\xi^2+m^2}=\Lambda\pm\omega$,
respectively,
then
\begin{equation}\label{eq:HardyBerthierGeorgescu}
\|F\|_{s-1}\leq C \|[J(D_m-\omega)-\lambda]F\|_{s},
\end{equation}
for some constant $C<\infty$
depending on $s$, $\lambda$ and $\omega$ only.
\end{lemma}
\begin{proof}
The proof
of \cite[Theorem 2]{MR880983}
works with a straightforward adaptation of the key \cite[Lemma 5]{MR880983},
which is a consequence of \cite[Appendix B]{MR0397194}
which in turn is valid in any dimension;
the assumptions needed to apply it are in the assumption of the Lemma.
\end{proof}

\begin{remark}
In \cite{MR880983},
Berthier and Georgescu proved a result similar
to \eqref{eq:HardyBerthierGeorgescu}
under the $L^1_{\mathrm{loc}}$ assumption on the Fourier
transform of $F$. Such an assumption provides that $(1\pm \jj
J)P^+(\xi)\big(J(d_m(\xi)-\omega)-\lambda\big)\hat{F}$
vanish on the spheres
$\sqrt{\xi^2+m^2}=\Lambda\pm\omega$,
respectively.
\end{remark}

Lemmata~\ref{lemma-f-d-f}, \ref{lemma-j-d-f} and~\ref{lemma-f-j-d}
complete the proof of the inclusion
\eqref{inclusion}.
Proposition~\ref{prop-point} follows
from \eqref{inclusion} and Lemma~\ref{lemma-k-v}.
\end{proof}

\begin{proposition}\label{prop:LAPLinearization}
Let $V:\,\R^n\times[-m,m]\to\End(\C^N)$ be measurable,
and assume that there are
$C<\infty$ and $\varepsilon>0$ such that
\eqref{lec-0} is satisfied.

Let
$\omega\sb 0\in[-m,m]$ and let
\[
\lambda\sb 0\in\jj\R,
\qquad
\abs{\lambda\sb 0}>m+\abs{\omega\sb 0},
\qquad
\lambda\sb 0\not\in\sigma\sb{\mathrm{p}}\big(J(D_m-\omega+V(\omega))\big).
\]
Then for any $s\in\big(1/2,(1+\varepsilon)/2\big)$
there exist an open neighborhood
$I\subset[-m,m]$ of $\omega\sb 0$
($I$ is a one-sided neighborhood
of $\omega\sb 0$
if $\omega\sb 0=\pm m$)
and an open neighborhood $U\subset\C$ of $\lambda\sb 0$
such that
for $\omega\in I$
the resolvent
of $J(D_m-\omega+V(\omega))$ at $\lambda\in\overline{U}\setminus\jj\R$
extends to a continuous mapping
\[
\big(J(D_m-\omega+V(\omega))-\lambda\big)^{-1}:\;
H^{-1/2}_s(\R^n,\C^N)
\to
H^{1/2}_{-s}(\R^n,\C^N),
\]
which is bounded uniformly in $\lambda\in\overline{U}\setminus \jj\R$.
\end{proposition}

\begin{proof}
If
$
\lambda\sb 0\not\in\sigma\sb{\mathrm{p}}\big(J(D_m-\omega\sb 0+V(\omega\sb 0))\big),
$
then, by Proposition~\ref{prop-point},
the operator
\[
A(\omega,\lambda)
=1+\big(J(D_m-\omega)-\lambda\big)^{-1}JV(\omega),
\]
\[
A(\omega,\lambda):\;H^{1/2}\sb{-s}(\R^n,\C^N)\to H^{1/2}\sb{-s}(\R^n,\C^N)
\]
is invertible at $(\omega\sb 0,\lambda\sb 0)$.
By the limiting absorption principle
(Lemma~\ref{lemma-lap-dirac}),
for any $s>1/2$,
there is
an open neighborhood $I\subset[-m,m]$ of $\omega\sb 0$
($I$ is a one-sided neighborhood if $\omega\sb 0=\pm m$)
and
an open neighborhood $U\subset\C$ of $\lambda\sb 0$
such that
$R_0(\lambda)=(J(D_m-\omega)-\lambda)^{-1}$
remains continuous in the
$H^{-1/2}_{s}\to H^{1/2}_{-s}$ operator topology
for $\omega\in I$,
$\lambda\in U$, $\Re\lambda\ge 0$.
Similarly,
$V(\omega):\,H^{1/2}_{-s}\to H^{-1/2}_s$
remains continuous in the  corresponding operator topology
for $\omega\in I\cap[-m,m]$
(cf. \eqref{lec-0}).
By continuity in $\omega$ and $\lambda$,
the operator
$A(\omega,\lambda)$ is continuous in the $H^{1/2}_{-s}\to H^{1/2}_{-s}$ operator topology,
remaining invertible
for $(\omega,\lambda)$
in an open neighborhood of $(\omega\sb 0,\lambda\sb 0)$,
with $\Re\lambda\ge 0$.
\end{proof}

Now we can finish the proof of
Theorem~\ref{theorem-b}.

\begin{lemma}\label{lemma-beyond-thresholds}
Let $\big(\omega\sb j\big)\sb{j\in\N}$,
$\omega\sb j\in(-m,m)$,
$\omega\sb j\to\omega\sb{0}\in[-m,m]$.
Assume that
\[
\lambda\sb{0}\in\jj\R,
\qquad
\abs{\lambda\sb{0}}>m+\abs{\omega\sb{0}},
\qquad
\lambda\sb{0}\not\in
\sigma\sb{\mathrm{p}}\big(J(D_m-\omega\sb{0}+V(\omega\sb{0}))\big).
\]
Then there is no sequence
$\lambda\sb j\in
\sigma\sb{\mathrm{p}}\big(J(D_m-\omega_j+V(\omega\sb j))\big)$
such that
$\lambda\sb j\to\lambda\sb{0}\in\jj\R$.
\end{lemma}

\begin{proof}
We use the last part of the proof of Proposition~\ref{prop:LAPLinearization}
to argue by contradiction.
Indeed, let $\big(\omega\sb j\big)\sb{j\in\N}$, $\omega\sb j\in(-m,m)$,
$\omega\sb j\to\omega\sb 0\in[-m,m]$,
and assume that
$\lambda\sb j\in\sigma\sb{\mathrm{p}}\big(J(D_m-\omega\sb j+V(\omega\sb j))\big)$,
$\lambda\sb j\to\lambda\sb{0}\in\jj\R$, $\abs{\lambda\sb{0}}>m+\abs{\omega\sb{0}}$.
Fix $s,\,s'$
such that
\[
\frac 1 2<s<s'<\frac{1+\varepsilon}{2},
\]
with $\varepsilon>0$ from \eqref{lec-0}.
Since the operator
\[
\big(J(D_m-\omega_j)-\lambda_j\big)^{-1}JV(\omega_j)
\]
is bounded from $L_{-s'}^2$ to $H^{1/2}_{-s}$ uniformly in $j\in\N$,
while the latter embeds compactly into $L^2_{-s'}$,
we conclude that
any sequence of eigenvectors (normalized in $L^2_{-s'}$)
associated to
$\lambda_j\in\sigma\sb{\mathrm{p}}
\big(J(D_m-\omega_j+V(\omega_j))\big)$
is compact in $L^2_{-s'}$,
converging to a nonzero vector from $H^{1/2}_{-s}$,
leading to a contradiction with the operator
\[
A(\lambda\sb{0})=1+\big(J(D_m-\omega\sb{0})-\lambda\sb{0}\big)^{-1}JV(\omega\sb{0}):
\;H^{1/2}_{-s}(\R^n,\C^N)\to H^{1/2}_{-s}(\R^n,\C^N)
\]
being invertible
by Proposition~\ref{prop:LAPLinearization}.
\end{proof}

This finishes the proof of
Theorem~\ref{theorem-b}.

\section{Bifurcations from the essential spectrum of the free Dirac operator}
\label{sect-vb}

In this section, we prove Theorem~\ref{theorem-vb}.
The proof follows from
Lemma~\ref{lemma-free-dirac}.

Let us consider
families of eigenvalues
in the limit of small amplitude solitary waves,
which may be present in the spectrum
up to the border of existence
of solitary waves: $\omega\to\omega\sb{0}\in\{\pm m\}$.
This situation could be considered
as the bifurcation of eigenvalues from the continuous spectrum
of the free Dirac equation.

\begin{lemma}\label{lemma-free-dirac}
Let $V(\omega)\in L^\infty(\R^n,\End(\C^N))$, $\omega\in[-m,m]$,
and let $\big(\omega\sb j\big)\sb{j\in\N}$
be a sequence such that
$\omega\sb j\in(-m,m)$
and $\omega\sb j\to\omega\sb{0}=\pm m$.
Assume that there is $\varepsilon>0$ such that
\begin{equation}\label{lim-s-half}
\lim\sb{j\to\infty}
\norm{\langle r\rangle\sp{1+\varepsilon}V(\omega\sb j)}
\sb{L^\infty(\R\sp n,\End(\C^N))}=0.
\end{equation}
If there is a sequence
$\big(\lambda\sb j\big)\sb{j\in\N}$
such that
$\lambda\sb j\in\sigma\sb{\mathrm{p}}(JL(\omega\sb j))$,
then
the only accumulation points of
$\big(\lambda\sb j\big)\sb{j\in\N}$
in the extended complex plane
are $z=0$ and $z=\pm 2 m \jj $.
\end{lemma}

\begin{proof}
Let us consider the case when
$\Re\lambda\sb j\ne 0$ for infinitely many $j\in\N$.
We need to show that
for any $\delta>0$
the point spectrum of $JL(\omega)$
is contained inside an open set
\[
U\sb\delta:=
\mathbb{D}\sb\delta(-2 m \jj )
\cup\mathbb{D}\sb\delta(0)\cup\mathbb{D}\sb\delta(2 m \jj ),
\]
as long as $\omega$ is sufficiently close to $\omega\sb{0}$.
Above,
$\mathbb{D}\sb\delta(z)$ denotes an open disc of radius $\delta$
around $z\in\C$.

Fix $\delta>0$.
Let $\abs{\omega-\omega\sb{0}}<\delta$;
then
$\pm \jj (m\pm\omega)\in U\sb\delta$.
Since the eigenvalues of $ J$ are $\pm \jj $,
the operator $ J( D\sb m-\omega)$
can be represented
as the direct sum of operators
$ \jj ( D\sb m-\omega)$ and $-\jj ( D\sb m-\omega)$.
By Lemma~\ref{lemma-lap-dirac}, for any $s>1/2$
the following map is bounded
uniformly for $z\in \C\setminus(\jj \R\cup U\sb\delta)$:
\begin{equation}\label{lap-jl0}
\big( J( D\sb m-\omega)-z\big)\sp{-1}:\;
L\sp 2\sb{s}(\R\sp n,\C^N)\to L\sp 2\sb{-s}(\R\sp n,\C^N),
\qquad
z\in \C\setminus(\jj \R\cup U\sb\delta).
\end{equation}
For appropriate values of $z\in\C$,
the resolvent of $JL(\omega)$
is expressed as
\begin{equation}\label{res}
(JL(\omega)-z)\sp{-1}
=
\big( J( D\sb m-\omega)-z\big)\sp{-1}
\frac{1}{1+ J V
\big( J( D\sb m-\omega)-z\big)\sp{-1}}.
\end{equation}
Thus, the action
\begin{equation}\label{action-bounded}
(JL(\omega)-z)\sp{-1}\;:
L\sp 2\sb{s}(\R\sp n,\C^N)\to L\sp 2\sb{-s}(\R\sp n,\C^N)
\end{equation}
is bounded uniformly in $z\in \C\setminus(\jj \R\cup U\sb\delta)$
as long as the operator
\[
V(\omega):\;
L\sp 2\sb{-s}(\R\sp n,\C^N)\to L\sp 2\sb{s}(\R\sp n,\C^N)
\]
of multiplication by $ V(x,\omega)$
has a sufficiently small norm;
it is enough to have
\begin{equation}\label{vr-small}
\norm{ V}\sb{L\sp 2\sb{-s}\to L\sp 2\sb{s}}
\norm{\big( J( D\sb m-\omega)-z\big)\sp{-1}}\sb{L\sp 2\sb{s}
\to L\sp 2\sb{-s}}<1/2.
\end{equation}
We choose
$s\in(1,(1+\varepsilon)/2)$, with $\varepsilon$ from \eqref{lim-s-half}.
Due to
the bound on the action \eqref{lap-jl0},
the inequality \eqref{vr-small} with $\omega=\omega_j$
is satisfied for $j$ sufficiently large,
since
\[
\lim\sb{j\to\infty}
\norm{ V(\omega_j)}
\sb{L\sp 2\sb{-s}(\R\sp n,\C^N)\to L\sp 2\sb{s}(\R\sp n,\C^N)}
\le
\lim\sb{j\to\infty}
\norm{\langle r\rangle\sp{2s} V(\omega\sb j)}
\sb{L^\infty(\R\sp n,\End(\C^N))}=0
\]
by the assumption of the lemma.
Due to the boundedness of the action
of \eqref{action-bounded},
uniformly in
$z\in\C\setminus(\jj\R\cup U\sb\delta)$,
for $j\in\N$ sufficiently large,
we conclude that
for these large $j$
the point spectrum of $JL(\omega\sb j)$
is inside $U\sb\delta$.

\medskip

Now let us consider the case when
$\Re\lambda\sb j=0$ for infinitely many $j\in\N$.
Without loss of generality,
let us assume that $\omega\sb 0=m$.
Let $\Lambda\sb j\in\R$ and $\Lambda\sb 0\in\R$ be such that
\[
\lambda\sb j=\jj\Lambda\sb j
\quad
\forall j\in\N,
\qquad
\lambda\sb j\to \lambda\sb 0=\jj\Lambda\sb 0.
\]
Let $L(\omega)=D_m-\omega+V(\omega)$,
and let
$\zeta\sb j$ be eigenvectors
corresponding to eigenvalues
$\lambda\sb j\in\sigma\sb{\mathrm{p}}(JL(\omega\sb j))$,
$\norm{\zeta\sb j}=1$ $\forall j\in\N$.
Applying the projections
$\Pi\sp\pm=\frac 1 2(1\mp\jj J)$
to the relation
$L(\omega\sb j)\zeta\sb j=-J\lambda\sb j\zeta\sb j
=-\jj J\Lambda\sb j\zeta\sb j$,
we have
\[
\begin{cases}
\Pi\sp{-}(D_m-\omega\sb j+\Lambda\sb j)\zeta\sb j
=-\Pi\sp{-}V(\omega\sb j)\zeta\sb j,
\\[1ex]
\Pi\sp{+}(D_m-\omega\sb j-\Lambda\sb j)\zeta\sb j
=-\Pi\sp{+}V(\omega\sb j)\zeta\sb j.
\end{cases}
\]
Due to $[J,D_m]=0$, the above relations take the form
\begin{equation}\label{t-r}
\begin{cases}
(D_m-\omega\sb j+\Lambda\sb j)
\Pi\sp{-}\zeta\sb j
=-\Pi\sp{-}V(\omega\sb j)\zeta\sb j,
\\[1ex]
(D_m-\omega\sb j-\Lambda\sb j)
\Pi\sp{+}\zeta\sb j
=-\Pi\sp{+}V(\omega\sb j)\zeta\sb j.
\end{cases}
\end{equation}
Assume that
\begin{equation}\label{lambda-zero-strange}
\Lambda\sb 0\not\in\{0,\pm 2m\}.
\end{equation}
Since
$\omega\sb j\to\omega\sb 0=m$
and
$\Lambda\sb j\to \Lambda\sb 0\not\in\{0,\pm 2m\}$,
without loss of generality,
we may assume that
either
\begin{equation}\label{either-or-1}
\omega\sb j-\Lambda\sb j\in(-m,m),
\quad
\forall j\in\N;
\qquad
\omega\sb j-\Lambda\sb j
\to
\omega\sb 0-\Lambda\sb 0\in(-m,m),
\end{equation}
or
\begin{equation}\label{either-or-2}
\omega\sb j-\Lambda\sb j\in\R\setminus[-m,m],
\quad
\forall j\in\N;
\qquad
\omega\sb j-\Lambda\sb j\to
\omega\sb 0-\Lambda\sb 0
\in
\R\setminus[-m,m].
\end{equation}
In the case \eqref{either-or-1},
the first relation from \eqref{t-r}
yields
\[
\Pi\sp{-}\zeta\sb j
=
-
(D_m-\omega\sb j+\Lambda\sb j)^{-1}
\Pi\sp{-}V(\omega\sb j)\zeta\sb j,
\]
and, due to the assumption
\eqref{lim-s-half} on $V(\omega\sb j)$,
$\Pi\sp{-}\zeta\sb j\to 0$ strongly in $H^1$.
In the case \eqref{either-or-2},
we apply Lemma~\ref{Lem:BerthierGeorgescuLogarithm},
arriving at
\[
\norm{\Pi\sp{-}\zeta\sb j}
\le
C\norm{(D_m-\omega\sb j+\Lambda\sb j)\Pi\sp{-}\zeta\sb j}\sb{L^2_1}
=
C
\norm{\Pi\sp{-}V(\omega\sb j)\zeta\sb j}\sb{L^2_1}
\le
C
\norm{\langle r\rangle V(\omega\sb j)\zeta\sb j}.
\]
The right-hand side goes to zero due to the assumption
\eqref{lim-s-half}.
Thus, in either case,
$\Pi\sp{-}\zeta\sb j\to 0$ in $L^2$.
Similarly,
$\Pi\sp{+}\zeta\sb j\to 0$ in $L^2$,
leading to a contradiction
to the assumption
$\norm{\zeta\sb j}=1$ $\forall j\in\N$.
This shows that
\eqref{lambda-zero-strange} can not be true.
\end{proof}

\newpage

\appendix

\section{Appendix: Unique continuation principle
for the Dirac operator}
\label{Sec:UniqueContinuation}

In this section, we provide a simple version of the unique continuation
principle  which is sufficient for the sake of our analysis.
The proof is an adaptation of \cite[Appendix]{MR880983}.

\begin{lemma}\label{lemma-unique-continuation}
Let $n\ge 1$.
Denote $D_0=-\jj\bm\alpha\cdot\bm\nabla$,
where $\alpha\sp j$, $1\le j\le n$, are the Dirac matrices.
Let $\nu\in C^\infty(\R^n)$
be spherically symmetric and
strictly monotonically increasing with $r$.
For $a\in\R^n$, denote
\[
\nu_a(x)=\nu(x-a),
\qquad x\in\R^n.
\]
Let $\Omega$ be an open connected set of $\R^n$
and let $V:\,\Omega\to\End(\C^N)$
be measurable.

Assume that there exists $\kappa\in(0,1)$
such that for any
$u\in H^1\sb{\mathrm{comp}}(\Omega,\C^N)$
and any $a\in\R^n\setminus\supp u$
there is a sequence $\tau_j\to\infty$
such that
\begin{equation}\label{assumpt}
\|e^{-\tau_j\nu_a}V u\|
\le\kappa\|e^{-\tau_j\nu_a}D\sb 0u\|,
\qquad \forall j\in\N.
\end{equation}
Assume that
$\psi\in H\sp 1\sb {\mathrm{loc}}(\Omega,\C^N)$
satisfies
\begin{equation}\label{assumpt-k}
\abs{(D\sb 0\psi)(x)}\leq \abs{(V\psi)(x)}
\qquad
\mbox{for $x$ almost everywhere in $\Omega$},
\end{equation}
and that $\psi\equiv 0$
in a non-empty open set $\Omega_0\subset\Omega$.
Then there exists an open set
$\Omega_1\subset \Omega$ such that
$\Omega_0\subsetneqq \Omega_1$,
with $V\psi\equiv 0$ and $D_0\psi\equiv 0$
almost everywhere in $\Omega_1$.
\end{lemma}

\begin{proof}
The proof below, considered as standard,
is given for completeness but boils
down for instance to~\cite[Lemma 1]{MR576454}.
If $\psi\equiv 0$ in $\Omega$, then there is nothing to prove.
Let us assume that $\psi$ is not identically zero
in $\Omega$
and that
$\Omega_0:=\Omega\setminus\supp\psi$
is not empty.
Then there is an open ball of size $R\in(0,1)$
located strictly inside $\Omega_0$
and touching the boundary
of $\supp\psi$
at a single point, which we denote $x\sb\ast$.
Shifting the coordinates, we assume that this ball
is centered at the origin; we now have:
\begin{equation}\label{single-point}
\mathbb{B}^n_R\cap\supp\psi=\varnothing,
\qquad
\p\mathbb{B}^n_R\cap\supp\psi=\{x\sb\ast\}\subset\Omega.
\end{equation}
Let
$\mathbb{B}^n_r(x\sb\ast)\subset\Omega$
be an open ball of radius $r\in(0,R)$
centered at $x\sb\ast$.
Let $\eta\in C^\infty_{\mathrm{comp}}(\mathbb{B}^n_r(x\sb\ast))$,
$\eta\equiv 1$ in $\mathbb{B}^n_{r/2}(x\sb\ast)$.
Due to \eqref{single-point},
\[
R<\inf\{\abs{x}\sothat
x\in\supp\psi\setminus\mathbb{B}^n_{r/2}(x\sb\ast)\}\le\infty;
\]
see Figure~\ref{fig-ucp}.
Therefore,
there is a finite value $R_1>0$
such that
\begin{equation}\label{def-r1}
R<R_1<\inf\{\abs{x}\sothat
x\in\supp\psi\setminus\mathbb{B}^n_{r/2}(x\sb\ast)\}.
\end{equation}
Since $\mathbb{B}^n_R$ is strictly inside $\Omega$,
we may take $R_1$ smaller if necessary
so that $\mathbb{B}^n_{R_1}\subset\Omega$.

\begin{figure}[htbp]
\begin{center}
\setlength{\unitlength}{1pt}
\begin{picture}(80,80)(0,-30)


\put(1,-8){$R$}
\put(0,0){\vector(1,0){20}}
\put(40,0){$\supp\psi$}
\put(15,-30){$\psi\at{\Omega_0}\equiv 0$}
\put(-7,2){$0$}
\put(45,30){$\supp\nabla\eta$}
\put(45,27){\vector(-1,-1){19.5}}
\put(0,0){\circle*{3}}
\put(20,0){\circle*{3}}

\put(0,0){\circle{40}}
\put(20,0){\circle{13}}
\put(20,0){\circle{24}}

\linethickness{1pt}
\qbezier(20,0)(20,20)(60,15)
\qbezier(20,0)(30,-20)(70,-10)
\qbezier(90,10)(90,20)(60,15)
\qbezier(90,10)(90,-8)(70,-10)

\end{picture}
\caption{\footnotesize
Distance from the origin to $\supp\psi\cap\supp\bm\nabla\eta$
is larger than $R$.
}
\label{fig-ucp}
\end{center}
\end{figure}
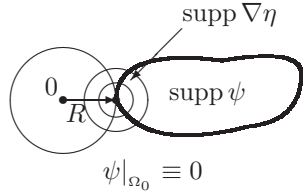

Applying \eqref{assumpt} to $u=\eta\psi$ (note that $0\not\in \supp u$) and then using \eqref{assumpt-k},
one has:
\begin{align*}
\|e^{-\tau_j \nu}V(\eta\psi)\|
\le
\kappa
\|e^{-\tau_j \nu}D\sb 0(\eta\psi)\|
&
\leq
\kappa\|e^{-\tau_j \nu}\eta D\sb
0\psi\|+
\kappa
\|e^{-\tau_j \nu}(\bm\alpha\cdot\bm\nabla\eta)\,\psi\|
\\[1ex]
&
\leq
\kappa\|e^{-\tau_j \nu}\eta V \psi\|
+\kappa\|e^{-\tau_j \nu}(\bm\alpha\cdot\bm\nabla\eta)\,\psi\|.
\end{align*}

\noindent
This allows us to conclude that
\begin{equation}\label{tau-tau-0}
\|e^{-\tau_j \nu}V\eta\psi\|
\leq
\frac{\kappa}{1-\kappa}\|e^{-\tau_j\nu}(\bm\alpha\cdot\bm\nabla\eta)\,\psi\|.
\end{equation}
Due to \eqref{def-r1},
$(\bm\alpha\cdot\bm\nabla\eta)\psi$
is supported outside of $\mathbb{B}^n_{R_1}$;
we conclude that
the right-hand side of \eqref{tau-tau-0} is bounded by
$e^{-\tau_j\nu(R_1)}\norm{(\bm\alpha\cdot\bm\nabla\eta)\,\psi}$.
With $\tau_j\to\infty$
and with $\nu$ being strictly monotonic in $r$,
\eqref{tau-tau-0} shows that
$\eta V\psi=0$ almost everywhere in
$\mathbb{B}^n_{r/2}(x\sb\ast)\cap\mathbb{B}^n_{R_1}$,
hence
$V\psi=0$
almost everywhere in
$\Omega_1:=\Omega_0
\cup
\{
\mathbb{B}^n_{r/2}(x\sb\ast)\cap\mathbb{B}^n_{R_1}
\}
$;
by \eqref{assumpt-k},
one also has
$(D_0\psi)\at{\Omega_1}=0$.
\end{proof}

\begin{lemma}\label{lemma-ede}
Let $R<\infty$.
Let $\nu\in C^\infty(\R^n)$
be spherically symmetric and
strictly monotonically increasing with $r$.
For any $c\in(0,\frac{\sqrt{3}}{2})$
and any $u\in H^1_{\mathrm{comp}}(\R^n,\C^N)$,
$0\not\in\supp u$,
there is $\tau_{c,\nu,u}<\infty$
such that for all $\tau\geq \tau_{c,\nu,u}$ one has
\[
\norm{e^{-\tau\nu}D_0 u}
\ge c\norm{D_0(e^{-\tau\nu}u)}.
\]
\end{lemma}

\begin{remark}
The proof shows that
the dependence of
$\tau\sb{c,\nu,u}$ on $u$
is via $\dist(0,\supp u)$
and $\mathop{\rm diam}(\supp u)$.
\end{remark}

\begin{proof}
Substituting $e^{\tau\nu}u\in H^1\sb{\mathrm{comp}}(\R^n,\C^N)$
in place of $u$,
we need to prove the inequality
\[
\norm{e^{-\tau\nu}D_0(e^{\tau\nu}u)}\ge c\norm{D_0 u};
\]
due to the identity
$e^{-\tau\nu}\circ D_0\circ e^{\tau\nu}=\bm\alpha\cdot(-\jj\nabla -\jj \tau\nabla\nu)$,
this is equivalent to proving
\begin{equation}\label{t-e}
\big\langle
u,\,\bm\alpha\cdot(-\jj\nabla
+\jj\tau\nabla\nu)\,
\bm\alpha\cdot(-\jj\nabla-\jj\tau\nabla\nu)
u
\big\rangle
\\
\geq
c^2
\big\langle u,\,
(-\Delta)
u
\big\rangle.
\end{equation}
For simplicity, we substitute $\tau$ by $1$
(later we will return $\tau$ into formulas).
Using \eqref{ts},
\begin{align*}
&
\bm\alpha\cdot(-\jj\nabla +\jj \nabla\nu)\,\bm\alpha\cdot(-\jj\nabla -\jj
\nabla\nu)
=
(-\jj\nabla +\jj \nabla\nu)\cdot(-\jj\nabla -\jj
\nabla\nu)
+ \jj \Sigma(-\jj\nabla +\jj \nabla\nu, -\jj\nabla -\jj
\nabla\nu)
\\[1ex]
&
= -\Delta+|\nabla \nu|^2-\Delta\nu
- \jj \Sigma(\nabla,
\nabla\nu)
+ \jj \Sigma(\nabla\nu,\nabla)
= -\Delta+|\nabla \nu|^2-\Delta\nu
+ \jj \Sigma_{j k}(\nabla_j\nu\nabla_k+\nabla_k
\nabla_j\nu)
\\[1ex]
&= -\Delta+|\nabla \nu|^2-\Delta\nu
+\jj \Sigma_{j k}\nabla_{j k}^2\nu+2\jj \Sigma_{j k}\nabla_k\nu\nabla_j
=-\Delta+|\nabla \nu|^2-\Delta\nu
+2\jj \Sigma_{j k}\nabla_k\nu\nabla_j,
\end{align*}
with
$\Sigma\sb{j k}=\frac{1}{2 \jj }
[\alpha\sp j,\alpha\sp k]$.
We denote the angular momentum tensor
by
$L_{j k}=x_k\jj\nabla_j-x_j\jj\nabla_k$
and set $G:=\Sigma\cdot L=\sum\sb{j,k}\Sigma\sb{j k}L\sb{j k}$.
Taking into account that $\nu$ is spherically symmetric
and denoting $\nu'=\p\sb r\nu$,
we have:
\begin{equation}
\label{eq:RadialSquareDirac}
 \bm\alpha\cdot(-\jj\nabla +\jj \nabla\nu)
\,\bm\alpha\cdot(-\jj\nabla -\jj \nabla\nu)
= -\Delta
+|\nu'|^2-\Delta\nu
+({\nu'}/{r})G.
\end{equation}
Given any $\theta>0$, we can proceed as follows:
\begin{eqnarray}\label{a-d-a-d-0}
&&
\bm\alpha\cdot(-\jj\nabla +\jj \nabla\nu)
\,\bm\alpha\cdot(-\jj\nabla -\jj \nabla\nu)
\nonumber
\\
&&
=
-\p\sb r^2
- \frac{n-1}{r}\p\sb r
-\frac{\Delta_{\mathbb{S}^{n-1}}}{r^2}
+|\nu'|^2-\Delta\nu
+ \frac{\nu'}{r}\left(G+\frac{n-2}{2}\right)
-\frac{(n-2)\nu'}{2r}
\nonumber
\\
&&
\ge
-\p\sb r^2
- \frac{n-1}{r}\p\sb r
-\frac{\Delta_{\mathbb{S}^{n-1}}}{r^2}
+|\nu'|^2-\Delta\nu
-\theta|\nu'|^2
-\frac {1}{4\theta r^2}\left(G+\frac{n-2}{2}\right)^2
-\frac{(n-2)\nu'}{2r}
\nonumber
\\
&&
=
-\p\sb r^2
- \frac{n-1}{r}\p\sb r
-
\Big(1-\frac{1}{4\theta}\Big)
\frac{\Delta_{\mathbb{S}^{n-1}}}{r^2}
+(1-\theta)|\nu'|^2-\Delta\nu
-\frac{(n-2)^2}{16\theta r^2}
-\frac{(n-2)\nu'}{2r}.
\end{eqnarray}
The last equality is due to the identity
$
\Big(G+\frac{n-2}{2}\Big)^2
=\Big(S-\frac{1}{2}\Big)^2
=-\Delta_{\mathbb{S}^{n-1}}+\left(\frac{n-2}{2}\right)^2
$,
where $S$ is the spin-orbit coupling operator;
see~\cite[p. 849]{MR1741375}.
Taking into account that
the radial part of the Laplace operator
is positive-definite
$-\p\sb r^2-((n-1)/r)\p\sb r\ge 0$,
and returning the factor $\tau$ at $\nu$,
we conclude from \eqref{a-d-a-d-0} that
\begin{eqnarray}\label{a-d-a-d}
\bm\alpha\cdot(-\jj\nabla +\jj\tau \nabla\nu)
\,\bm\alpha\cdot(-\jj\nabla -\jj\tau \nabla\nu)
\ge
-
\Big(1-\frac{1}{4\theta}\Big)\Delta
+(1-\theta)|\tau\nu'|^2-\tau\Delta\nu
-\frac{(n-2)^2}{16\theta r^2}
-\frac{(n-2)\tau\nu'}{2r}.
\end{eqnarray}
Given $c\in(0,\sqrt{3}/2)$,
let $\theta\in(0,1)$ be such that $1-\frac{1}{4\theta}=c^2\in(0,3/4)$.
The right-hand side of \eqref{a-d-a-d} is greater than or equal to
\[
\Big(1-\frac{1}{4\theta}\Big)(-\Delta)
=c^2(-\Delta)
\]
(thus yielding \eqref{t-e})
once we make sure that
the following expression is positive everywhere on $\supp u$:
\begin{equation}\label{i-p}
(1-\theta)
\tau^2|\nu'|^2-\tau\Delta\nu
-\frac{1}{4\theta r^2}
\left(\frac{n-2}{2}\right)^2
-\frac{|(n-2)\tau\nu'|}{2r}.
\end{equation}
Since $\theta<1$,
taking into account that
$\inf\sb{x\in\supp u}\p\sb r\nu(x)>0$
(since $\supp u$ is compact)
and that $r\ge \dist(0,\supp u)>0$,
we conclude that
there is $\tau_{c,\nu,u}<\infty$
such that \eqref{i-p} is positive for all
$\tau\ge\tau_{c,\nu,u}$ and all $x\in\supp u$.
\end{proof}

The higher-dimensional version of the unique continuation principle in
\cite[Appendix]{MR880983} follows from the H\"older inequality
and the Sobolev embedding,
\begin{equation}\label{b-s}
\|V \psi\|_{L^2(\Omega,\C^N)}\leq
A_n\|V\|_{L^n(\Omega,\End(\C^N))}\|\psi\|_{\dot H^1(\Omega,\C^N)},
\qquad n\ge 3,
\end{equation}
with
$A_n>0$ the best Sobolev constant,
valid for all
open sets $\Omega\subset\R^n$,
$V\in L^n_{\mathrm{loc}}(\R^n,\End(\C^N))$,
and $\psi\in H^1_{\mathrm{comp}}(\Omega,\C^N)$.
Here it is:

\begin{theorem}\label{theorem-ln}
Let $n\ge 1$.
Let $\Omega$ be an open connected set of $\R^n$.
Let $V:\,\Omega\to\End(\C^N)$
be measurable,
such that
$V\in L^n_{\mathrm{loc}}(\Omega,\End(\mathbb{C}^N))$ if $n\neq 2$ or $V\in L^q_{\mathrm{loc}}(\Omega,\End(\mathbb{C}^N))$ with $q>2$ if $n=2$.
Assume that
$\psi\in H\sp 1\sb {\mathrm{loc}}(\Omega,\C^N)$
is such that
$\psi=0$ almost everywhere in a non-empty open
set $\Omega_0\subset\Omega$
and
\[
|(D\sb 0\psi)(x)|\leq |(V\psi)(x)|
\qquad
\mbox{for $x$ almost everywhere in $\Omega$}.
\]
Then $\psi=0$ almost everywhere in $\Omega$.
\end{theorem}

\begin{remark}
Similar (slightly weaker) results have also been obtained
in \cite{MR865834,MR880983,MR1280386};
see also the survey \cite{MR934297}.
The counterexamples in~\cite{MR1880829}
to the unique continuation principle
constructed for the case of the Laplace operator
suggest that the above result is optimal.
\end{remark}

\begin{proof}
In the one-dimensional case, the statement is
a consequence of the uniqueness of ODE solutions.
More precisely, the relation
$
|\psi'(x)|<|V\psi(x)|
$,
with $x$ in an open interval $\Omega\subset\R$,
together with the assumption $\psi(a)=0$ at some $a\in\Omega$,
leads to
\[
|\psi(x)|
\le
|W(x)|\sup_{y\in[a,x]}|\psi(y)|,
\qquad
x\in\Omega,
\]
where
$W(x):=\int_a^x|V(y)|\,dy$
is continuous in $\Omega$
since $V\in L^1_{\mathrm{loc}}(\Omega)$.
One deduces that $\psi(x)\equiv 0$
in the closure
of an open neighborhood of $a$ where $|W|<1$.
Then, by induction, $\psi$ vanishes identically in $\Omega$.

Let us now assume that $n\geq 3$.
It is enough to assume that $\Omega$ is bounded and so small that
\begin{equation}\label{v-ln-small}
A_n\norm{V}\sb{L^n(\Omega,\End(\C^N))}<\frac{1}{2},
\end{equation}
with $A_n$ the best Sobolev constant from \eqref{b-s}.
Fix $\nu\in C^\infty(\R^n)$ as above (spherically symmetric,
strictly monotonically increasing with $\abs{x}$, $x\in\R^n$).
We claim that $V$ satisfies the assumption~\eqref{assumpt}
of Lemma~\ref{lemma-unique-continuation},
with some $\kappa\in(0,1)$.
Indeed, let $u\in H^1_{\mathrm{comp}}(\Omega,\C^N)$
and
$a\in\Omega\setminus\supp u$.
By Lemma~\ref{lemma-ede},
there is $\tau_0=\tau_{\frac 1 2,\nu,u(\cdot-a)}<\infty$
such that for any $\tau\geq \tau_0$
one has
\begin{equation}\label{e-t-u}
\norm{e^{-\tau\nu_a}D_0 u}\ge \frac 1 2\norm{D_0(e^{-\tau\nu_a}u)}.
\end{equation}
Using the Sobolev embedding \eqref{b-s},
we have:
\begin{eqnarray}\label{comb}
&
\|e^{-\tau \nu_a }V u\|_{L^2(\Omega,\C^N)}
\leq
A_n\|V\|_{L^n(\Omega,\End(\C^N))}
\|e^{-\tau \nu_a}u\|_{\dot{H}^1(\Omega,\C^N)}
\\[1ex]
&
=
A_n\|V\|_{L^n(\Omega,\End(\C^N))}
\|D_ 0 e^{-\tau \nu_a}u\|_{L^2(\Omega,\C^N)}
\leq
2 A_n\|V\|_{L^n(\Omega,\End(\C^N))}
\|e^{-\tau \nu_a}D_0 u\|_{L^2(\Omega,\C^N)}.
\nonumber
\end{eqnarray}
Above, the equality is due to
$\norm{u}\sb{\dot H^1}^2
=\langle \nabla u,\nabla u\rangle
=\langle u,(-\Delta)u\rangle
=\langle D_0 u,D_0 u\rangle$,
while the last inequality is due to \eqref{e-t-u}.
One concludes from \eqref{comb}
that $V$ satisfies \eqref{assumpt}
with
\[
\kappa=2A_n\norm{V}\sb{L^n(\Omega,\End(\C^N))},
\]
with $\kappa<1$
as long as $\Omega$ is small enough
so that \eqref{v-ln-small} is satisfied.

Let $\psi\in H^1_{\mathrm{loc}}(\Omega,\C^N)$
and assume that
\begin{equation}\label{assume-omega-0}
\Omega_0:=\Omega\setminus\supp\psi\ne\varnothing,
\qquad
\psi\at{\Omega}\not\equiv 0.
\end{equation}
Then, by Lemma~\ref{lemma-unique-continuation},
there exists an open set
$\Omega_1\subset \Omega$ such that
\begin{equation}\label{omega-omega}
\Omega_0
\subsetneqq \Omega_1,
\end{equation}
with $V\psi=0$ and $D_0\psi=0$
almost everywhere in $\Omega_1$.
We claim that, as the matter of fact,
this would lead to $\psi\at{\Omega_1}\equiv 0$.
Since $V\equiv 1$ is in $L^n\sb{\mathrm{loc}}(\Omega_1,\End(\C^N))$,
by the above argument, it satisfies the assumption
\eqref{assumpt} of Lemma~\ref{lemma-unique-continuation}
(with some $\kappa\in(0,1)$).
Moreover, since $D_0\psi=0$ almost everywhere in $\Omega_1$,
the assumption \eqref{assumpt-k} is also satisfied
(with $V=1$).
Due to \eqref{omega-omega},
$\Omega_0'=\Omega_1\setminus\supp\psi\supset\Omega_0$
is non-empty.
Assume that
$\psi\at{\Omega_1}\not\equiv 0$.
Then, by Lemma~\ref{lemma-unique-continuation},
there exists an open set
$\Omega'\subset \Omega_1$
such that
\begin{equation}\label{omega-omega-prime}
\Omega_0'
=\Omega_1\setminus\supp \psi\subsetneqq \Omega',
\end{equation}
with $V\psi=\psi=0$
(since now $V=1$)
and $D_0\psi\equiv 0$
almost everywhere in $\Omega'$.
Thus, we would have $\psi=0$ almost everywhere in $\Omega'$,
contradicting
\eqref{omega-omega-prime}.
We conclude that $\psi\at{\Omega_1}\equiv 0$,
but this, in turn, leads to a contradiction
with \eqref{omega-omega}.
Therefore,
\eqref{assume-omega-0}
results in a contradiction;
we conclude that if $\Omega_0=\Omega\setminus\supp\psi$
is non-empty, then $\psi$ vanishes almost everywhere on $\Omega$.

In dimension $n=2$, the proof is similar to the one above,
with the following adaptation of \eqref{comb}.
Given $V\in L^q_{\mathrm{loc}}(\Omega,\R^2)$ with $q>2$,
the H\"older inequality yields
\[
\|e^{-\tau \nu_a }V u\|_{L^2(\Omega,\C^N)}
\leq
\|V\|_{L^q(\Omega,\End(\C^N))}
\|e^{-\tau \nu_a}u\|_{L^{p^\ast}(\Omega,\C^N)},
\]
where the second factor is bounded
with the aid of the Sobolev inequality
as follows:
\begin{eqnarray}
\|e^{-\tau \nu_a}u\|_{L^{p^\ast}(\Omega,\C^N)}
\le
c
\|\nabla
e^{-\tau \nu_a}u\|_{L^p(\Omega,\C^N)}
\le
c
\mathop{\mathrm{vol}}(\Omega)^{\frac 1 p -\frac 1 2}
\|\nabla e^{-\tau \nu_a}u\|_{L^2(\Omega,\C^N)},
\nonumber
\end{eqnarray}
with $c=c(p)<\infty$ independent of $\Omega$;
above, $p^\ast>2$ and $p\in(1,2)$ satisfy
\[
\frac 1 q+\frac{1}{p^\ast}=\frac 1 2,
\qquad
\frac{1}{p^\ast}=\frac{1}{p}-\frac 1 n,
\qquad
n=2.
\]
The set $\Omega\subset\R^2$ is to be small enough so that
\[
\kappa
:=
2c\mathop{\mathrm{vol}}(\Omega)^{\frac 1 p -\frac 1 2}
\|V\|_{L^q(\Omega,\End(\C^N))}<1,
\]
and then,
using \eqref{e-t-u} as before, we arrive at
\[
\|e^{-\tau \nu_a }V u\|_{L^2(\Omega,\C^N)}
\le
\frac{\kappa}{2}
\|D_0 e^{-\tau \nu_a} u\|_{L^2(\Omega,\C^N)}
\le
\kappa
\|e^{-\tau \nu_a}D_0 u\|_{L^2(\Omega,\C^N)},
\]
with $\kappa<1$;
hence, $V$ satisfies the assumption (\ref{assumpt}).
The rest of the argument is unchanged.
\end{proof}

\bibliographystyle{sima-doi}
\bibliography{linear-a}

\providecommand{\etalchar}[1]{$^{#1}$}
\def\cprime{$'$} \def\polhk#1{\setbox0=\hbox{#1}{\ooalign{\hidewidth
  \lower1.5ex\hbox{`}\hidewidth\crcr\unhbox0}}} \def\cprime{$'$}
  \def\cprime{$'$} \def\cprime{$'$} \def\cprime{$'$}
\begin{thebibliography}{BdMBMP93}

\bibitem[Abe98]{MR1618672}
S.~Abenda, \href{http://www.numdam.org/item?id=AIHPA_1998__68_2_229_0}{{\em
  Solitary waves for {M}axwell-{D}irac and {C}oulomb-{D}irac models\/}}, Ann.
  Inst. H. Poincar\'e Phys. Th\'eor. {\bf 68} (1998), pp. 229--244.

\bibitem[ABG82]{MR685028}
W.~Amrein, A.-M. Berthier, and V.~Georgescu, {\em Estimations du type {H}ardy
  ou {C}arleman pour des op\'erateurs diff\'erentiels \`a coefficients
  op\'eratoriels\/}, C. R. Acad. Sci. Paris S\'er. I Math. {\bf 295} (1982),
  pp. 575--578.

\bibitem[Agm75]{MR0397194}
S.~Agmon,
  \href{http://archive.numdam.org/ARCHIVE/ASNSP/ASNSP_1975_4_2_2/ASNSP_1975_4_2_2_151_0/ASNSP_1975_4_2_2_151_0.pdf}{{\em
  Spectral properties of {S}chr\"odinger operators and scattering theory\/}},
  Ann. Scuola Norm. Sup. Pisa Cl. Sci. (4) {\bf 2} (1975), pp. 151--218.

\bibitem[AS83]{PhysRevLett.50.1230}
A.~Alvarez and M.~Soler,
  \href{http://link.aps.org/doi/10.1103/PhysRevLett.50.1230}{{\em Energetic
  stability criterion for a nonlinear spinorial model\/}}, Phys. Rev. Lett.
  {\bf 50} (1983), pp. 1230--1233.

\bibitem[AS86]{PhysRevD.34.644}
A.~Alvarez and M.~Soler,
  \href{http://link.aps.org/doi/10.1103/PhysRevD.34.644}{{\em Stability of the
  minimum solitary wave of a nonlinear spinorial model\/}}, Phys. Rev. D {\bf
  34} (1986), pp. 644--645.

\bibitem[BC12a]{MR2892774}
G.~Berkolaiko and A.~Comech,
  \href{http://dx.doi.org/10.1051/mmnp/20127202}{{\em On spectral stability of
  solitary waves of nonlinear {D}irac equation in 1{D}\/}}, Math. Model. Nat.
  Phenom. {\bf 7} (2012), pp. 13--31.

\bibitem[BC12b]{MR2924465}
N.~Boussaid and S.~Cuccagna,
  \href{http://dx.doi.org/10.1080/03605302.2012.665973}{{\em On stability of
  standing waves of nonlinear {D}irac equations\/}}, Comm. Partial Differential
  Equations {\bf 37} (2012), pp. 1001--1056.

\bibitem[BC14]{BournaveasCandy}
N.~{Bournaveas} and T.~{Candy}, \href{http://arxiv.org/abs/1407.0655}{{\em
  {Global well-posedness for the massless cubic {D}irac equation}\/}}, ArXiv
  e-prints  (2014), \eprint{1407.0655}.

\bibitem[BCS15]{MR3311594}
G.~Berkolaiko, A.~Comech, and A.~Sukhtayev,
  \href{http://dx.doi.org/10.1088/0951-7715/28/3/577}{{\em Vakhitov-{K}olokolov
  and energy vanishing conditions for linear instability of solitary waves in
  models of classical self-interacting spinor fields\/}}, Nonlinearity {\bf 28}
  (2015), pp. 577--592.

\bibitem[BdMBMP93]{MR1241704}
A.~Boutet~de Monvel-Berthier, D.~Manda, and R.~Purice,
  \href{http://www.numdam.org/item?id=AIHPA_1993__58_4_413_0}{{\em Limiting
  absorption principle for the {D}irac operator\/}}, Ann. Inst. H. Poincar\'e
  Phys. Th\'eor. {\bf 58} (1993), pp. 413--431.

\bibitem[BG87]{MR880983}
A.~Berthier and V.~Georgescu,
  \href{http://dx.doi.org/10.1016/0022-1236(87)90007-3}{{\em On the point
  spectrum of {D}irac operators\/}}, J. Funct. Anal. {\bf 71} (1987), pp.
  309--338.

\bibitem[BG10]{MR2718928}
N.~Boussaid and S.~Gol{\'e}nia,
  \href{http://dx.doi.org/10.1007/s00220-010-1099-3}{{\em Limiting absorption
  principle for some long range perturbations of {D}irac systems at threshold
  energies\/}}, Comm. Math. Phys. {\bf 299} (2010), pp. 677--708.

\bibitem[BH92]{MR1162140}
E.~Balslev and B.~Helffer,
  \href{http://dx.doi.org/10.1016/0196-8858(92)90009-L}{{\em Limiting
  absorption principle and resonances for the {D}irac operator\/}}, Adv. in
  Appl. Math. {\bf 13} (1992), pp. 186--215.

\bibitem[BH15]{MR3314499}
I.~Bejenaru and S.~Herr,
  \href{http://dx.doi.org/10.1007/s00220-014-2164-0}{{\em The cubic {D}irac
  equation: small initial data in {$H^1(\mathbf{R}^3)$}\/}}, Comm. Math. Phys.
  {\bf 335} (2015), pp. 43--82.

\bibitem[BH16]{MR3477346}
I.~Bejenaru and S.~Herr,
  \href{http://dx.doi.org/10.1007/s00220-015-2508-4}{{\em The {C}ubic {D}irac
  {E}quation: {S}mall {I}nitial {D}ata in
  {$H^{\frac{1}{2}}(\mathbb{R}^2)$}\/}}, Comm. Math. Phys. {\bf 343} (2016),
  pp. 515--562.

\bibitem[Bog79]{MR592382}
I.~L. Bogolubsky, \href{http://dx.doi.org/10.1016/0375-9601(79)90442-0}{{\em On
  spinor soliton stability\/}}, Phys. Lett. A {\bf 73} (1979), pp. 87--90.

\bibitem[Bou96]{MR1391520}
N.~Bournaveas, \href{http://dx.doi.org/10.1080/03605309608821204}{{\em Local
  existence for the {M}axwell-{D}irac equations in three space dimensions\/}},
  Comm. Partial Differential Equations {\bf 21} (1996), pp. 693--720.

\bibitem[Bou06]{MR2259214}
N.~Boussaid, \href{http://dx.doi.org/10.1007/s00220-006-0112-3}{{\em Stable
  directions for small nonlinear {D}irac standing waves\/}}, Comm. Math. Phys.
  {\bf 268} (2006), pp. 757--817.

\bibitem[Bou08]{MR2466169}
N.~Boussaid, \href{http://dx.doi.org/10.1137/070684641}{{\em On the asymptotic
  stability of small nonlinear {D}irac standing waves in a resonant case\/}},
  SIAM J. Math. Anal. {\bf 40} (2008), pp. 1621--1670.

\bibitem[BPZ98]{PhysRevLett.80.5117}
I.~V. Barashenkov, D.~E. Pelinovsky, and E.~V. Zemlyanaya,
  \href{http://link.aps.org/doi/10.1103/PhysRevLett.80.5117}{{\em Vibrations
  and oscillatory instabilities of gap solitons\/}}, Phys. Rev. Lett. {\bf 80}
  (1998), pp. 5117--5120.

\bibitem[BSV87]{PhysRevD.36.2422}
P.~Blanchard, J.~Stubbe, and L.~V\`azquez,
  \href{http://link.aps.org/doi/10.1103/PhysRevD.36.2422}{{\em {Stability of
  nonlinear spinor fields with application to the {G}ross-{N}eveu model}\/}},
  Phys. Rev. D {\bf 36} (1987), pp. 2422--2428.

\bibitem[Can11]{MR2829499}
T.~Candy, \href{http://projecteuclid.org/euclid.ade/1355703201}{{\em Global
  existence for an {$L^2$} critical nonlinear {D}irac equation in one
  dimension\/}}, Adv. Differential Equations {\bf 16} (2011), pp. 643--666.

\bibitem[CGG14]{MR3208458}
A.~Comech, M.~Guan, and S.~Gustafson,
  \href{http://dx.doi.org/10.1016/j.anihpc.2013.06.001}{{\em On linear
  instability of solitary waves for the nonlinear {D}irac equation\/}}, Ann.
  Inst. H. Poincar\'e Anal. Non Lin\'eaire {\bf 31} (2014), pp. 639--654.

\bibitem[{Chu}08]{chugunova-thesis}
M.~{Chugunova},
  \href{http://digitalcommons.mcmaster.ca/dissertations/AAINR57289}{{\em
  Spectral stability of nonlinear waves in dynamical systems ({D}octoral
  {T}hesis)\/}}, McMaster University, Hamilton, Ontario, Canada, 2008.

\bibitem[CKMS10]{PhysRevE.82.036604}
F.~Cooper, A.~Khare, B.~Mihaila, and A.~Saxena,
  \href{http://dx.doi.org/10.1103/PhysRevE.82.036604}{{\em Solitary waves in
  the nonlinear {D}irac equation with arbitrary nonlinearity\/}}, Phys. Rev. E
  {\bf 82} (2010), p. 036604.

\bibitem[CKS{\etalchar{+}}15]{2015arXiv151203973C}
J.~{Cuevas-Maraver}, P.~G. {Kevrekidis}, A.~{Saxena}, A.~{Comech}, and
  R.~{Lan}, \href{http://arxiv.org/abs/1512.03973}{{\em {Stability of solitary
  waves and vortices in a 2D nonlinear Dirac model}\/}}, ArXiv e-prints
  (2015), \eprint{1512.03973}.

\bibitem[CP06]{MR2217129}
M.~Chugunova and D.~Pelinovsky, \href{http://dx.doi.org/10.1137/050629781}{{\em
  Block-diagonalization of the symmetric first-order coupled-mode system\/}},
  SIAM J. Appl. Dyn. Syst. {\bf 5} (2006), pp. 66--83.

\bibitem[CPS16]{MR3462129}
A.~Contreras, D.~E. Pelinovsky, and Y.~Shimabukuro,
  \href{http://dx.doi.org/10.1080/03605302.2015.1123272}{{\em {$L^2$} orbital
  stability of {D}irac solitons in the massive {T}hirring model\/}}, Comm.
  Partial Differential Equations {\bf 41} (2016), pp. 227--255.

\bibitem[CPV05]{MR2094265}
S.~Cuccagna, D.~Pelinovsky, and V.~Vougalter,
  \href{http://dx.doi.org/10.1002/cpa.20050}{{\em Spectra of positive and
  negative energies in the linearized {NLS} problem\/}}, Comm. Pure Appl. Math.
  {\bf 58} (2005), pp. 1--29.

\bibitem[Cuc14]{MR3180733}
S.~Cuccagna, \href{http://dx.doi.org/10.1090/S0002-9947-2014-05770-X}{{\em On
  asymptotic stability of moving ground states of the nonlinear {S}chr\"odinger
  equation\/}}, Trans. Amer. Math. Soc. {\bf 366} (2014), pp. 2827--2888.

\bibitem[CV86]{MR847126}
T.~Cazenave and L.~V{\'a}zquez,
  \href{http://projecteuclid.org/getRecord?id=euclid.cmp/1104115255}{{\em
  Existence of localized solutions for a classical nonlinear {D}irac field\/}},
  Comm. Math. Phys. {\bf 105} (1986), pp. 35--47.

\bibitem[dSS96]{1996PhRvE..54.1969D}
C.~M. {de Sterke}, D.~G. {Salinas}, and J.~E. {Sipe},
  \href{http://dx.doi.org/10.1103/PhysRevE.54.1969}{{\em {Coupled-mode theory
  for light propagation through deep nonlinear gratings}\/}}, Phys. Rev. E {\bf
  54} (1996), pp. 1969--1989.

\bibitem[EGS96]{MR1386737}
M.~J. Esteban, V.~Georgiev, and {\'E}.~S{\'e}r{\'e},
  \href{http://dx.doi.org/10.1007/BF01254347}{{\em Stationary solutions of the
  {M}axwell-{D}irac and the {K}lein-{G}ordon-{D}irac equations\/}}, Calc. Var.
  Partial Differential Equations {\bf 4} (1996), pp. 265--281.

\bibitem[ES95]{MR1344729}
M.~J. Esteban and {\'E}.~S{\'e}r{\'e},
  \href{http://projecteuclid.org/getRecord?id=euclid.cmp/1104273565}{{\em
  Stationary states of the nonlinear {D}irac equation: a variational
  approach\/}}, Comm. Math. Phys. {\bf 171} (1995), pp. 323--350.

\bibitem[ES02]{MR1897689}
M.~J. Esteban and {\'E}.~S{\'e}r{\'e},
  \href{http://dx.doi.org/10.3934/dcds.2002.8.381}{{\em An overview on linear
  and nonlinear {D}irac equations\/}}, Discrete Contin. Dyn. Syst. {\bf 8}
  (2002), pp. 381--397, current developments in partial differential equations
  (Temuco, 1999).

\bibitem[EV97]{MR1434039}
M.~Escobedo and L.~Vega,
  \href{http://dx.doi.org/10.1137/S0036141095283017}{{\em A semilinear {D}irac
  equation in {$H^s({\bf R}^3)$} for {$s>1$}\/}}, SIAM J. Math. Anal. {\bf 28}
  (1997), pp. 338--362.

\bibitem[Fed96]{MR1401125}
B.~V. Fedosov, \href{http://dx.doi.org/10.1007/978-3-642-48944-0_3}{{\em Index
  theorems\/}}, \href{http://dx.doi.org/10.1007/978-3-642-48944-0_3}{in
  \href{http://dx.doi.org/10.1007/978-3-642-48944-0_3}{{\em Partial
  differential equations, {VIII}\/}}}, vol.~65 of {\em Encyclopaedia Math.
  Sci.\/}, pp. 155--251, Springer, Berlin, 1996.

\bibitem[FFK56]{PhysRev.103.1571}
R.~Finkelstein, C.~Fronsdal, and P.~Kaus,
  \href{http://link.aps.org/doi/10.1103/PhysRev.103.1571}{{\em Nonlinear spinor
  field\/}}, Phys. Rev. {\bf 103} (1956), pp. 1571--1579.

\bibitem[FLR51]{PhysRev.83.326}
R.~Finkelstein, R.~LeLevier, and M.~Ruderman,
  \href{http://link.aps.org/doi/10.1103/PhysRev.83.326}{{\em Nonlinear spinor
  fields\/}}, Phys. Rev. {\bf 83} (1951), pp. 326--332.

\bibitem[Geo79]{MR576454}
V.~Georgescu, {\em On the unique continuation property for {S}chr\"odinger
  {H}amiltonians\/}, Helv. Phys. Acta {\bf 52} (1979), pp. 655--670.

\bibitem[GN74]{1974PhRvD..10.3235G}
D.~J. {Gross} and A.~{Neveu},
  \href{http://dx.doi.org/10.1103/PhysRevD.10.3235}{{\em {Dynamical symmetry
  breaking in asymptotically free field theories}\/}}, Phys. Rev. D {\bf 10}
  (1974), pp. 3235--3253.

\bibitem[GO12]{MR2916078}
V.~Georgiev and M.~Ohta,
  \href{http://projecteuclid.org/getRecord?id=euclid.jmsj/1335444402}{{\em
  Nonlinear instability of linearly unstable standing waves for nonlinear
  {S}chr\"odinger equations\/}}, J. Math. Soc. Japan {\bf 64} (2012), pp.
  533--548.

\bibitem[Gro66]{MR0190520}
L.~Gross, \href{http://dx.doi.org/10.1002/cpa.3160190102}{{\em The {C}auchy
  problem for the coupled {M}axwell and {D}irac equations\/}}, Comm. Pure Appl.
  Math. {\bf 19} (1966), pp. 1--15.

\bibitem[GSS87]{MR901236}
M.~Grillakis, J.~Shatah, and W.~Strauss,
  \href{http://dx.doi.org/10.1016/0022-1236(87)90044-9}{{\em Stability theory
  of solitary waves in the presence of symmetry. {I}\/}}, J. Funct. Anal. {\bf
  74} (1987), pp. 160--197.

\bibitem[GW08]{MR2513792}
R.~H. Goodman and M.~I. Weinstein,
  \href{http://dx.doi.org/10.1016/j.physd.2008.04.009}{{\em Stability and
  instability of nonlinear defect states in the coupled mode
  equations---analytical and numerical study\/}}, Phys. D {\bf 237} (2008), pp.
  2731--2760.

\bibitem[GWH01]{MR1825812}
R.~H. Goodman, M.~I. Weinstein, and P.~J. Holmes,
  \href{http://dx.doi.org/10.1007/s00332-001-0002-y}{{\em Nonlinear propagation
  of light in one-dimensional periodic structures\/}}, J. Nonlinear Sci. {\bf
  11} (2001), pp. 123--168.

\bibitem[Hei57]{RevModPhys.29.269}
W.~Heisenberg, \href{http://link.aps.org/doi/10.1103/RevModPhys.29.269}{{\em
  Quantum theory of fields and elementary particles\/}}, Rev. Mod. Phys. {\bf
  29} (1957), pp. 269--278.

\bibitem[His00]{MR1785381}
P.~D. Hislop, {\em Exponential decay of two-body eigenfunctions: a review\/},
  in {\em Proceedings of the {S}ymposium on {M}athematical {P}hysics and
  {Q}uantum {F}ield {T}heory ({B}erkeley, {CA}, 1999)\/}, vol.~4 of {\em
  Electron. J. Differ. Equ. Conf.\/}, pp. 265--288 (electronic), Southwest
  Texas State Univ., San Marcos, TX, 2000.

\bibitem[Huh13]{MR3063953}
H.~Huh, \href{http://dx.doi.org/10.1007/s11005-013-0622-9}{{\em Global
  solutions to {G}ross-{N}eveu equation\/}}, Lett. Math. Phys. {\bf 103}
  (2013), pp. 927--931.

\bibitem[IM99]{MR1743451}
A.~Iftimovici and M.~M{\u{a}}ntoiu,
  \href{http://dx.doi.org/10.1023/A:1007625918845}{{\em Limiting absorption
  principle at critical values for the {D}irac operator\/}}, Lett. Math. Phys.
  {\bf 49} (1999), pp. 235--243.

\bibitem[Iva38]{jetp.8.260}
D.~D. Ivanenko, {\em Notes to the theory of interaction via particles\/}, Zh.
  \'Eksp. Teor. Fiz {\bf 8} (1938), pp. 260--266.

\bibitem[Jer86]{MR865834}
D.~Jerison, \href{http://dx.doi.org/10.1016/0001-8708(86)90096-4}{{\em Carleman
  inequalities for the {D}irac and {L}aplace operators and unique
  continuation\/}}, Adv. in Math. {\bf 62} (1986), pp. 118--134.

\bibitem[JK79]{MR544248}
A.~Jensen and T.~Kato,
  \href{http://projecteuclid.org/getRecord?id=euclid.dmj/1077313577}{{\em
  Spectral properties of {S}chr\"odinger operators and time-decay of the wave
  functions\/}}, Duke Math. J. {\bf 46} (1979), pp. 583--611.

\bibitem[Ken87]{MR934297}
C.~E. Kenig, {\em Carleman estimates, uniform {S}obolev inequalities for
  second-order differential operators, and unique continuation theorems\/}, in
  {\em Proceedings of the {I}nternational {C}ongress of {M}athematicians,
  {V}ol. 1, 2 ({B}erkeley, {C}alif., 1986)\/}, pp. 948--960, Amer. Math. Soc.,
  Providence, RI, 1987.

\bibitem[KS02]{MR1897705}
T.~Kapitula and B.~Sandstede,
  \href{http://dx.doi.org/10.1137/S0036141000372301}{{\em Edge bifurcations for
  near integrable systems via {E}vans function techniques\/}}, SIAM J. Math.
  Anal. {\bf 33} (2002), pp. 1117--1143.

\bibitem[KT02]{MR1880829}
H.~Koch and D.~Tataru, \href{http://dx.doi.org/10.1515/crll.2002.003}{{\em
  Sharp counterexamples in unique continuation for second order elliptic
  equations\/}}, J. Reine Angew. Math. {\bf 542} (2002), pp. 133--146.

\bibitem[KY99]{MR1741375}
H.~Kalf and O.~Yamada, \href{http://dx.doi.org/10.2977/prims/1195143358}{{\em
  Note on the paper: ``{S}trong unique continuation property for the {D}irac
  equation'' by {L}. {D}e {C}arli and {T}. \={O}kaji\/}}, Publ. Res. Inst.
  Math. Sci. {\bf 35} (1999), pp. 847--852.

\bibitem[LG75]{PhysRevD.12.3880}
S.~Y. Lee and A.~Gavrielides,
  \href{http://dx.doi.org/10.1103/PhysRevD.12.3880}{{\em Quantization of the
  localized solutions in two-dimensional field theories of massive
  fermions\/}}, Phys. Rev. D {\bf 12} (1975), pp. 3880--3886.

\bibitem[Man94]{MR1280386}
N.~Mandache, \href{http://dx.doi.org/10.1007/BF00750142}{{\em Some remarks
  concerning unique continuation for the {D}irac operator\/}}, Lett. Math.
  Phys. {\bf 31} (1994), pp. 85--92.

\bibitem[Mer88]{MR949625}
F.~Merle, \href{http://dx.doi.org/10.1016/0022-0396(88)90018-6}{{\em Existence
  of stationary states for nonlinear {D}irac equations\/}}, J. Differential
  Equations {\bf 74} (1988), pp. 50--68.

\bibitem[MM86]{MR832537}
P.~Mathieu and T.~F. Morris, \href{http://dx.doi.org/10.1139/p86-042}{{\em
  Charged spinor solitons\/}}, Canad. J. Phys. {\bf 64} (1986), pp. 232--238.

\bibitem[MNNO05]{MR2108356}
S.~Machihara, M.~Nakamura, K.~Nakanishi, and T.~Ozawa,
  \href{http://dx.doi.org/10.1016/j.jfa.2004.07.005}{{\em Endpoint {S}trichartz
  estimates and global solutions for the nonlinear {D}irac equation\/}}, J.
  Funct. Anal. {\bf 219} (2005), pp. 1--20.

\bibitem[MNT10]{MR2666663}
S.~Machihara, K.~Nakanishi, and K.~Tsugawa,
  \href{http://dx.doi.org/10.1215/0023608X-2009-018}{{\em Well-posedness for
  nonlinear {D}irac equations in one dimension\/}}, Kyoto J. Math. {\bf 50}
  (2010), pp. 403--451.

\bibitem[MQC{\etalchar{+}}12]{PhysRevE.86.046602}
F.~G. Mertens, N.~R. Quintero, F.~Cooper, A.~Khare, and A.~Saxena,
  \href{http://link.aps.org/doi/10.1103/PhysRevE.86.046602}{{\em Nonlinear
  {D}irac equation solitary waves in external fields\/}}, Phys. Rev. E {\bf 86}
  (2012), p. 046602.

\bibitem[Pel11]{MR2883845}
D.~Pelinovsky, \href{http://arxiv.org/abs/1011.5925}{{\em Survey on global
  existence in the nonlinear {D}irac equations in one spatial dimension\/}},
  \href{http://arxiv.org/abs/1011.5925}{in
  \href{http://arxiv.org/abs/1011.5925}{{\em Harmonic analysis and nonlinear
  partial differential equations\/}}}, RIMS K\^oky\^uroku Bessatsu, B26, pp.
  37--50, Res. Inst. Math. Sci. (RIMS), Kyoto, 2011, \eprint{1011.5925}.

\bibitem[PS12]{MR2985264}
D.~E. Pelinovsky and A.~Stefanov,
  \href{http://dx.doi.org/10.1063/1.4731477}{{\em Asymptotic stability of small
  gap solitons in nonlinear {D}irac equations\/}}, J. Math. Phys. {\bf 53}
  (2012), pp. 073705, 27.

\bibitem[PS14]{MR3147657}
D.~E. Pelinovsky and Y.~Shimabukuro,
  \href{http://dx.doi.org/10.1007/s11005-013-0650-5}{{\em Orbital stability of
  {D}irac solitons\/}}, Lett. Math. Phys. {\bf 104} (2014), pp. 21--41.

\bibitem[PSK04]{PhysRevE.70.036618}
D.~E. Pelinovsky, A.~A. Sukhorukov, and Y.~S. Kivshar,
  \href{http://link.aps.org/doi/10.1103/PhysRevE.70.036618}{{\em Bifurcations
  and stability of gap solitons in periodic potentials\/}}, Phys. Rev. E {\bf
  70} (2004), p. 036618, \eprint{nlin/0405019}.

\bibitem[RnRnSV74]{PhysRevD.10.517}
A.~F. Ra\~nada, M.~F. Ra\~nada, M.~Soler, and L.~V\'azquez,
  \href{http://dx.doi.org/10.1103/PhysRevD.10.517}{{\em Classical
  electrodynamics of a nonlinear {D}irac field with anomalous magnetic
  moment\/}}, Phys. Rev. D {\bf 10} (1974), pp. 517--525.

\bibitem[RS78]{MR0493421}
M.~Reed and B.~Simon, {\em Methods of modern mathematical physics. {I}{V}.
  {A}nalysis of operators\/}, Academic Press [Harcourt Brace Jovanovich
  Publishers], New York, 1978.

\bibitem[Sol70]{PhysRevD.1.2766}
M.~Soler, \href{http://dx.doi.org/10.1103/PhysRevD.1.2766}{{\em Classical,
  stable, nonlinear spinor field with positive rest energy\/}}, Phys. Rev. D
  {\bf 1} (1970), pp. 2766--2769.

\bibitem[SS94]{MartijnDeSterke1994203}
C.~M.~D. Sterke and J.~Sipe,
  \href{http://www.sciencedirect.com/science/article/pii/S0079663808705158}{{\em
  Gap solitons\/}}, Progress in Optics {\bf 33} (1994), pp. 203--260.

\bibitem[ST10]{MR2588476}
S.~Selberg and A.~Tesfahun,
  \href{http://projecteuclid.org/euclid.die/1356019318}{{\em Low regularity
  well-posedness for some nonlinear {D}irac equations in one space
  dimension\/}}, Differential Integral Equations {\bf 23} (2010), pp. 265--278.

\bibitem[Str89]{MR1032250}
W.~A. Strauss, {\em Nonlinear wave equations\/}, vol.~73 of {\em CBMS Regional
  Conference Series in Mathematics\/}, Published for the Conference Board of
  the Mathematical Sciences, Washington, DC, 1989.

\bibitem[SV86]{MR848095}
W.~A. Strauss and L.~V{\'a}zquez,
  \href{http://dx.doi.org/10.1103/PhysRevD.34.641}{{\em Stability under
  dilations of nonlinear spinor fields\/}}, Phys. Rev. D (3) {\bf 34} (1986),
  pp. 641--643.

\bibitem[Thi58]{MR0091788}
W.~E. Thirring, \href{http://dx.doi.org/10.1016/0003-4916(58)90015-0}{{\em A
  soluble relativistic field theory\/}}, Ann. Physics {\bf 3} (1958), pp.
  91--112.

\bibitem[VK73]{VaKo}
N.~G. Vakhitov and A.~A. Kolokolov,
  \href{http://dx.doi.org/10.1007/BF01031343}{{\em Stationary solutions of the
  wave equation in the medium with nonlinearity saturation\/}}, Radiophys.
  Quantum Electron. {\bf 16} (1973), pp. 783--789.

\bibitem[Wak66]{wakano-1966}
M.~Wakano, \href{http://ptp.ipap.jp/link?PTP/35/1117/}{{\em Intensely localized
  solutions of the classical {D}irac-{M}axwell field equations\/}}, Progr.
  Theoret. Phys. {\bf 35} (1966), pp. 1117--1141.

\bibitem[XST13]{MR3066202}
J.~Xu, S.~Shao, and H.~Tang,
  \href{http://dx.doi.org/10.1016/j.jcp.2013.03.031}{{\em Numerical methods for
  nonlinear {D}irac equation\/}}, J. Comput. Phys. {\bf 245} (2013), pp.
  131--149.

\bibitem[Yam73]{MR0320547}
O.~Yamada, \href{http://dx.doi.org/10.2977/prims/1195192961}{{\em On the
  principle of limiting absorption for the {D}irac operator\/}}, Publ. Res.
  Inst. Math. Sci. {\bf 8} (1972/73), pp. 557--577.

\end{thebibliography}
\end{document}